\documentclass[preprint,12pt]{imsart} 

\RequirePackage{amsthm,amsmath,amsfonts,amssymb}
\RequirePackage[numbers, sort&compress]{natbib}
\RequirePackage[colorlinks,citecolor=blue,urlcolor=black]{hyperref}
\RequirePackage{graphicx}

\usepackage[margin=1.1in]{geometry}

\RequirePackage[ruled,vlined]{algorithm2e}
\usepackage{enumerate}

\startlocaldefs
\numberwithin{equation}{section}
\theoremstyle{plain}
\newtheorem{theorem}{Theorem}[section]
\newtheorem{definition}{Definition}[section]

\newtheorem{corollary}{Corollary}[section]
\newtheorem{lemma}{Lemma}[section]

\newtheorem{mainres}{Main Result}

\endlocaldefs

\let \hat \widehat
\newcommand{\scbf}[1]{\textbf{\textsc{#1}}}
\newcommand{\madgq}{MAD$_{\rm GQ}$}
\newcommand{\madgqbf}{{\bf MAD}$_{\bf GQ}$}
\newcommand{\madgqlasso}{MAD$_{\rm GQlasso}$}
\newcommand{\madgqlassobf}{{\bf MAD}$_{\bf GQlasso}$}
\newcommand{\vertiii}[1]{{\vert\kern-0.25ex\vert\kern-0.25ex\vert #1 
    \vert\kern-0.25ex\vert\kern-0.25ex\vert}}

\newcommand{\less}{\vspace{-2mm}}

\newcommand{\more}{\vspace{2mm}}

\begin{document}

\begin{frontmatter}

\title{\normalsize { GRAPH QUILTING: GRAPHICAL MODEL SELECTION\\ FROM PARTIALLY OBSERVED COVARIANCES}}\runtitle{Graph Quilting}

\begin{aug}
\author[A]{\fnms{Giuseppe}~\snm{Vinci}\ead[label=e1]{gvinci@nd.edu}},
\author[B]{\fnms{Gautam}~\snm{Dasarathy}\ead[label=e2]{gautamd@asu.edu}}
\and
\author[C]{\fnms{Genevera I.}~\snm{Allen}\ead[label=e3]{gallen@rice.edu}}

\runauthor{Vinci, Dasarathy, and Allen}

\address[A]{\scriptsize \sl Department of Applied and Computational Mathematics and Statistics, University of Notre Dame\printead[presep={,\ }]{e1}}

\address[B]{\scriptsize \sl Department of Electrical, Computer and Energy Engineering, Arizona State University\printead[presep={,\ }]{e2}}

\address[C]{\scriptsize \sl Department of Electrical and Computer Engineering, Rice University\printead[presep={,\ }]{e3}}
\end{aug}

\begin{abstract}
Graphical model selection is a seemingly impossible task when many pairs of variables are never jointly observed; this requires inference of conditional dependencies with no observations of corresponding marginal dependencies. This under-explored statistical problem arises in neuroimaging, for example, when different partially overlapping subsets of neurons are recorded in non-simultaneous sessions.  We call this statistical challenge the ``Graph Quilting'' problem.  We study this problem in the context of sparse inverse covariance learning, and focus on Gaussian graphical models where we show that missing parts of the covariance matrix yields an unidentifiable precision matrix specifying the graph.  Nonetheless, we show that, under mild conditions, it is possible to correctly identify edges connecting the observed pairs of nodes.  Additionally, we show that we can recover a minimal superset of edges connecting variables that are never jointly observed.  Thus, one can infer conditional relationships even when marginal relationships are unobserved, a surprising result!  To accomplish this, we propose an $\ell_1$-regularized partially observed likelihood-based graph estimator and provide performance guarantees in population and in high-dimensional finite-sample settings. We illustrate our approach using synthetic data, as well as for learning functional neural connectivity from calcium imaging data.
\end{abstract}



\end{frontmatter}

\section{Introduction}\label{sec:intro}\less
Probabilistic graphical models have been widely used as a computationally effective and statistically sound depiction of the dependence structure of large numbers of random variables \cite{lauritzen1996graphical}.  These have been applied in neuroscience \cite{vinci2018adjusted,vinci2018adjustedB,yatsenko2015improved}, 
genomics \cite{allen2013local, dobra2004sparse, 
gallopin2013hierarchical, kramer2009regularized, yin2011sparse}, finance \cite{carvalho2007dynamic}, physics \cite{pelizzola2005cluster}, and national security, among many others. In a graphical model, the dependence structure of $p$ variables $X_1,...,X_p$ is encoded by a graph $G=(V,E)$, where the vertices or nodes $V=\{1,...,p\}$ represent the variables, and edges in $E$ connect nodes to reflect conditional dependence relationships.  Given data with $n$ samples, $X^{(1)},...,X^{(n)}$, we are interested in learning the graphical model structure, sometimes called structural learning or graph selection \cite{drton2017structure}, in possibly high-dimensional settings where $p>n$. There has been an abundance of work on this problem \cite{banerjee2015bayesian,ravikumar2011high,rothman2008sparse,yuan2010high,yuan2007model}, but we study a new version of the graph selection problem that arises when many pairs of variables are never jointly observed.

\begin{figure}[t!]
\center
\includegraphics[width=1\textwidth]{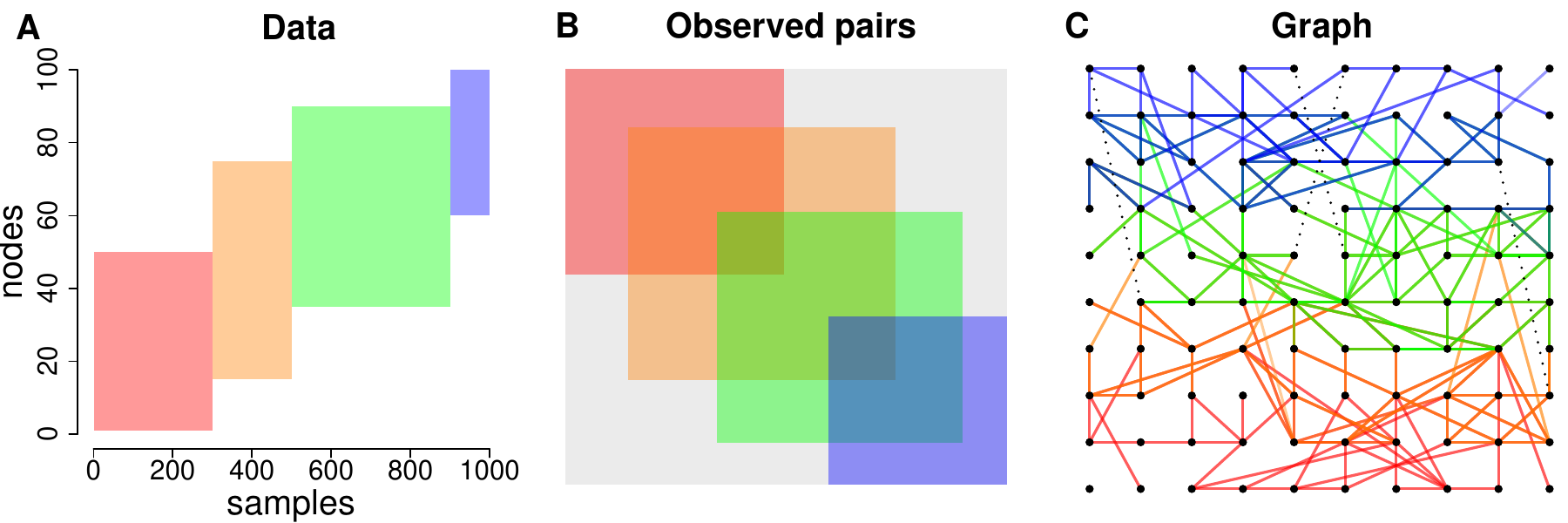}
\caption{
{\bf (A)} Four different subsets of nodes are observed across a total of $1000$ samples. {\bf (B)} Observed variable pairs. Gray entries are never jointly observed. {\bf (C)} Graph with edges colored in correspondence to the observed pairs of nodes; dotted edges connect pairs of nodes that are never observed jointly.  The objective of Graph Quilting is to recover not only the edges connecting jointly observed nodes (colored edges) but also those connecting nodes never jointly observed (dotted edges). 
}
\label{fig:intro}
\end{figure}

\subsection{The Graph Quilting problem}\label{sec:introGQproblem}\less Suppose now that the vectors $X^{(1)},...,X^{(n)}$ are not fully observed in such a manner that several pairs of the $p$ variables are never observed jointly across the all $n$ samples. For instance, Figure~\ref{fig:intro} illustrates the case where four different subsets $V_1, V_2, V_3, V_4\subset V$ of $p=100$ nodes are observed across a total of $n=1000$ data points but, although $\cup_kV_k=V$, many pairs of nodes are never observed jointly. Indeed, we only observe pairs in the set $O=\cup_kV_k\times V_k\subset V\times V$. These circumstances generate a fundamental problem: we have no empirical evidence about the marginal dependence between any pair of variables in $O^c$. This situation brings us to pose several questions: Can we learn the structure of a graphical model for the pairs of observed variables in $O$ when the covariance is not fully observed?  Even more challenging, can we learn the graphical model structure for pairs of variables that are never jointly observed in $O^c$?  In other words, is it possible to infer conditional dependence with no knowledge of marginal dependence? We call this challenging situation the ``Graph Quilting'' problem. This name is evocative of the implicit task of recovering or estimating the graph by ``quilting'' together multiple graphical structures relative to the observed subsets of nodes. In this paper, we illustrate why the Graph Quilting problem is so challenging and also prove surprisingly strong results about graph recovery.

\subsection{Graph Quilting in applied contexts}\label{sec:appliedcontexts}\less
The Graph Quilting problem is not only a theoretical conundrum. It arises in several applied contexts where variables are not simultaneously measured or there is extreme or structural missingness.   Such areas include biomedicine, communications, chemistry, material science, medical records, national security, and finance.  But to further illustrate how the Graph Quilting problem arises in practice, we highlight two examples from biomedicine: neuroscience and genomics.  

\paragraph*{Neuroscience} 
Functional connectivity is the statistical dependence of neurons' activities, and can be represented by a graph estimated from data of {\it in vivo} simultaneously recorded neurons. Studying functional connectivity helps us understand how neurons interact with one another while they process information under different stimuli and other experimental conditions \cite{vinci2016separating, vinci2018adjusted, vinci2018adjustedB, yatsenko2015improved}. New ambitious neuroscience projects involve recording the activities of tens-to-hundreds of thousands of neurons in three-dimensional portions of brain via calcium imaging technology \cite{bae2021functional}. A fundamental trade-off between temporal and spatial resolution characterizes this technology: the more neurons we aim to record from simultaneously, the coarser the time resolution. Since important neuronal activity patterns happen on very short time scales, it is often preferred to record the activities of a subset of neurons at once with a fine temporal resolution rather than recording the activities of the entire neuronal population simultaneously with a coarse time resolution. Yet, if these subsets are recorded {\it nonsimultaneously}, only a subset $O$ of all neuronal pairs may have joint observations, while the rest $O^c$ remain unobserved.  But, scientists are interested in learning the functional connectivity patterns of all neurons, not just those in $O$; this motivates our Graph Quilting problem.

\paragraph*{Genomics}
There has been increasing interest in using {\it single cell RNA-sequencing} (scRNAseq) to measure gene expression levels of each individual cell, allowing scientists to study the genomic changes that lead to cell-specific functions or dysfunctions. However, the data quality of scRNAseq is much poorer than that of bulk RNAseq due to {\it dropouts} or {\it dropdowns}, a technical artifact where genes appear to have zero expression because the scRNAseq technology can only capture a small fraction of the transcriptome of each cell \cite{chen2018scrmd,gong2018drimpute,huang2018saver,jeong2020prime,kolodziejczyk2015technology,tracy2019rescue,zhu2018unified}.  Thus, many genes, and especially those with lower expression levels, are missed by this technology.  Possible missingness is often so extreme in scRNAseq, that many genes pairs have no joint non-zero measurements \cite{gan2022correlation}.  Hence, estimating a gene co-expression network from these data leads to our Graph Quilting problem.

\subsection{Related literature}\less
We identify three existing lines of research that are related to our Graph Quilting problem. First, several approaches have been proposed to deal with covariance and graph estimation in situations affected by missing data \cite{dasarathy2019gaussian,dasarathy2016active, kolar2012estimating, lauritzen1995algorithm, loh2011high, stadler2012missing}. These methods assume that all variable pairs have been simultaneously measured on at least a subset of the observations or that the missingness is random in such a way that all pairs of variables are observed jointly at least once with high-probability, i.e. $O=V\times V$. Unfortunately, our Graph Quilting problem is characterized by an amount of non-random structural missingness such that $O \subset V \times V$.  Thus, the approaches proposed in this line of research are not applicable to the Graph Quilting problem. 

Since the Graph Quilting problem is characterized by structural graph learning from an incomplete covariance matrix, one may suggest that covariance completion methods offer a possible solution.  In fact, there has been much work in this area focusing on positive definite matrix completions \cite{bakonyi1995maximum, dempster1972covariance, grone1984positive, laurent2009matrix, vandenberghe1998determinant}, positive semi-definite or low-rank matrix completions \cite{bhargava2017active, bishop2014deterministic, candes2010matrix, candes2009exact, pfau2013robust, xu2016new}, and some specific statistical models for covariance completion of neural data \cite{turaga2013inferring,wohrer2010linear}.  Yet, this literature on covariance completion has focused on accuracy of the covariance estimate and not the accuracy of the precision matrix or recovery of the graphical model structure, the focus of this paper.  Nevertheless, covariance completion could yield a possible solution to the Graph Quilting problem; we discuss this possibility in Section 2, but choose to pursue a more direct approach to Graph Quilting in this paper.

Finally, many may note that our Graph Quilting problem is closely related to the latent variable graphical model (LVGM) problem introduced in \cite{chandrasekaran2012}, which seeks to learn the graph structure of the observed set of nodes in the presence of latent or hidden nodes.  There has been much interest and work on this important problem \cite{chandrasekaran2012, meng2014learning, wang2020learning, xu2017speeding, yatsenko2015improved}.  One can view our Graph Quilting problem as a composite latent variable graphical model problem where the variables unobserved from each set $V_k$ are treated as latent variables.  Yet, our Graph Quilting problem differs from and is more challenging than that of the LVGM problem in key ways.  First, LVGM is interested only in graph recovery amongst the observed nodes.  In our Graph Quilting problem, we seek to recover the graph amongst the observed pairs of nodes, $O$, but also amongst the unobserved pairs of nodes, $O^c$, a much harder problem.  Secondly, our Graph Quilting problem permits a fully general and arbitrary set of jointly observed variables, $O$, compared to the LVGM problem which assumes $O$ is a Cartesian product, $O = V_k \times V_k$. Importantly, our approach allows, but does not require, overlapping sets of observations. Finally, in this paper, we focus on Graph Quilting where we assume that we have observed each of the variables at least once, $\cup_k V_k = V$, which differs from the LVGM problem where many variables are hidden and never observed.  But as discussed briefly in Section 3, we show that our Graph Quilting approach and theory extends to the case where $\cup_k V_k \subset V$ and hence the LVGM problem as well.

\subsection{Main contributions}\less
In this paper, we focus on the Graph Quilting problem for structural recovery in the Gaussian graphical model \cite{lauritzen1996graphical, ravikumar2011high, yuan2007model}, although our methods and theory are suitable for general sparse inverse covariance learning.  Here, $X = (X_1,...,X_p)^T \sim N(\mu,\Sigma)$, 
with mean vector $\mu$, and $p\times p$ positive definite precision matrix $\Theta=\Sigma^{-1}$ with the property $\Theta_{ij}= 0$ $\Leftrightarrow$ $(i,j)\notin E$ $ \Leftrightarrow$ $X_i\perp X_j\mid \{X_k\}_{k\in V\setminus \{i,j\}}$, denoting conditional independence. 
In this framework, we define the Graph Quilting problem as the problem of estimation of $\Theta$ and $G$ from an incomplete set of empirical covariances $\hat\Sigma_O=(\hat\Sigma_{ij})_{(i,j)\in O}$.  We briefly summarize our main contributions for this problem. 

\paragraph*{The challenges of Graph Quilting} The fundamental question is whether it is even possible to recover the graph $G$, and the precision matrix $\Theta$, from  a subset of the true covariance matrix, $\Sigma_O$. This is an underdetermined system of observations, and it is therefore unsurprising that this is not possible in general. However, such issues are routinely handled in  modern high-dimensional statistics research by effectively constraining the search space to low-complexity models such as sparse vectors or low-rank matrices. We show that the situation is significantly more challenging for the Graph Quilting problem:

\begin{mainres}[\scbf{Graph Identifiability}]
$G$ is identifiable from $\Sigma_O$ alone if and only if $E\subseteq O$, even if the cardinality of $E$ is known.
\end{mainres}

This result shows that the Graph Quilting problem is generally impossible, even if we know the true sparsity, or number of edges, of the graph!  Conditions for graph recovery suggest that to recover edges, we must jointly observe all of the pairs of variables connected via an edge, a completely unrealistic assumption that practically implies prior knowledge of the graph and hence negates the need for graph structural learning. This result also shows that standard approaches to solving underdetermined problems, such as assuming sparsity, will also fail; clearly one needs to leverage more structure. In this paper, we propose a natural set of assumptions that will allow us to break this information-theoretic barrier, and take the first steps toward tackling this important problem. 

\paragraph*{Graph Quilting Model Selection}
Based on the above result and observations, we show that imposing the condition $\Theta_{O^c}=0$ (even when this is not true) lets us obtain an approximation that is sufficiently close to $\Theta$ so as to recover $G$ in a wide variety of situations. Moreover, we can do this in a computationally efficient manner via a convex program.  Specifically, to recover $\Theta$ and $G$ given an incomplete empirical covariance matrix $\hat\Sigma_O$, we propose the \madgqlassobf{} (MAximum Determinant$_{\bf GQlasso}$), an $\ell_1$ regularized estimator given by
\[\hat{\tilde\Theta} ~=~ \underset{\Theta\succ 0,\Theta_{O^c}=0}{\arg\max} ~\log\det\Theta-\sum_{(i,j)\in O}\Theta_{ij}\hat\Sigma_{ij}-\Vert\Lambda\odot\Theta\Vert_{1,{\rm off}},\]
where $\Lambda$ is a matrix of nonnegative penalties, $\Vert M\Vert_{1,{\rm off}}=\sum_{i\neq j}|M_{ij}|$, and the constraint $\Theta_{O^c}=0$ rules out the dependence of the likelihood function on the unobserved empirical covariances $\hat\Sigma_{O^c}$. We prove the following main results about the graph structural learning of this estimator:
\begin{mainres}[\scbf{Graph recovery in $O$}] Under appropriate conditions, $\exists \tau>0$ such that the graph estimate $\hat E = \big\{(i,j):i\neq j, |\hat{\tilde\Theta}_{ij}|>\tau\big\}$ satisfies $\hat{E}_O = E_O$ with high probability.
\end{mainres}
In other words, hard thresholding the \madgqlassobf{} estimator yields consistent graph selection in $O$.  The next natural question, and the much more challenging problem, is whether we can recover the graph structure in $O^c$. Recall that the graph in $O^c$ is not identifiabile, so the best we can hope to achieve is a minimal superset of edges in $O^c$ that cover all possible graph structures consistent with $\hat{\Sigma}_{O}$.  Toward this end, we devise a scheme using Schur complements and hard thresholding to detect all potential graph structures in $O^c$:
\begin{mainres}[\scbf{Graph recovery in $O^c$}] Under appropriate conditions, our approach (Algorithm~\ref{algo:fullrecoveryKn}) recovers a set $\hat{\mathcal{S}}$ that is guaranteed to be a superset of the edges in $O^c$, $\hat{\mathcal{S}}~\supseteq ~E_{O^c}$, with high probability. Under additional assumptions, we show that this is the minimal possible superset achievable.
\end{mainres}
This surprisingly strong result demonstrates that it is indeed possible to recover some conditional dependence and independence relationships for pairs of variables that are never jointly observed and for which we have no measurement of marginal dependence.

\paragraph*{Organization}
We define our Graph Quilting problem, study graph identifiability, discuss possible solutions, and introduce our Maximum Determinant Graph Quilting approach in Section~\ref{sec:gqprobl}.  In Section~\ref{sec:popK} we study the Graph Quilting problem and our approach in the population setting, and in Section~\ref{sec:estimation}, we additionally leverage results from high-dimensional graph structural recovery to prove graph selection consistency for our problem in finite samples with high probability.  We illustrate the properties of our graph estimator in simulations in Section~\ref{sec:simulations} and through the analysis of calcium imaging data in Section~\ref{sec:data}.

\section{Characterization of the Graph Quilting Problem}\label{sec:gqprobl}\less
We formally define the Graph Quilting problem for Gaussian Graphical models and sparse inverse covariance learning.  We discuss why this is challenging through an unidentifiability result, and introduce how we solve this problem by proposing an estimator that we will study in the remainder of the paper.

\subsection{Graph Quilting problem}\label{sec:gqproblemsimpleintro}\less
Let $X=(X_1,...,X_p)^T\sim N(\mu,\Sigma)$, where $\Theta = \Sigma^{-1}\succ 0 $ is a $p \times p$ positive definite precision matrix which encodes the conditional dependence graph $G$ with edge set $E=\{(i,j) :i\neq j, \Theta_{ij}\neq 0\}$. The parameter $\Theta$ and thereby the graph $G$ are typically estimated via penalized likelihood maximization based on the sample covariance matrix $\hat \Sigma$ computed from $n$ fully observed data vectors $X^{(1)},...,X^{(n)}\stackrel{\rm i.i.d.}{\sim}N(\mu,\Sigma)$.

As we previously motivated, there are many situations in which we observe incomplete data in such a way that a complete estimate of the sample covariance matrix is no longer available. For example, suppose we observe multiple datasets $\mathbf{X}_1,...,\mathbf{X}_K $, with a pattern similar to Figure~\ref{fig:intro}A, where $\mathbf{X}_k\in\mathbb{R}^{n_k\times |V_k|}$ contains $n_k>1$ samples of vectors of nodes $V_k\subset V$, and $\bigcup_k V_k = V$. The set of jointly observed pairs of nodes across the available samples is given by $O=\bigcup_{k=1}^K V_k\times V_k$, and if $O^c\neq\emptyset$, then we can only obtain an incomplete sample covariance matrix $\hat\Sigma_O = (\hat\Sigma_{ij})_{(i,j)\in O}$, where each entry $\hat\Sigma_{ij}$ is computed using all available joint observations $(X_i^{(r)},X_j^{(r)})$ across the available samples (Definition~\ref{def:obscov}, Section~\ref{sec:estimator}).

The situation described above brings us to pose several questions. Can we infer conditional dependence with no knowledge of marginal dependence? In particular, can we estimate $\Theta$ and/or $G$ from an incomplete empirical covariance matrix $\hat\Sigma_O$? Is it possible to recover both $\Theta_O$ and the more challenging case of $\Theta_{O^c}$?  We call this challenging situation the ``Graph Quilting problem'', a name that evokes the implicit task of recovering the graph by {\it quilting} together multiple graphical structures relative to the observed subsets of nodes.

\subsection{Non-Identifiability: the challenges of Graph Quilting}\label{sec:identifiability}\less

Recovering the full conditional dependence graph of all $p$ nodes given partially observed covariances is extremely challenging because it requires one to infer a multiplicity of conditional dependence statements even for unobserved node pairs. It is certainly possible to estimate a graph for any node subset $V_k\subset V$ for which all pairs have been observed, but such graph would represent the dependence structure of those nodes unconditionally on the others. Indeed, for any set $A\subset V$, the Schur complement gives $\Sigma_{AA}^{-1} = \Theta_{AA}-\Theta_{AA^c}\Theta_{A^cA^c}^{-1}\Theta_{A^cA} $, so in general $\Sigma_{AA}^{-1}\neq \Theta_{AA}$. Moreover, such approach would not yield any recovery of the graph in $O^c$. Alternatively, we could attempt to approximate $\hat\Sigma_{O^c}$ to obtain a full covariance matrix. But, how can we do so in a manner that would allow us to correctly recover the inverse covariance matrix and corresponding graphical structure? 

Similar challenges, where we seek to estimate parameters from an underdetermined set of measurements, are routinely handled in high-dimensional statistics through structural assumptions like sparsity or low-rankness. So, one may suggest to make similar assumptions for our Graph Quilting problem.  Unfortunately, we show in the following result that recovering the graph $G$ from $\Sigma_O$ is impossible even if we know the exact level of sparsity in $G$:
\begin{theorem}[\scbf{Graph Identifiability}]
\label{thm:identifiability}
$G$ is identifiable from $\Sigma_O$ alone if and only if $E\subseteq O$, even if the cardinality of $E$ is known. 
\end{theorem}
This result shows that even if we make the very strong assumption of knowing the true graph level of sparsity, we still cannot identify the graph from $\Sigma_O$ unless $E$ is entirely contained in $O$; this essentially assumes foreknowledge of $E$ and negates the need for graph selection. This result seems like we set out to study an impossible problem.  But as we will establish in Section 3, breaking the problem up to consider recovery in $O$ separately from $O^c$, we show that under certain assumptions, the graph in $O$ is identifiable and while the graph in $O^c$ is not identifiable, a minimal superset of the graph in $O^c$ is identifiable.

\subsection{Our proposed solution}\less
Given the challenges with Graph Quilting, it is clear that we need to impose some additional structure or assumptions to begin to tackle our problem.  One may think of several possible methodological directions in which to proceed; we outline three broad families of approaches here:
\begin{enumerate}[(a)]
\item {\it Observed likelihood methods}, which exploit the log-likelihood function of the observed data, $\ell(\Theta;\mathbf{X}_1,...,\mathbf{X}_K) =$ $\sum_{k=1}^Kn_k\big\{\log\det\Sigma_{V_kV_k}^{-1}-{\rm tr}\big(\Sigma_{V_kV_k}^{-1}\hat\Sigma_{V_kV_k}\big)\big\}$, 
where $\Sigma_{V_kV_k}^{-1}=\Theta_{V_kV_k}-\Theta_{V_kV_k^c}\Theta_{V_k^cV_k^c}^{-1}\Theta_{V_k^cV_k} $ (Schur complement).
\item {\it Two-step methods}, which first perform covariance matrix completion on $\Sigma_O$ or $\hat\Sigma_O$, and then retrieve the precision matrix and associated graph; 
\item {\it Observed covariance methods}, which reconstruct the precision matrix from $\Sigma_O$ or $\hat\Sigma_O$ directly by maximizing the partial log-likelihood function $\ell(\Theta,\hat\Sigma_O)=\log\det\Theta - \sum_{(i,j)\in O}\Theta_{ij}\hat\Sigma_{ij}$ with constraints on $\Theta_{O^c}$. 
\end{enumerate}
In this paper, we pursue the observed covariance approach, (c), but pause to discuss the other options and justify our choice.  Observed likelihood methods, (a), seem the most direct, but upon further inspection it is unclear how to make this computationally or statistically tractable.  The precision matrix is fragmented across $K$ pieces of the observed likelihood function via several linked Schur complements; as $K$ grows and for general observation patterns $O$ that we consider, this approach quickly becomes intractable. Approach (b) similarly raises concerns of tractability and identifiability since the completion of $\Sigma_O$ needs to be done in a manner that constrains the element of this matrix's inverse to be sparse so as to preform graph selection.  Recently, \cite{chang2022low} studied this approach and showed that while covariance completion can approximately estimate the graph in practice, there are many challenges to providing theoretical guarantees on graph identifiability and recovery.  While we do not choose to pursue these approaches in this paper, additional thorough investigation of approaches (a) and (b) are fruitful avenues for future research.  

In this paper, we tackle the Graph Quilting problem by studying observed covariance models, (c), as this approach is naturally computationally tractable, and in the sequel, we show that under appropriate conditions it has favorable statistical characteristics. The optimization problem at the heart of approach (c) may be recast as the following: 
\begin{equation}\label{eq:gqgenprob}
\tilde\Theta =  \underset{\Theta\succ 0,~\Theta_{O^c}\in \mathcal{C}}{\arg\max}~ \log\det\Theta - \sum_{(i,j)\in O}\Theta_{ij}\Sigma_{ij},
\end{equation} 
where the objective function does not depend on the unobserved covariances of the set $O^c$, and $\mathcal{C}$ is a set of admissible values of $\Theta_{O^c}$.  With no appropriate constraint $\Theta_{O^c}\in \mathcal{C}$, the optimization problem would have infinitely many solutions. Hence, our approach is to impose suitable constraints on $\Theta_{O^c}$ that will allow us to recover the graph structure under reasonable assumptions.  
To this end, we focus on a specific instance of Equation~(\ref{eq:gqgenprob}):
\begin{definition}[\madgqbf]\label{def:gq0}
The \madgq{} approximation of $\Theta=\Sigma^{-1}$ given $\Sigma_O$ is
\begin{equation}\label{eq:gq0}
~~~~~~~~\tilde\Theta ~:=~  \underset{\Theta\succ 0, ~\Theta_{O^c}=0}{\arg\max}~ \log\det \Theta - \sum_{(i,j)\in O}\Theta_{ij}\Sigma_{ij}.~~~~~~~~~~(\textsc{MAD}_{\rm GQ})
\end{equation}
\end{definition}
We call the solution in Equation~(\ref{eq:gq0}) ``\madgq{}'' because of its relationship with the {\it maximum determinant} positive definite covariance matrix completion:
\begin{lemma}\label{lemma:maxdetopt}
Equation~(\ref{eq:gq0}) is equivalent to the max-det problem
\begin{equation}\label{eq:maxdet}
\tilde\Theta^{-1}:=~\tilde\Sigma = \underset{S\succ 0, ~S_O=\Sigma_O}{\arg\max} ~\det S,
\end{equation}
which has a unique solution if $\Sigma_O$ is completable to a positive definite matrix. 
\end{lemma}
Equation~(\ref{eq:maxdet}) has been investigated as a covariance completion approach \cite{bakonyi1995maximum, dempster1972covariance, grone1984positive} corresponding to the maximum entropy distribution with covariance constraints over the set $O$. Yet, the reliability of the retrieved precision matrix given by $\tilde\Theta$ and the associated edge set $\tilde E$ is completely unexplored. If the assumption $E\subseteq O$ of Theorem~\ref{thm:identifiability} is correct, then the  reconstructed \madgq{} matrix $\tilde\Theta$ matches $\Theta$ exactly, and thereby the graph $G$ is perfectly recovered. If $E\not\subseteq O$, then, in general, $\tilde\Theta\neq \Theta$, so that the graphical structure of $\tilde\Theta$ will not match $G$. Indeed, erroneously assuming that some pairs of nodes are conditionally independent would force the rest of the recovered network to adjust in order to reflect the dependence pathways expressed by $\Sigma_O$. However, a striking property of \madgq{} is the following:
\begin{theorem}[\scbf{No false negatives in $O$}]\label{theo:nofnd}
Let $\tilde E=\{(i,j):i\neq j,\tilde\Theta_{ij}\neq 0\}$ be the edge set induced by the \madgq{} solution $\tilde\Theta$ in Equation~(\ref{eq:gq0}). Then the property $E_O\subseteq \tilde E_O$ holds almost everywhere. 
\end{theorem}
Theorem~\ref{theo:nofnd} establishes that the \madgq{} solution $\tilde\Theta$ induces {\em no false negative edges} in $O$, except for a negligible set of positive definite matrices. More precisely, if we let $\mathcal{M}_E$ be the set of   $p\times p$ positive definite matrices all supported on the graphical structure $E$, then the property $E_O \subseteq \tilde{E}_O$ is only violated on a set that has negligible measure with respect to the Lebesgue measure on $\mathcal{M}_E$. To see this intuitively, suppose that $E_O$ is nonempty  and let $\Delta_{ij}:=\tilde\Theta_{ij}-\Theta_{ij}$. Notice now that having a false negative  $(i,j)\in E_O$ would require $\Theta_{ij}=-\Delta_{ij}$. The set of matrices that exactly satisfy the latter equality constitutes a lower dimensional manifold which occupies zero volume in the set $\mathcal{M}_E$. However, in the next sections, we show that we can go much farther: under additional assumptions we can recover the graph in $O$ exactly. Let us define the smallest edge magnitude in $O$,
\begin{equation}\label{eq:nu}
\nu~:=~\min_{(i,j)\in E_O} |\Theta_{ij}|,
\end{equation}
and the maximum off-diagonal distortion produced by our \madgq{} solution in $O$,
\begin{equation}\label{eq:delta}
\delta~:=~\max_{(i,j)\in O, i\neq j}|\Theta_{ij}-\tilde\Theta_{ij}|.
\end{equation}
We assume that $O$ contains at least one edge, so that $\nu$ exists. Moreover, define
\begin{equation}\label{eq:Etau}
\tilde E^\tau ~:=~ \big\{(i,j):i\neq j, |\tilde\Theta_{ij}|>\tau\big\}
\end{equation}
that is the graph obtained by thresholding the \madgq{} matrix at level $\tau$. The next lemma identifies a sufficient and necessary condition for the recovery of the graph in $O$ via $\tilde E_O^\tau$:
\begin{lemma}[\scbf{Exact graph recovery in $O$}]\label{lemma:exactO}
We have $\tilde E_O^{\tau} = E_O$, and ${\rm sign}(\tilde\Theta_{ij})={\rm sign}(\Theta_{ij})$,  $\forall(i,j)\in E_O$, if and only if  $\delta<\nu/2$ and $\tau\in [\delta,\nu-\delta)$.
\end{lemma}
This lemma states that, as long as the maximum distortion $\delta$ is sufficiently small ($\delta<\nu/2$),  we can recover the set of edges in $O$ and their signs  exactly by simply thresholding the entries of $\tilde\Theta_O$ at any level $\tau\in[\delta,\nu-\delta)$. However, when does the condition $\delta<\nu/2$ hold? It certainly depends on $\Theta$ and $O$, since both the distortion $\delta$ and the minimal magnitude $\nu$ depend on $\Theta$ and $O$. Theorem~\ref{thm:identifiability} guarantees that $\Theta_{O^c}=0$ implies $\delta=0<\nu/2$, and it is reasonable to expect that diverging only slightly from this case should still yield $\delta<\nu/2$. But how far can $\Theta_{O^c}$ diverge from the null case $\Theta_{O^c}=0$? In the context of Graph Quilting, there are other natural questions that present themselves: Can we recover any information about the graph in $O^c$? How do we deal with the finite sample case where $\Sigma_O$ is replaced by an empirical estimate $\hat \Sigma_O$ that is not guaranteed to be completable to a positive definite matrix as required by Lemma~\ref{lemma:maxdetopt}? These questions are the focus of the remainder of this paper.

\section{Graph Recovery: Population Analysis}\label{sec:popK}\less
We begin by investigating the Graph Quilting problem at the population level. 
That is, we assume that we have perfect access to $\Sigma_O$, a portion of the true covariance matrix, where $O\subset V\times V$ is represented as $O=\cup_{k=1}^KV_k\times V_k$, with $V_1,...,V_K\subset V$, $\cup_k V_k=V$, and smallest possible $K$. Our aim is to reconstruct the graph $G$, or the sparsity pattern in $\Theta$. In Sections~\ref{sec:popO} and \ref{sec:popOc}, we investigate the graph recovery in $O$ and $O^c$ separately, and then condense the results into one algorithm in  Section~\ref{sec:popFull}.  Appendix~\ref{app:popcaseK2} contains additional results for the special case $K=2$.
\less

\subsection{Graph Recovery in \texorpdfstring{$O$}{Lg}}\label{sec:popO}\less
Theorem~\ref{theo:nofnd} in Section~\ref{sec:identifiability} guarantees that the \madgq{} solution $\tilde\Theta$ induces no false negative edges in $O$, except on a set of measure zero. Moreover, Lemma~\ref{lemma:exactO} states that if $\delta<\nu/2$ then we can recover the edge set and signs in $O$ exactly by simply thresholding the entries of $\tilde\Theta_O$ at level $\tau\in [\delta,\nu-\delta)$. We now show that if the edges in $O^c$ are sufficiently weak, then $\delta<\nu/2$, so exact graph recovery in $O$ is possible! Specifically, let $\gamma:=\Vert\Theta_{O^c}\Vert_\infty$ be the largest magnitude in $\Theta_{O^c}$. We show that, for a given precision matrix $\Theta$ and observation set $O$, there exists a threshold $\alpha(\Theta,O)>0$ such that $\gamma<\alpha(\Theta,O)$ implies $\delta<\nu/2$. Indeed, if $\gamma\approx 0$, we expect $\Vert\Theta-\tilde\Theta\Vert_\infty\approx 0$ because $\Theta_{O^c}=0$ implies $\tilde\Theta=\Theta$, by the Graph Identifiability Theorem~\ref{thm:identifiability}. For illustration, in Section~\ref{sec:caseK2} we further discuss the results in the more analytically tractable special case $K=2$, where $O=\cup_{k=1}^2V_k\times V_k$,   $V_1\neq V_2$, and $V_1\cup V_2=V$, and we derive explicit expressions of the threshold $\alpha(\Theta,O)$.

To let our main theorem work, we need to define an appropriate class of matrices:
\begin{definition}\label{def:TO}
Let $\mathcal{T}_O := \{\Theta\succ 0 ~:~ A\succ 0, A_O=\Theta_O, A_{O^c}=0\}$ be the set of all $p\times p$ positive definite matrices that would remain positive definite even if their entries in $O^c$ were replaced by zeros.
\end{definition}
The following theorem states our main result for the recovery of $E_O$:
\begin{theorem}[\scbf{Exact graph recovery in $O$}]\label{theo:popO} 
If $\Theta\in\mathcal{T}_O$, then there exists a threshold $\alpha(\Theta,O)>0$ depending on $\Theta$ and $O$ such that, if $\gamma:=\Vert\Theta_{O^c}\Vert_\infty<\alpha(\Theta,O)$, then $\tilde E_O^{\tau} = E_O$ (Equation~(\ref{eq:Etau})) for all $\tau\in[\delta,\nu-\delta)$, and ${\rm sign}(\tilde\Theta_{ij})={\rm sign}(\Theta_{ij})$,  for all $(i,j)\in E_O$.
\end{theorem}
In the proof of Theorem~\ref{theo:popO} in Appendix~\ref{app:proofmains}, we first demonstrate the existence of a continuous function $\bar\delta(\gamma)$ of $\gamma$ that upper-bounds $\delta$ and with value $\bar\delta(0)=0$. Then, since $0\le\delta\le\bar\delta(\gamma)$, we note that $\gamma\to 0^+$ implies $\delta\to 0^+$, guaranteeing  the existence of the positive threshold $\alpha(\Theta,O)$ for which $\gamma<\alpha(\Theta,O)$ implies $\delta<\nu/2$.  Finally, Lemma~\ref{lemma:exactO} is applied. 

To illustrate the theoretical results of this section, in Figure~\ref{fig:poprecovery} we present an example with $p=40$ nodes, $K=5$ node subsets (see Appendix~\ref{app:figdetails} for details), and  $\gamma$ small enough to ensure $\delta<\nu/2$. In Figure~\ref{fig:poprecovery}(A) we show the support of the precision matrix $\Theta$ (black dots) and the set $O$ of observed node pairs (colored regions). Several edges are present in $O^c$. In Figure~\ref{fig:poprecovery}(B) we display the support of the \madgq{} matrix $\tilde\Theta$, which contains several false positives in $O$ (green dots), several false negatives in $O^c$ (red dots), but no false negatives in $O$ as per  Theorem~\ref{theo:nofnd}. Finally, in Figure~\ref{fig:poprecovery}(C) we plot the \madgq{} edge set $\tilde E^{\tau}$ (Equation~(\ref{eq:Etau})) with $\tau=\nu/2\in[\delta,\nu-\delta)$. The set $\tilde E^{\tau}$ perfectly matches the true edge set over $O$ as per Theorem~\ref{theo:popO}, since all false positives had magnitudes smaller than the threshold $\nu/2$. Figures~\ref{fig:poprecovery}(E)-(G) are analogous to Figures~\ref{fig:poprecovery}(A)-(C), but the sets $V_1,...,V_5$ are random subsets of $V$. Figures~\ref{fig:poprecovery}(D) and (H) are about the graph recovery in $O^c$, which is discussed in Section~\ref{sec:popOc}.\less

\begin{figure}[t!]
\center
\includegraphics[width=1\textwidth]{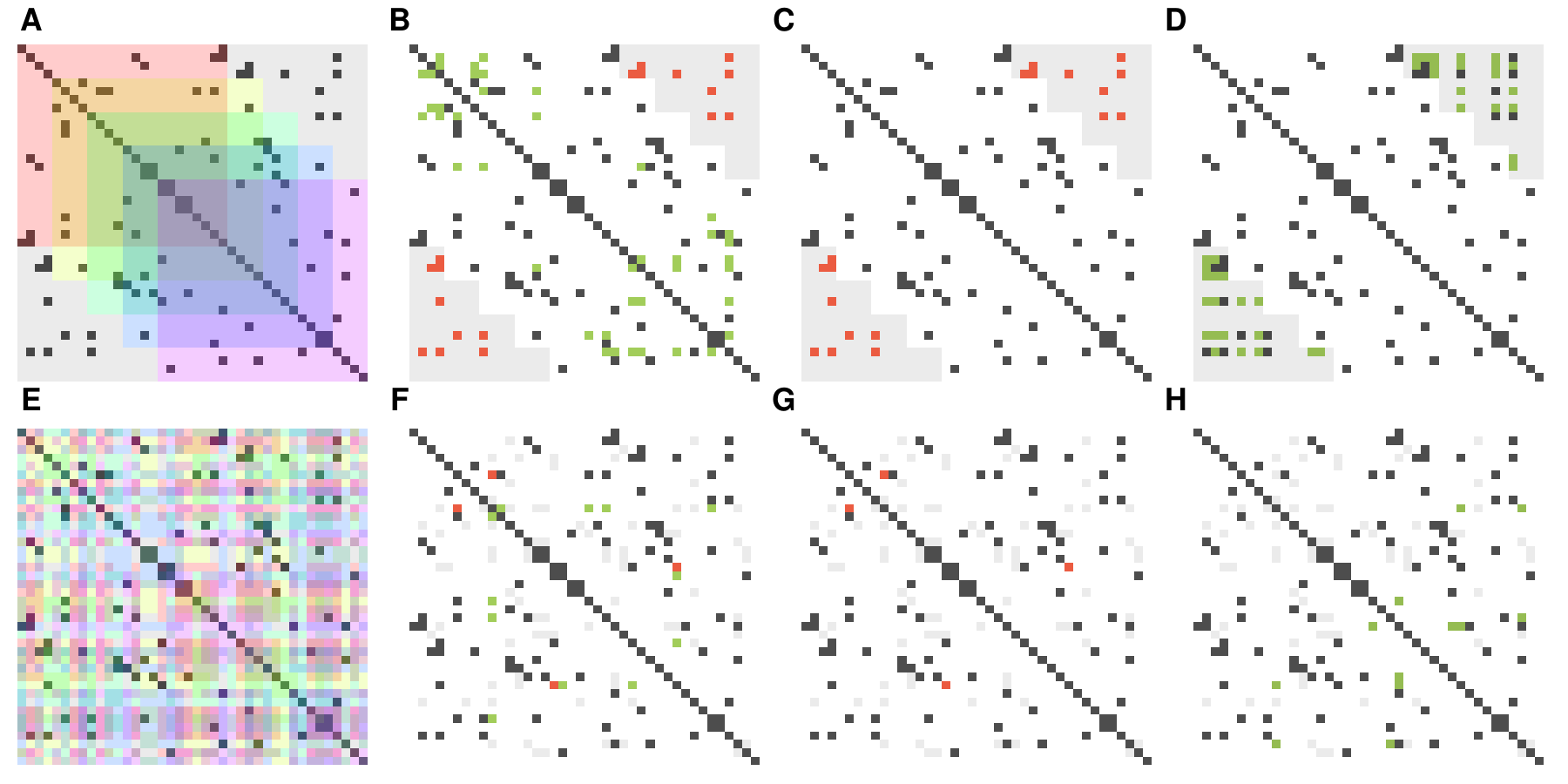}\less\less
\caption{
Example of graph recovery via \madgq{}. 
{\bf (A)} Support of $\Theta$ ($p=40$) and observed node pairs set $O$ (colored area), and $O^c$ (grey area). 
{\bf (B)} Support of the \madgq{} matrix $\tilde\Theta$: the green entries denote false positive edges, whereas the red ones are false negatives. No false negatives are in $O$ as per Theorem~\ref{theo:nofnd}.
{\bf (C)} Recovered graph $\tilde E^{\tau}$ (Equation~(\ref{eq:Etau})) with $\tau=\nu/2\in[\delta,\nu-\delta)$. In this example, the largest magnitude of $\Theta_{O^c}$, $\gamma$, is sufficiently small, so $\tilde E^\tau_O$ perfectly matches the true edge set $E_O$, as per Theorem~\ref{theo:popO}. 
{\bf (D)} Fully recovered edge set $\mathcal{E}^\tau$ via Algorithm~\ref{algo:fullrecoveryK}, consisting of the union of the recovered graph in $O$ as in (C), with the minimal superset of edges in $O^c$. 
{\bf (E)-(H)}. Analogous to (A)-(D), except that the sets $V_1,...,V_5$ are random.
}\label{fig:poprecovery}\less
\end{figure}

\subsubsection{Special case \texorpdfstring{$K=2$}{Lg}}\label{sec:caseK2}\less\less
In this section, we focus on a simple but practically relevant illustrative case where we observe only two vertex subsets $V_1$ and $V_2$. For this case, the \madgq{} optimization problem in Equation~(\ref{eq:gq0}) has a tractable closed-form solution (Equation~(\ref{thetaMD3x3}), Appendix~\ref{app:abcproofs}), which allows us to analyze the Graph Quilting problem in greater detail analytically. For simplicity of exposition, we shall let $V_1=A\cup B$ and $V_2=B\cup C$, where $A,B,C$ is a partition of $V$, so that $O^c=(A\times C)\cup (C\times A)$ and $B$ contains the overlapping vertices between the two observation sets $V_1$ and $V_2$. 
We identify three situations where the condition $\delta<\nu/2$ is satisfied, and we specify them in terms of $\gamma:=\Vert\Theta_{O^c}\Vert_\infty=\Vert\Theta_{AC}\Vert_\infty$:
\begin{enumerate}
    \item[(A1).] $E_{AC}=\emptyset$, i.e. $\gamma=0$.
    \item[(A2).] $B$ is disconnected from $A$ and $C$ and $0<\gamma <\sqrt{\frac{\nu\lambda_{\min}}{2d_{O^c}^2}} $, where $\lambda_{\min}$ is the smallest eigenvalue of $\Theta$, and $d_{O^c}$ is the max node-degree in the sub-graph $E_{AC}$.
    \item[(A3).] $0<\gamma<\frac{-b+\sqrt{b^2+2a\nu}}{2a}$, where $a=d_{O^c}^2(\lambda_{\min}^{-1}+2q^2d_B^2\gamma_B^2\lambda_{\min}^{-3})$, $b=d_{O^c}(d_B\gamma_B\lambda_{\min}^{-1}+2qd_B^2\gamma_B^2\lambda_{\min}^{-2})$, 
 $q=\max\{|A|,|C|\}$, $d_B$ is the largest number of edges from one node in $B$ to $A$ or to $C$, and $\gamma_B=\Vert\Theta_{B(AC)}\Vert_\infty$. 
\end{enumerate}
The following is a corollary of Theorem~\ref{theo:popO} for the case $K=2$:
\begin{corollary}[\scbf{Exact Graph Recovery in $O$ (special case $K=2$)}]\label{coro:popOABC} If Condition (A1) or (A2) or (A3) hold, then $\delta <\nu/2$  and for any $\tau\in [\delta,\nu-\delta)$, we have $\tilde E_O^{\tau} = E_O$ (Equation~(\ref{eq:Etau})), and ${\rm sign}(\tilde\Theta_{ij})={\rm sign}(\Theta_{ij})$,  $\forall(i,j)\in E_O$.
\end{corollary}
Condition (A1) corresponds to the simplest situation depicted by Theorem~\ref{thm:identifiability}, where  $E_{AC}=\emptyset$ guarantees $\tilde\Theta=\Theta$, yielding $\delta=0<\nu/2$. Conversely, conditions (A2) and (A3) exploit several, rather technical, matrix inequalities given in Appendices~\ref{app:matrixineq}
and \ref{app:popcaseK2}, which explicitly relate the magnitude $\gamma$ of the strongest edge in $O^c$ to other quantities  characterizing $\Theta$. 
We can see that the exact graph recovery in $O$ is easier to accomplish when the magnitude $\nu$ of the weakest edge in $O$ and the smallest eigenvalue $\lambda_{\rm min}$ of $\Theta$ are large, while the size $q$ and maximum node degree $d_{O^c}$ in $O^c$ are small. Finally, note that (A3) reduces to (A2) if $\gamma_B\to 0$. In Appendix~\ref{app:popcaseK2} we discuss the special case $K=2$ in more detail.

\paragraph*{The latent variable graphical model} 
In this paragraph we illustrate the relationship between Graph Quilting in the case $K=2$, and the problem of estimating a conditional dependence graph in the presence of latent variables. Suppose that $V = A \cup C$ where $A\cap C=\emptyset$, and that the nodes in $C$ are hidden. It is known that $\Sigma_{AA}^{-1} ~=~ \Theta_{AA}-\Theta_{AC}\Theta_{CC}^{-1}\Theta_{CA}$, where $\Theta_{AA}$ is the $A\times A$ portion of the precision matrix $\Theta$ encoding the dependence structure of nodes $A$ conditionally on $C$, while the second term of the right-hand-side has rank no larger than $|A|$, and accounts for the network effects of the hidden nodes in $C$. Based on this fact, \cite{chandrasekaran2012} proposed to estimate $E_{AA}$ -- the graph structure in $\Theta_{AA}$ -- by first estimating the inverse covariance matrix of $A$ as $\hat{\Sigma_{AA}^{-1}} ~=~ \hat S-\hat L$, where $\hat S$ is a sparse matrix and $\hat L$ is a low rank matrix, and then taking the support of $\hat S$ as an estimate of $E_{AA}$. 

Suppose now we are in a Graph Quilting scenario with $O=(A\times A)\cup(C\times C)$. Then, the \madgq{} solution (Equation~(\ref{eq:gq0})) is equal to
\begin{equation}
\tilde\Theta = \left[
\begin{array}{cc}
\Sigma_{AA}^{-1} & 0\\
0 & \Sigma_{CC}^{-1}
\end{array} \right]
\end{equation}
First note that Theorem~\ref{theo:nofnd} guarantees that $\tilde\Theta_O$ contains no false negatives.  In other words, ignoring the hidden nodes in the latent variable graphical model problem yields no false negatives and can only lead to false positives.   Moreover, Theorem~\ref{theo:popO} and Corollary~\ref{coro:popOABC} establish that, under appropriate conditions, $E_O$ can be perfectly recovered from $\tilde\Theta_O$ by assigning edges wherever $|\tilde\Theta_{ij}|>\tau$, for any $\tau\in[\delta,\nu-\delta)$. But since $\tilde\Theta_{AA}$ is not a function of $\Sigma_{CC}$, this thresholding is valid even if we only observe $A\times A$, that is even if nodes $C$ are unobserved! The following corollary summarizes this result: 
\begin{corollary}[\scbf{Latent variable graphical model}]\label{coro:poplvgm}
Suppose we observe $\Sigma_{AA}$, and 
 $C=V\setminus A\neq\emptyset$ are hidden nodes. If $\gamma <\sqrt{\frac{\nu\lambda_{\min}}{2d_{O^c}^2}} $, then $\delta <\nu/2$  and for any $\tau\in [\delta,\nu-\delta)$ we have $\left\{(i,j)\in A\times A:~i\neq j, \left|[\Sigma_{AA}^{-1}]_{ij}\right|>\tau\right\} ~=~ E_{AA}$, and ${\rm sign}([\Sigma_{AA}^{-1}]_{ij})={\rm sign}(\Theta_{ij})$ for all $(i,j)\in E_{AA}$.
\end{corollary}
This corollary states that the subgraph connecting the nodes $A$ within the full conditional dependence graph of $V=A\cup C$ can be retrieved by just appropriately thresholding the entries of $\Sigma_{AA}^{-1}$. Indeed, under the assumptions of the corollary, $\Sigma_{AA}^{-1}$ contains no false negative edges (also in agreement with Theorem~\ref{theo:nofnd}), but only weak false positive edges which are all eliminated by the thresholding operation at level $\tau$, with no risk of producing any false negative edges. Consequently, at the estimation level, it may be possible to avoid estimating the two matrix components $S$ and $L$ of the sparse and low-rank decomposition \cite{chandrasekaran2012}, but rather just obtain a good estimate of $\Sigma_{AA}^{-1}$ to threshold, involving the estimation of a much smaller number of parameters. This approach has been recently explored in \cite{wang2021thresholded}.

\subsection{Graph recovery in \texorpdfstring{$O^c$}{Lg} via Oracle Distortions in \texorpdfstring{$O$}{Lg}}\label{sec:popOc}\less
Recovering the edge set $E_{O^c}$ from $\Sigma_O$ is a seemingly impossible task because it requires us to verify conditional dependences of variable pairs with no information about their marginal dependences. However, here we show that with some assumptions it is actually possible to retrieve substantial information about $E_{O^c}$, even with no knowledge of $\Sigma_{O^c}$! This is possible because the distortions between $\Theta_O$ and $\tilde\Theta_O$ have a pattern which depends on the precise edge structure in $O^c$, and they can be used to triangulate the plausible graphical structures in $O^c$. 

This section is organized as follows. We first study how the distortions propagate in $\tilde\Theta_O$ depending on the precise position of the edges in $O^c$. We then introduce the definition of minimal superset of the edge set $E_{O^c}$ based on the oracle knowledge of the distortions. The oracle results presented in this section, while impractical, constitute the theoretical foundations of the more practical approach proposed in Section~\ref{sec:popFull}, which does not require oracle knowledge of the distortions in $O$. 

\paragraph*{Notation} The following graph theoretic terminology will help characterize the graph recovery in $O^c$. Let $U \subseteq V$ be an arbitrary subset of nodes and let $G_U$ denote the subgraph of $G$ induced by $U$, i.e., the graph whose vertex set is $U$ and edge set is $E \cap (U\times U)$; indeed, $G_V = G$. We will let $N(i):=\{j\in V: (i,j)\in E\}$ denote the neighborhood of $i$. Two nodes $i$ and $j$ are neighbours (a.k.a. adjacent) if $i\in N(j)$, or equivalently, $j\in N(i)$. We further let $N_U(i):=N(i)\cap U$ be the set of neighbours of $i$ that are in $U$. Given two subsets $U,F\subseteq V$, we let $N_U(F) := \bigcup_{i\in F}N_U(i)\subseteq U$ be the set of nodes in $U$ that are neighbours of one or more nodes in $F$. Two nodes $i,j\in U$ are $U$-connected if they are  connected through some path completely within $U$.

\subsubsection{Distortion propagation}\less
The main component of our approach to recover edges in $O^c$ is given by the following fundamental property that entangles the \madgq{} matrix  $\tilde\Theta$ with the true precision matrix $\Theta$ through $\Sigma$ by virtue of the Schur complement: 
\begin{lemma}[\madgqbf{} \scbf{Entanglement}]\label{lemma:schurgen}
For any set $U\subseteq V$ such that $U\times U\subseteq O$,
\begin{equation}\label{eq:thetabarvk1}
\tilde\Theta_{UU}-\tilde\Theta_{UU^c}\tilde\Theta_{U^cU^c} ^{-1}\tilde\Theta_{U^cU}~=~\Sigma_{UU}^{-1}~=~\Theta_{UU}-\Theta_{UU^c}\Theta_{U^cU^c}^{-1}\Theta_{U^cU}.
\end{equation}
\end{lemma}
In order to use this result for the identification of the edges in $O^c$, first let us define some useful quantities: the $k$-th \madgq{} Schur complement is given by
\begin{equation}\label{eq:madgqschur}
\tilde\Theta^{(k)} ~:=~ \tilde\Theta_{V_kV_k}-\tilde\Theta_{V_kV_k^c}\tilde\Theta_{V_k^cV_k^c}^{-1}\tilde\Theta_{V_k^cV_k},
\end{equation}
and the $k$-th block distortion of the node pair $(i,j)\in V_k\times V_k$ is given by
\begin{equation}\label{eq:distijk}
\delta^{(k)}_{ij} ~:=~ \Theta_{ij}-\tilde\Theta^{(k)}_{ij},
\end{equation}
where $\tilde\Theta^{(k)}_{ij}$ is the entry of $\tilde\Theta^{(k)}$ relative to the node pair $(i,j)$. 
Moreover, let 
\begin{equation}
    H_i ~:=~ \left\{j\in V: (i,j)\in O^c\right\}
\end{equation}
be the set of nodes that are not jointly observed with node $i$ (i.e. $\Sigma_{ij}$ is not observed if $j\in H_i$). The next theorem precisely describes the relationship between distortions and edges in $O^c$:
\begin{theorem}[\scbf{Distortion Propagation}]\label{theo:distprK}
Let $k\in\{1,...,K\}$:
\begin{enumerate}[(i).]
\item For any $i\in V_k$ with $H_i\neq\emptyset$, we have
\begin{equation}\label{eq:deltaii}
\delta^{(k)}_{ii}> 0 ~~~\text{if}~~~ \Theta_{iH_i}\neq 0.
\end{equation}
If $H_i=V_k^c$, then the condition is sufficient and necessary.
\item For any $i,j\in V_k$ with $H_i\neq\emptyset$ and $i\neq j$, almost everywhere, we have
\begin{equation}\label{eq:distpropijcond}
\delta^{(k)}_{ij}\neq 0 ~~\text{if}~~ \exists h\in N_{H_i}(i) \text{ and } \exists l\in N_{H_i}(j) \text{ s.t. } h=l \text{ or } h \text{ is }  H_i\text{-connected to}~l. 
\end{equation}
If $H_i=V_k^c$, then the condition is sufficient and necessary.
\item 
For any $i,j\in V_k$, if $\delta_{ij}^{(k)}\neq0$ then $\delta_{ii}^{(k)}>0$ and $\delta_{jj}^{(k)}>0$.
\end{enumerate}
\end{theorem}
Part~(i) of Theorem~\ref{theo:distprK} states that if node $i$ is neighbour of some node in $H_i$, then there will be a distortion on the diagonal entry of node $i$ in every \madgq{} Schur complement $\tilde\Theta^{(k)}$ where $i\in V_k$. Part~(ii) states that an entry $(i,j)$ of a \madgq{} Schur complement is distorted if nodes $i$ and $j$ are connected through some path of length $>1$ completely within $H_i\cup\{i,j\}$ and/or $H_j\cup\{i,j\}$. Indeed, the \madgq{} optimization (Equation~(\ref{eq:maxdet})) sets $\tilde\Theta_{iH_i}=0$ and $\tilde\Theta_{jH_j}=0$, thereby disrupting any dependence path between $i$ and $j$ and generating a distortion. Finally, part~(iii) reveals that an off-diagonal entry $(i,j)$ of a \madgq{} Schur complement is distorted only if the diagonals $(i,i)$ and $(j,j)$ are distorted. The following corollary of Theorem~\ref{theo:distprK} focuses on the specific effects of an edge $(i,j)\in O^c$ on the entries of the \madgq{} Schur complements:
\begin{corollary}\label{coro:distpropagK}
For any $(i,j)\in O^c$, we have
\begin{enumerate}[(i).]
    \item $\Theta_{ij}\neq 0\Rightarrow \delta^{(k)}_{ii},\delta^{(h)}_{jj}> 0$, for all $k$ and $h$ such that $i\in V_k$ and $j\in V_h$.
    \item $\Theta_{ij}\neq 0\Rightarrow \delta^{(k)}_{is}\neq 0$, for all $k$ and $s$ such that $i\in V_k$, and $s\in V_k\setminus\{i\}$ is $(H_i\cup\{s\})$-connected to $j$ [a.e.].
\end{enumerate}
\end{corollary}
Part~(i) of the corollary states that if $(i,j)\in E_{O^c}$, then the diagonal entries $(i,i)$ and $(j,j)$ of all related \madgq{} Schur complements will be distorted. Part~(ii) states that if $(i,j)\in E_{O^c}$, then there will also be a distortion in the off-diagonal entry $(i,s)$ of all related \madgq{} Schur complements as long as node $j$ is connected to $s$ through some path of length $>1$ completely within $H_i\cup\{s,j\}$.

\subsubsection{Superset minimality}\less
In this section we establish that with incomplete covariance information it is at least possible to recover a minimal superset of $E_{O^c}$ by exploiting the types of distortions considered in the Distortion Propagation Theorem~\ref{theo:distprK}. The minimal superset is defined as follows:
\begin{definition}[\scbf{Minimal Superset  of $E_{O^c}$}]\label{def:minimalityK}
Let 
\begin{equation}
    \mathcal{D}_Q(\Sigma, O):=\big\{(i,j,k)\in Q: ~\delta^{(k)}_{ij}\neq 0\big\}
\end{equation}
be the set of known distortions over the entries $Q\subseteq \bigcup\limits_{k=1}^K V_k\times V_k \times \{k\} $ of the Schur complements $\tilde\Theta^{(1)},...,\tilde\Theta^{(K)}$, and let 
\begin{equation}\label{eq:admiss}
\mathcal{A}(\Sigma,O,Q):=\left\{\Sigma'\succ 0: ~\Sigma'_O=\Sigma_O,~ \mathcal{D}_Q(\Sigma',O)=\mathcal{D}_Q(\Sigma,O)\right\}
\end{equation}
be the set of all positive definite covariance matrices that agree with the observed $\Sigma_O$ and distortions $\mathcal{D}_Q(\Sigma, O)$. A set $\mathcal{S}\subseteq O^c$ is the minimal superset of $E_{O^c}$ with respect to $\Sigma_O$ and $\mathcal{D}_Q(\Sigma,O)$ if it satisfies the following properties:
\begin{enumerate}[(i).]
\item $\forall\Sigma'\in\mathcal{A}(\Sigma,O,Q)$ we have $E_{O^c}'\subseteq \mathcal{S}$;
\item $\forall\mathcal{S}' \subsetneq \mathcal{S}$, $\exists\Sigma'\in\mathcal{A}(\Sigma,O,Q)$ such that $E'_{O^c} \cap (\mathcal{S}\setminus \mathcal{S}') \neq \emptyset$.   
\end{enumerate}
\end{definition}
Thus, a minimal superset $\mathcal{S}$  of $E_{O^c}$ given the set of known (oracle) distortions $\mathcal{D}_Q(\Sigma,O)$ is the smallest possible superset in the sense that it includes all plausible graphical structures $E_{O^c}$ that would induce the same known (oracle) distortions in the \madgq{} Schur complements. Thus, any other set $\mathcal{S}'\neq\mathcal{S}$ is either not a superset of $E_{O^c}$, or it is larger than $\mathcal{S}$, or it does not include one or more plausible edges. An expression of the minimal superset is
\begin{equation}
\mathcal{S} := \bigcup\limits_{\Sigma'\in \mathcal{A}(\Sigma,O,Q)}\big\{(i,j)\in O^c:~ \left[\Sigma'^{-1} \right]_{ij}\neq 0\big\}
\end{equation}
In the following, we will consider the cases where we have oracle knowledge of all distortions on the diagonal entries of the \madgq{} Schur complements, or on the off-diagonals.

\subsubsection{Oracle minimal superset recovery}\less 
Towards the statement of our main Theorem~\ref{theo:OcK} for the oracle recovery of $E_{O^c}$, let us first define some quantities. Define the set
\begin{equation}\label{eq:supersetK}
\mathcal{S}_{\rm diag} ~:=~ O^c\cap (D_{\rm diag}\times D_{\rm diag})
\end{equation}
where $D_{\rm diag}=\big\{i\in V: ~\delta_{ii}^{(k)}>0,~\forall k ~\text{s.t}.~ i\in V_k\big\}$ is the set of nodes with at least one diagonal distortion in every related Schur complement. Moreover, define the set 
\begin{equation}\label{eq:supersetKoff}
    \mathcal{S}_{\rm off} ~:=~ O^c\cap (D_{\rm off}\times D_{\rm off})
\end{equation}
where $D_{\rm off}= \big\{i\in V:~\delta^{(k)}_{i(-i)}\neq 0,~\forall k ~\text{s.t}.~ i\in V_k\big\} $ is the set of nodes with at least one off-diagonal distortion in every related Schur complement, and $\delta^{(k)}_{i(-i)}=(\delta^{(k)}_{ik})_{j\in V_{k}\setminus\{i\}}$ is the vector of distortions on the row of node $i$ in the \madgq{} Schur complement $\tilde\Theta^{(k)}$. Furthermore, consider the following assumption: 
\begin{itemize}
\item[(A4).] For every node $i\in V$ with $N_{H_i}(i)\neq\emptyset$, we have that for every $k$ such that $i\in V_k$, there exists at least one node $j\in V_k\setminus \{i\}$ that is $(H_i\cup\{j\})$-connected to some node in $N_{H_i}(i)$.
\end{itemize}
We are now ready to state our main theorem for the oracle recovery of $E_{O^c}$:
\begin{theorem}[\scbf{Oracle Minimal Superset of $E_{O^c}$}]\label{theo:OcK}
Let $\mathcal{S}_{\rm diag}$ and $\mathcal{S}_{\rm off}$ be the sets in Equations~(\ref{eq:supersetK}) and (\ref{eq:supersetKoff}). Then, in the sense of Definition~\ref{def:minimalityK}:
\begin{enumerate}[(i).]
\item The set $\mathcal{S}_{\rm diag}$ is the minimal superset of $E_{O^c}$ given the set of diagonal distortions $\mathcal{D}_{\rm diag}(\Sigma, O) = \big\{(i,i,k): \delta^{(k)}_{ii}>0 \big\}$. 
\item The set $\mathcal{S}_{\rm off}\subseteq \mathcal{S}_{\rm diag}$ is the minimal superset of $E_{O^c}$ [a.e.] given the set of  off-diagonal distortions $\mathcal{D}_{\rm off}(\Sigma, O)=\big\{(i,j,k):\delta^{(k)}_{ij}\neq 0, i\neq j\big\}$ if and only if Assumption~(A4) holds.
\end{enumerate}
\end{theorem}
Part~(i) of the theorem establishes (constructively) than for any set returned that is smaller that $\mathcal{S}_{\rm diag}$, there are problems where one necessarily will fail to detect true edges in $E_{O^c}$. Part~(ii) of the theorem establishes that, under assumption~(A4), the set $\mathcal{S}_{\rm off}$ in Equation~(\ref{eq:supersetKoff}) is the minimal superset of $E_{O^c}$ based on the knowledge of the off-diagonal distortions.

\subsection{Full graph recovery}\label{sec:popFull}\less
We now condense the results of Sections~\ref{sec:popO} and \ref{sec:popOc} into one algorithm, Algorithm~\ref{algo:fullrecoveryK}, for the recovery of the full edge set $E$. This algorithm does not require the oracle knowledge of the distortions for the recovery of the edges in $O^c$, but instead it only exploits the off-diagonal entries in the \madgq{} Schur Complements that are identified as distorted because their magnitudes are too small. Theorem~\ref{theo:fullrecoveryK} establishes the properties of the output edge set $\mathcal{E}^\tau$ of Algorithm~\ref{algo:fullrecoveryK}, and requires the following assumption:
\begin{enumerate}
    \item[(A5).]  If $\delta_{i(-i)}^{(k)}\neq 0$, then there exists $j\neq i$ such that $0<|\tilde\Theta_{ij}^{(k)}|<\delta$.
\end{enumerate}
\begin{theorem}[\scbf{GQ Graph recovery (population case)}]\label{theo:fullrecoveryK}
If Assumptions (A4)-(A5) and the conditions of Theorem~\ref{theo:popO} hold such that $\delta<\nu/2$, then, for any $\tau\in [\delta,\nu-\delta)$, the output edge set $\mathcal{E}^\tau$ of Algorithm~\ref{algo:fullrecoveryK} satisfies  $\mathcal{E}^\tau_O=E_O$ and $\mathcal{E}_{O^c}^\tau=\mathcal{S}_{\rm off}$, where $\mathcal{S}_{\rm off} $ is the minimal superset of $E_{O^c}$ [a.e.] in Equation~(\ref{eq:supersetKoff}).
\end{theorem}
Theorem~\ref{theo:fullrecoveryK} combines Theorem~\ref{theo:popO} and Theorem~\ref{theo:OcK}. 
The set $\mathcal{E}_O^\tau\equiv \tilde E_O^\tau$ equals the true edge set $E_O$ since $\delta<\nu/2$ and $\tau\in[\delta,\nu-\delta)$, as per Theorem~\ref{theo:popO}. This means that no off-diagonal entry of $\Theta_O$ has magnitude in the interval $(0,\tau]$. Hence, if $0<|\tilde\Theta^{(k)}_{ij}|<\tau$, then $\delta^{(k)}_{ij}\neq 0$. Thus, under Assumption~(A5), if $\tau\in[\delta,\nu-\delta)$, then the set $W_\tau$ in Algorithm~\ref{algo:fullrecoveryK} contains every node $i$ that is associated with at least one off-diagonal distortion in every \madgq{} Schur Complement $\tilde\Theta^{(k)}$ where $i\in V_k$. Therefore, $W_\tau$ matches the set $D_{\rm off}$ in Equation~(\ref{eq:supersetKoff}) and thereby $\mathcal{E}_{O^c}^\tau\equiv\mathcal{S}_{\rm off}$, where, under Assumption~(A4), $\mathcal{S}_{\rm off}$ is the minimal superset of $E_{O^c}$ as per Theorem~\ref{theo:OcK}. Examples of full graph recovery are shown in Figures~\ref{fig:poprecovery}(D) and \ref{fig:poprecovery}(H).

\begin{algorithm}[t]\normalsize
\scbf{Input}: $V_1,...,V_K$, $\Sigma_O$, $\tau>0$\;
 \begin{enumerate}
    \item  Compute the \madgq{} matrix
    \[
    \tilde\Theta ~=~  \underset{\Theta\succ 0, ~\Theta_{O^c}=0}{\arg\max}~ \log\det \Theta - \sum_{(i,j)\in O}\Theta_{ij}\Sigma_{ij}
    \]
    \item Find the edge set $\tilde E_O^\tau=\big\{(i,j)\in O: i\neq j, |\tilde\Theta_{ij}|>\tau\big\}$.
    
    \item For $k=1,...,K$, compute the Schur complement
    \[
    \tilde\Theta^{(k)}~:=~\tilde\Theta_{V_kV_k}-\tilde\Theta_{V_kV_k^c}\tilde\Theta_{V_k^cV_k^c}^{-1}\tilde\Theta_{V_k^cV_k}
    \] 
    \item Obtain the node set
    \[W_\tau=\big\{i\in V:~ \forall k ~s.t.~i\in V_k,~\exists j\neq i, 0<|\tilde\Theta^{(k)}_{ij}|< \tau\big\}
    \]
    \item Obtain the set $\mathcal{U}_\tau=O^c\cap(W_\tau\times W_\tau)$.
\end{enumerate}
\scbf{Output}:  Edge set
\begin{equation}\label{eq:fulledegqpop}
\mathcal{E}^\tau ~=~ \tilde E_O^\tau\cup \mathcal{U}_\tau\vspace{-6mm}
\end{equation}
\caption{\scbf{GQ graph recovery (population case)}}\label{algo:fullrecoveryK}
\end{algorithm}

\section{Graph Recovery: Finite Sample Analysis}\label{sec:estimation}\less
In Section~\ref{sec:popK}, we investigated the Graph Quilting problem at the population level, where we have access to the true incomplete covariance matrix $\Sigma_O$. In this section, we investigate the Graph Quilting problem in the finite sample setting, where the population quantity $\Sigma_O$ is replaced by an empirical estimate $\hat\Sigma_O$. This setting is more challenging because $\hat\Sigma_O$ is not guaranteed to be completable to a positive definite matrix and, consequently, the \madgq{} optimization (Equation~(\ref{eq:gq0})) based on $\hat\Sigma_O$ in place of $\Sigma_O$ is not guaranteed to produce a unique solution. We circumvent this issue by using regularization. We propose the \madgqlassobf{}, an $\ell_1$-regularized variant of the \madgq{} that performs simultaneously precision matrix reconstruction and regularized estimation based on $\hat\Sigma_O$. The \madgqlasso{} estimator $\hat{\tilde\Theta}$ is well defined in high dimensions and converges to the \madgq{} solution $\tilde\Theta$ (Equation~(\ref{eq:gq0})) with rates similar to the graphical lasso \cite{yuan2007model,ravikumar2011high}. We use the \madgqlasso{} to construct a graph estimator following the procedures developed in Section~\ref{sec:popK}. 
In Section~\ref{sec:estimator}, we define our estimators, and in Section~\ref{sec:statprop}, we establish their statistical properties.

\subsection{Estimators}\label{sec:estimator}\less
Let $X^{(1)}$, ..., $X^{(n)}$ be independent and identically distributed $p$--dimensional random vectors with mean vector $\mu\in\mathbb{R}^p$ and $p\times p$ positive definite covariance matrix $\Sigma\succ 0$. Let $V^{(r)}=\{i\in V:X_i^{(r)}\text{ is observed}\}$ be the set of nodes that are observed on sample $r$ and let $n_{ij}=|\{r:(i,j)\in V^{(r)}\times V^{(r)}\}|$ be the joint sample size for node pair $(i,j)$. Moreover, let $O=\bigcup_{r=1}^nV^{(r)}\times V^{(r)}$. We define the observed sample covariance as follows: 
\begin{definition}[\scbf{Observed Sample Covariance}]\label{def:obscov}
The observed sample covariance of the pair of nodes $(i,j)\in O$ is given by
\begin{equation}\label{eq:compcov}
\hat\Sigma_{ij} ~:=~ \frac{1}{n_{ij}}\sum_{r:(i,j)\in V^{(r)}\times V^{(r)}}(X_i^{(r)}-m_i)(X_j^{(r)}-m_j)
\end{equation}
where $m_k=\frac{1}{n_{kk}}\sum_{r:k\in V^{(r)}}X_k^{(r)}$, or $m_k=\mathbb{E}[X_k]$ if known.
\end{definition}
Assuming $\mathbb{E}[X_i^2]<\infty$ for all $i\in V$, by Weak Law of Large Numbers, $\hat\Sigma_{ij}\stackrel{P}{\to}\Sigma_{ij}$ as $n_{ij}\to\infty$, 
for any $(i,j)\in O$. Yet, for finite sample sizes $(n_{ij})_{(i,j)\in O}$, the principal minors of the incomplete matrix $\hat\Sigma_O$ are not all guaranteed to be positive, in which case $\hat\Sigma_O$ may not be completed into a positive definite matrix and, consequently, the \madgq{} problem in Equation~(\ref{eq:gq0}) would not yield a unique solution if based on $\hat\Sigma_O$ in place of $\Sigma_O$. We use regularization to overcome this problem and to further improve estimation accuracy in high-dimensions. We propose the \madgqlasso{}, an $\ell_1$-regularized variant of the \madgq{} optimization problem (Equation~(\ref{eq:gq0})):
\begin{definition}[{\bf MAD$_{\mathbf{GQlasso}}$}]\label{def:l1gq}
The MAD Graph Quilting lasso is the solution of the $\ell_1$-penalized optimization problem \begin{equation}\label{eq:l1gq}
~~~\hat{\tilde\Theta} ~:=~ \underset{\Theta\succ 0,\Theta_{O^c}=0}{\arg\max}~~ \log\det\Theta - \sum_{(i,j)\in O}\Theta_{ij}\hat\Sigma_{ij}  ~-~\Vert \Lambda\odot \Theta\Vert_{1,\rm off} ~~~~(\textsc{MAD}_{\rm GQlasso})
\end{equation}
where $ \hat\Sigma_{ij}$ is the observed sample covariance defined in Equation~(\ref{eq:compcov}), $\Lambda=[\lambda_{ij}]\in\mathbb{R}_{0,+}^{p\times p}$ is a matrix of nonnegative penalty parameters,  $\odot$ denotes the Hadamard entrywise matrix product, and $\Vert M\Vert_{1,{\rm off}}=\sum_{i\neq j}|M_{ij}|$ is the $\ell_1$ matrix norm computed only over the off-diagonals of the matrix $M\in\mathbb{R}^{p\times p}$.\vspace{-1mm}
\end{definition}
The \madgqlasso{} optimization problem in Equation~(\ref{eq:l1gq}) combines the \madgq{} problem in Equation~(\ref{eq:gq0}), which imposes the constraint $\Theta_{ij}=0$ for all $(i,j)\in O^c$, with an $\ell_1$ penalty over the off-diagonal entries of $\Theta$. The following lemma guarantees that Equation~(\ref{eq:l1gq}) has a unique solution as long as the diagonals of $\hat\Sigma_O$ are positive, without requiring all principal minors of $\hat\Sigma_O$ to be positive:
\begin{lemma}\label{lemma:gqlassoexun} The \madgqlasso{} optimization problem in Equation~(\ref{eq:l1gq}) has a unique solution if $\Vert\hat\Sigma_O\Vert_\infty<\infty$, and $\hat\Sigma_{ii}>0$ for all $i\in V$, and $\lambda_{ij}>0$, for all $(i,j)\in O$, $i\neq j$.
\end{lemma}

Note that also the graphical lasso estimator \cite{yuan2007model} imposes an $\ell_1$ penalty which enforces sparse solutions, but it assumes $O\equiv V\times V$ and $n_{ij}\equiv n\ge 2$ for all node pairs $(i,j)$. Therefore, the \madgqlasso{} framework is more general than the graphical lasso, although it is an estimator of $\tilde\Theta$, rather than $\Theta$. It is also important to notice that the \madgqlasso{} optimization problem is not equivalent to a graphical lasso where we set $\hat\Sigma_{O^c}=0$ and do not impose the constraint $\Theta_{O^c}=0$. Indeed, in such case $\Theta_{O^c}$ would still be active in the optimization. 

We define the GQ graph estimator as the output $\hat{\mathcal{E}}$ of  Algorithm~\ref{algo:fullrecoveryKn}, which is based on the \madgqlasso{} and is the finite sample version of Algorithm~\ref{algo:fullrecoveryK}, where  $X^{(1)}_{V^{(1)}},...,X^{(n)}_{V^{(n)}}$ are the observed data and $V_1,...,V_K\subset V$ are such that $O:=\cup_{r=1}^n V^{(r)}\times V^{(r)} = \cup_{k=1}^K V_k\times V_k$, with smallest possible $K$. 
This algorithm follows the structure of  Algorithm~\ref{algo:fullrecoveryK}, except that the \madgq{} matrix $\tilde\Theta$ is replaced by the \madgqlasso{} estimator $\hat{\tilde\Theta}$ and, compared with the set $W_\tau$, the set of nodes $\hat W_{\tau_0,\tau_1}$ involves the additional threshold parameters $\tau_0$ and $\tau_1$ to better deal with the randomness of the \madgqlasso{} Schur complements. For example, $\tau_0>0$ is essential to reduce the number of false positive edges in $\hat{\mathcal{E}}_{O^c}$. The theorems presented next establish the optimal oracle choices of $\Lambda$, $\tau_0$, $\tau$, and $\tau_1$ as functions of sample size, number of nodes, max node degree, and size of $O^c$.

\vspace{-1mm}

\begin{algorithm}[t]\normalsize
\scbf{Input}: Observed data $X^{(1)}_{V^{(1)}},...,X^{(n)}_{V^{(n)}}$;
$V_1,...,V_K$; $\tau,\tau_0,\tau_1\ge 0$;
$\Lambda\in\mathbb{R}_{0,+}^{p\times p}$\;
\begin{enumerate}
\item Compute the observed covariances $\hat\Sigma_O=\big(\hat\Sigma_{ij}\big)_{(i,j)\in O}$ (Equation~(\ref{eq:compcov})) based on\\ the observed data, where $O=\cup_{k=1}^K V_k\times V_k$.
    \item  Compute the \madgqlasso{} matrix
    \[
    \hat{\tilde\Theta} ~=~ \underset{\Theta\succ 0,\Theta_{O^c}=0}{\arg\max}~~ \log\det\Theta - \sum_{(i,j)\in O}\Theta_{ij}\hat\Sigma_{ij}  ~-~\Vert \Lambda\odot \Theta\Vert_{1,\rm off}
    \]
    \item Find the edge set $\hat E_O^\tau=\big\{(i,j)\in O: i\neq j, |\hat{\tilde\Theta}_{ij}|>\tau\big\}$.
    \item For $k=1,...,K$, compute the Schur complement
    \[
 \hat{\tilde\Theta}^{(k)}~:=~ \hat{\tilde\Theta}_{V_kV_k}- \hat{\tilde\Theta}_{V_kV_k^c} \hat{\tilde\Theta}_{V_k^cV_k^c}^{-1} \hat{\tilde\Theta}_{V_k^cV_k}
    \] 
    \item Obtain the set
    \[\hat W_{\tau_0,\tau_1}=\big\{i\in V:~ \forall k ~s.t.~i\in V_k,~\exists j\neq i, \tau_0<| \hat{\tilde\Theta}^{(k)}_{ij}|< \tau_1\big\}
    \]
    \item Obtain the set $\hat{\mathcal{U}}_{\tau_0,\tau_1}=O^c\cap(\hat W_{\tau_0,\tau_1}\times \hat W_{\tau_0,\tau_1})$.
\end{enumerate}
\scbf{Output}:  Edge set
\begin{equation}
\hat{\mathcal{E}} ~:=~ \hat E_O^\tau\cup \hat{\mathcal{U}}_{\tau_0,\tau_1}\vspace{-6mm}
\end{equation}
\caption{\scbf{GQ graph recovery (finite sample case)}}\label{algo:fullrecoveryKn}
\end{algorithm}

\subsection{Statistical properties of the estimators}\label{sec:statprop}\less
In this section we establish the statistical properties of the estimators proposed in Section~\ref{sec:estimator}. We first specify notation and assumptions. We then state two theorems: Theorem~\ref{theo:convrate}, which establishes the rates of convergence of the \madgqlasso{}  (Equation~(\ref{eq:l1gq})) as an estimator of the \madgq{} matrix   (Equation~(\ref{eq:gq0})), and 
Theorem~\ref{theo:graphfinite}, which establishes the graph structure recovery guarantees of the GQ graph estimator $\hat{\mathcal{E}}$ produced by Algorithm~\ref{algo:fullrecoveryKn}. We further restate the results for the special case of sub-Gaussian (and Gaussian) random variables in Corollaries~\ref{coro:subgaussianrate} and \ref{coro:subgaussiangraph}.

\vspace{-1mm}

\subsubsection{Notation and assumptions}\less
Let $X^{(1)}$,..., $X^{(n)}$ be independent and identically distributed $p$-dimensional random vectors with mean vector $\mu=0$, positive definite covariance matrix $\Sigma\succ 0$, precision matrix $\Theta=\Sigma^{-1}$, and edge set $E=\{(i,j):i\neq j,\Theta_{ij}\neq 0\}$. 
Let $V^{(r)}=\{i\in V:X_i^{(r)}\text{ is observed}\}$ and $V_1,...,V_K\subset V$ be such that $O:=\cup_{r=1}^n V^{(r)}\times V^{(r)} = \cup_{k=1}^K V_k\times V_k$, with smallest possible $K$. Let $n_{ij}=|\{r:(i,j)\in V^{(r)}\times V^{(r)}\}|$ be the joint sample size for node pair $(i,j)$ and let $\bar n = \min_{(i,j)\in O}n_{ij}$ be the minimal joint sample size over the set $O$. Let $\tilde\Theta$ be the \madgq{} precision matrix in Equation~(\ref{eq:gq0}) based on $\Sigma_O=\left(\Sigma_{ij}\right)_{(i,j)\in O}$, and  $\tilde\Sigma=\tilde\Theta^{-1}$ and $\tilde E=\{(i,j):i\neq j, ~\tilde\Theta_{ij}\neq 0\}$. 
Moreover, let $\hat\Sigma_{ij}$ be the observed sample covariance of nodes $(i,j)$ (Equation~(\ref{eq:compcov})), and define the global and the local tail functions
\begin{eqnarray}\label{eq:tailmax}
\sigma(m,\varepsilon) ~&:=&~ \max\limits_{(i,j)\in O}~ \sigma_{ij}(m,\varepsilon),\\
\label{eq:tail}\sigma_{ij}(m,\varepsilon) ~&:=&~ \inf\big\{\sigma\ge 0 : P(|\hat\Sigma_{ij}-\Sigma_{ij}|>\sigma)\le\varepsilon^{-1}\big\},
\end{eqnarray}
where $\sigma_{ij}(m,\varepsilon)$ 
describes the tail behavior of $\hat\Sigma_{ij}$ with sample size $n_{ij}=m$, and $\varepsilon>0$; note that $\sigma_{ij}(m,\varepsilon)$ is nondecreasing with $\varepsilon\in (1,\infty)$. Recall $\delta=\max_{(i,j)\in O, i\neq j}|\Theta_{ij}-\tilde\Theta_{ij}|$ (Equation~(\ref{eq:delta})) and $\nu=\min_{(i,j)\in E_O} |\Theta_{ij}|$ (Equation~(\ref{eq:nu})), and define 
\begin{equation}\label{eq:psi}
    \psi~:=~\min_{(i,j,k):~0<|\tilde\Theta_{ij}^{(k)}|<\delta} \min\big(|\tilde\Theta_{ij}^{(k)}|,~~\delta
    -|\tilde\Theta_{ij}^{(k)}|\big)
\end{equation}
where $\tilde\Theta^{(k)}$ is the $k$-th \madgq{} Schur complement in Equation~(\ref{eq:madgqschur}). 
Let $d:=\max_{i\in V}|\{j:(i,j)\in O,\Theta_{ij}\neq 0\}|$ be the maximum row-degree of $\Theta_O$ (note that $d\ge 1$), and let $\tilde d$ be the maximum row-degree of $\tilde\Theta_O$. Finally, we shall say that a random variable $W$ is sub-Gaussian with sub-Gaussianity parameter $\omega>0$ if $\mathbb{E}\left[e^{tW}\right]\le e^{\mathbb{E}[W] t+\omega^2t^2/2},\forall t\in\mathbb{R}$.

Consider the following assumptions:
\begin{enumerate}
\item[(A6).] For all $i\in V$, $n\ge n_{ii}\ge\bar n >1$.\label{assumptionA6}
\item[(A7).] $|O^c| = \lceil \eta p^2\rceil $, where $\eta\in[0,1-p^{-1})$.
\item[(A8).] $\exists \alpha\in(0,1]$ such that $\max_{e\in O\cap S^c}\Vert\Gamma_{eS}(\Gamma_{SS})^{-1}\Vert_1\le 1-\alpha$, where $\Gamma=\tilde \Sigma\otimes\tilde \Sigma$ and $S=\{(i,j):\tilde\Theta_{ij}\neq 0\}$.
\item[(A9).] $\sigma_{ij}(m,\varepsilon) $ decreases with $m$.
\item[(A10).] $d,\tilde d\ge 2$.
\end{enumerate}
Assumption~(A6) guarantees that $(i,i)\in O$, for all $i\in V$, and every pair $(i,j)\in O$ has at least two joint observations to compute the empirical covariance $\hat\Sigma_{ij}$. Assumption~(A7) introduces the parameter $\eta\in[0,1)$, which measures the relative size of the set $O^c$ of unobserved pairs of nodes. We will see that, even though a smaller $\eta$ means more observed node pairs, a larger $\eta$ also implies a higher probability of concentration of $\widehat{\tilde\Theta}$ near $\tilde\Theta$, because a larger portion $\Theta_{O^c}$ is set to zero in Equation~(\ref{eq:l1gq}) and is not estimated. Assumption~(A8) is the mutual incoherence condition required in \cite{ravikumar2011high} for the convergence of the graphical lasso, except that here it is imposed on $\tilde\Theta$ rather than $\Theta$. The mutual incoherence condition limits the influence of the pairs of disconnected nodes on the pairs of connected nodes. Assumption~(A9) guarantees that the observed sample covariances concentrate around their target values as the sample size increases. Finally, Assumption (A10) guarantees that $E_O\neq\emptyset$ and $\tilde E_O\neq\emptyset$. \more

\subsubsection{Main theorems}\label{sec:rates}\less
The next theorem establishes the rate of convergence of the \madgqlasso{}  $\widehat{\tilde\Theta}$ as an estimator of the \madgq{} matrix $\tilde\Theta$ in the entrywise $\ell_\infty$-norm:\more\more

\begin{theorem}[\scbf{Convergence rate of \madgqlassobf{}}]\label{theo:convrate}
Let $\hat{\tilde\Theta}$ be the \madgqlasso{} estimator in Equation~(\ref{eq:l1gq}) with $\Lambda_{ij}=\tfrac{8}{\alpha}\sigma(\bar n,p^b)$ for all $(i,j)\in O$, where $b> 2+\tfrac{\log(1-\eta)}{\log p}$. If Assumptions (A6)--(A10) hold, then there exists $\bar n^*>1$ 
(Equation~(\ref{eq:nbarmin})) 
such that, for any $\bar n\ge\bar n^*$, with probability larger than $1-(1-\eta)p^{2-b} $ we have
\begin{equation}\label{eq:gqrateconv}
\Vert\widehat{\tilde\Theta}-\tilde\Theta\Vert_\infty  ~\le~  C\sigma(\bar n,p^b),
\end{equation}
where $\Vert *\Vert_\infty$ is the entrywise $\ell_\infty$-norm and $C$ (Equation~(\ref{eq:constconv})) depends on $\alpha$ and $\Gamma$. \more\more

\end{theorem}
Equation~(\ref{eq:gqrateconv}) specifies a hyper-cubic region centered at $\tilde\Theta$, and the estimator $\hat{\tilde\Theta}$ lies in this region with probability larger than $1-(1-\eta)p^{2-b} $. The size of this region is proportional to the global tail function $\sigma(\bar n,p^b)$ in Equation~(\ref{eq:tailmax}) so, by Assumption~(A9), it decreases with the minimal joint sample size $\bar n$, and it is nondecreasing with the number of nodes $p$ and the user-defined parameter $b$. The probability of concentration $ 1-(1-\eta)p^{2-b}$ decreases with $p$, but increases with $b$ and with $\eta\in[0,1)$, with minimum at $\eta=0$ corresponding 
to the graphical lasso with fully observed data \cite{ravikumar2011high}. This behavior is consistent with the fact that  $\widehat{\tilde\Theta}$ actually estimates only the entries of $\tilde\Theta_O$, while $\widehat{\tilde\Theta}_{O^c}\equiv0$ is trivially an exact estimate of $\tilde\Theta_{O^c}\equiv0$. Finally, note that explicit expressions of the required minimal sample size $\bar n^*$ and of the scalar $C$ are given in Equations~(\ref{eq:nbarmin}) and (\ref{eq:constconv}) in Appendix~\ref{app:proofmainsFiniteSample}, where it can be seen that $\bar n^*$ and $C$ decrease with the incoherence parameter $\alpha\in (0,1]$, and $\bar n^*$ increases with $\tilde d$. To better interpret Theorem~\ref{theo:convrate}, let us consider the special case of sub-Gaussian data:\more
\more\more

\begin{corollary}[\scbf{Convergence rate of \madgqlassobf{} (sub-Gaussian)}]\label{coro:subgaussianrate}
Under the conditions of Theorem~\ref{theo:convrate}, if, for each $i=1,...,p$, the random variable $X^{(j)}_i/\sqrt{\Sigma_{ii}}$ is sub-Gaussian with sub-Gaussianity parameter $\omega>0$, then 
for any $\bar n\ge\bar n_{\rm SG}^*:=\lceil H\tilde d^2(b\log p+\log 4) \rceil$, with probability larger than $1-(1-\eta)p^{2-b} $ we have 
\begin{equation}\label{eq:glassorate}
\Vert\widehat{\tilde\Theta}-\tilde\Theta\Vert_\infty ~\le~ C_{\rm SG} \sqrt{\tfrac{b\log p+\log 4}{\bar n}}
\end{equation}
where $H$ and $C_{\rm SG}$ (Equations~(\ref{eq:subgausssampcomplconst}) and (\ref{eq:CSG})) depend on $\alpha$, $\Gamma$, $\omega$, and $\max_i\Sigma_{ii}$. 
\end{corollary}
We can see that, if the data are sub-Gaussian, in the case where $H$ is constant with respect to $\tilde d$ and $p$, the sample complexity scales with $\tilde d$ and $p$ in the same way as the graphical lasso with full data \cite{ravikumar2011high}, while the probability of concentration is higher when $\eta>0$, as discussed above. This result, indeed, holds also for Gaussian data (case $\omega=1$). The following theorem identifies minimal sample sizes and optimal parameters for Algorithm~\ref{algo:fullrecoveryKn} to recover $E_O$ exactly and the minimal superset of $E_{O^c}$ with high probability:\more

\begin{theorem}[\scbf{GQ Graph recovery (finite samples)}]\label{theo:graphfinite}
Suppose Assumptions~(A6)--(A10) hold, and assume $\delta<\nu/2$. Let $\bar n^*$ be the minimal joint sample size required by Theorem~\ref{theo:convrate}, and let $ \hat{\mathcal{E}}$ be the output edge set of Algorithm~\ref{algo:fullrecoveryKn} with input parameters $\Lambda$ as in Theorem~\ref{theo:convrate}, and $\tau_0$, $\tau$, and $\tau_1$ as indicated below. Then, the following two results hold:
\begin{enumerate}[(i).]
\item {\sc Exact graph recovery in $O$.} If $\bar n\ge\bar n^*_O:=\max\{\bar n^*, \min\{m: C\sigma(m,p^b)< \frac{\nu}{2}-\delta\}\}$ and $\tau\in[\delta_{\bar n,p},~\nu-\delta_{\bar n,p})$, where $C$ is the scalar in Theorem~\ref{theo:convrate} and $\delta_{\bar n,p}:=\delta+C\sigma(\bar n,p^b)$, then, with probability larger than $1-(1-\eta)p^{2-b}$, we have $\hat{\mathcal{E}}_O = E_O$.
\item {\sc Minimal superset graph recovery in $O^c$.} If $\bar n\ge\bar n^{*}_{O^c} := \max\{\bar n^*, \min\{m: D\min\{\sqrt{p+s},\tilde d\}\sigma(m,p^b)< \frac{\min(\psi,\lambda_{\rm min}(\tilde\Theta))}{2}\}\}$, $\tau_0=D\min\{\sqrt{p+s},\tilde d\}\sigma(\bar n,p^b) $, and $\tau_1\in[\delta-\tau_0,~\nu-\tau_0] $, where $D$ is a scalar depending on $C$ and the condition number of $\tilde\Theta$, $s$ is the number of nonzero off-diagonals of $\tilde\Theta$, and $\psi>0$ is defined in Equation~(\ref{eq:psi}), then, under Assumptions~(A4)-(A5), with probability larger than $1-(1-\eta)p^{2-b} $, we have $\hat{\mathcal{E}}_{O^c}=\mathcal{S}_{\rm off}$, where $\mathcal{S}_{\rm off}$ is the minimal superset of $E_{O^c}$ in Equation~(\ref{eq:supersetKoff}).
\more

\end{enumerate}
\end{theorem}
We can see that the minimal sample size $\bar n_O^*$ required for the exact recovery of $E_O$ generally increases with $p$ and $\tilde d$, and decreases with the gap $\nu/2-\delta$. The minimal sample size $\bar n_{O^c}^*$ required for the recovery of the minimal superset of $E_{O^c}$ generally increases with $p$ and $\tilde d$, and decreases with $\psi$ and $\lambda_{\rm min}(\tilde\Theta)$. The optimal value of $\tau_0$ and the intervals of optimal values for $\tau$ and $\tau_1$ approach their population counterparts as $\bar n$ increases. To better interpret Theorem~\ref{theo:graphfinite}, let us consider the sub-Gaussian case:\more
\begin{corollary}[\scbf{GQ Graph recovery (finite samples, sub-Gaussian)}]\label{coro:subgaussiangraph}
Under the conditions of Corollary~\ref{coro:subgaussianrate} and Theorem~\ref{theo:graphfinite}, we have:
\begin{enumerate}[(i).]
\item The exact graph recovery in $O$ established in Theorem~\ref{theo:graphfinite}(i) holds with 
\begin{equation}
\bar n^*_O=\max\left\{\bar n_{\rm SG}^*, \left\lceil C_{\rm SG}^2\tfrac{b\log p+\log4}{(\nu/2-\delta)^2}\right\rceil \right\},~~~~~ \delta_{\bar n,p}=\delta+C_{\rm SG} \sqrt{\tfrac{b\log p+\log 4}{\bar n}},
\end{equation}
where $C_{\rm SG}$ is the scalar in Equation~(\ref{eq:glassorate}). 
\item The minimal superset graph recovery in $O^c$ established in Theorem~\ref{theo:graphfinite}(ii) holds with 
\begin{equation}
\bar n^{*}_{O^c} = \max\left\{\bar n_{\rm SG}^*, \left\lceil\tfrac{4D_{\rm SG}^2\min\{p+s,\tilde d^2\}(b\log p+\log 4)}{\min(\psi^2,\lambda_{\rm min}(\tilde\Theta)^2)}\right\rceil \right\}, ~
\tau_0 =D_{\rm SG}\sqrt{\tfrac{\min\{p+s,\tilde d^2\}(b\log p+\log 4)}{\bar n}},
\end{equation}
where $D_{\rm SG}$ (Equation~(\ref{eq:Dsg})) depends on $C_{\rm SG}$ and the condition number of $\tilde\Theta$.
\end{enumerate}
\end{corollary}
Thus, if the data are sub-Gaussian, provided that the parameters involved in the factors $H,C_{SG}$, and $D_{SG}$ are constant with respect to $\tilde d$ and $p$, then the required minimal sample size $\bar n$ for graph recovery in $O$ and in $O^c$ is proportional to $\tilde d^2 (b\log p +\log 4)$. If $\tilde d$ is a constant or $\tilde d = o\left(\sqrt{\frac{p}{\log p}}\right)$, then we can say that Graph Quilting is also possible in the ``$p\gg\bar n$'' regime.

\section{Simulations}\label{sec:simulations}\less
We now verify the statistical properties of the estimator \madgqlasso{} empirically with an extensive simulation study. In Section~\ref{sim:rates}, we verify the rate of convergence established by Theorem~\ref{theo:convrate}, and in Section~\ref{sim:graph}, we assess the graph recovery performance of the GQ graph estimator produced by Algorithm~\ref{algo:fullrecoveryKn}.
\begin{figure}[t!]
\center
\includegraphics[width=1\textwidth]{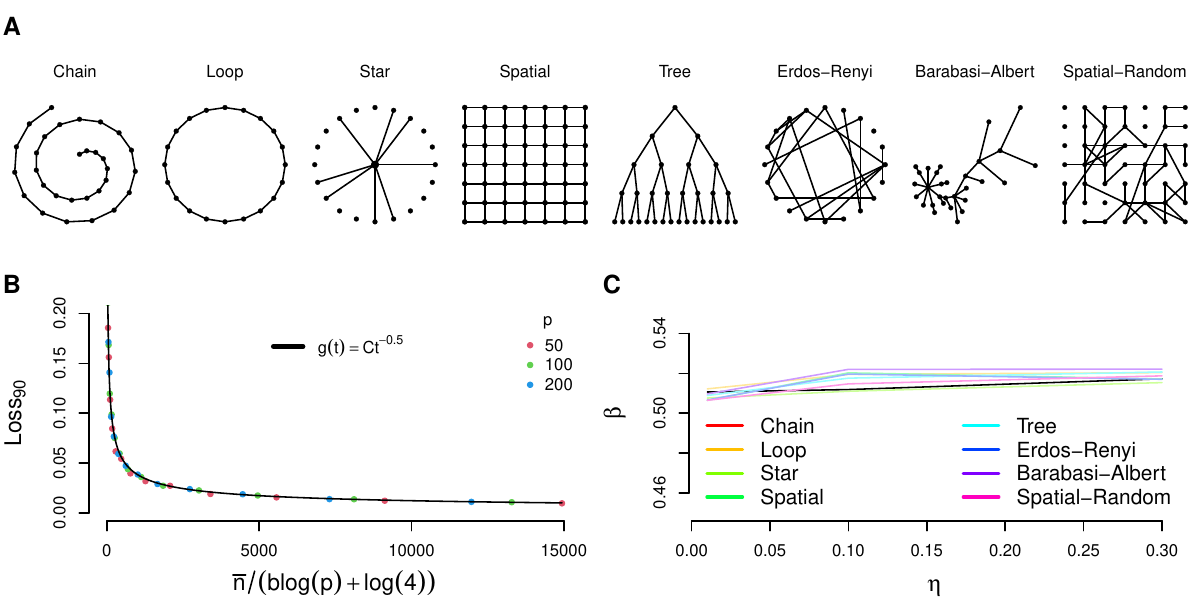}
\caption{Convergence rate of the \madgqlasso{}. {\bf (A)} Classes of graphs used in simulations. {\bf (B)} Ninetieth percentile of $\ell_\infty$ distance between \madgqlasso{} and \madgq{} ($\textsc{Loss}_{90}$) for chain graphs versus rescaled minimum sample size $\bar n/(b\log p+\log 4)$, with proportion of missingness $\eta=.1$ and $p = 50, 100, 200$. All points concentrate near the function $g(t)=Ct^{-\beta}$, with $\beta=1/2$ and some constant $C>0$, in agreement with Equation~(\ref{eq:glassorate}). {\bf (C)} Results of repeated simulation in (B) for various graph structures, where we fit the curve $g(t)=Ct^{-\beta}$ to the computed losses given different levels missingness ($\eta\in (0,0.3)$). In all cases the estimated $\beta$ is slightly larger than $1/2$.}\label{fig:sim1}
\end{figure}

\subsection{Rates of convergence}\label{sim:rates}\less
For a given $p\times p$ precision matrix $\Theta$ with graphical structure belonging to one of the classes illustrated in Figure~\ref{fig:sim1}(A) (see details in Appendix~\ref{app:simdet1}), we generate $M=50$ datasets, each one containing $n$ observations $X^{(1)},..., X^{(n)}\stackrel{\rm i.i.d.}{\sim}N(0,\Theta^{-1})$. Then, for each dataset, we retain data to reflect an observational scheme with $K=3$ subsets of nodes $V_1,V_2,V_3$ (Equation~(\ref{eq:simvk}), Appendix~\ref{app:simdet3}) with missingness proportion $\eta\in (0,1) $, and finally compute $\hat{\tilde\Theta}$ (Equation~(\ref{eq:l1gq})) with oracle penalty parameters (Appendix~\ref{app:simdet2}), and the $\ell_\infty$ distortion $u=\Vert \hat{\tilde\Theta}-\tilde\Theta\Vert_\infty$. In Figure~\ref{fig:sim1}(B) we present the results for the case of a chain graph and $\eta =0.1$. The figure shows the $90$th empirical quantile ($\textsc{Loss}_{90}$) of the computed distortions $u_1,...,u_M$ plotted versus the scaled minimum joint sample size $\bar n/(b\log p+\log 4)$, for a range of sample sizes $500\le n\le 50000$, number of nodes $p\in\{50,100, 200\}$, and $b$ such that concentration of probability (Theorem~\ref{theo:convrate}) is $1-(1-\eta)p^{2-b}=0.9$. Thus, 90\% of the computed distortions $u_1,...,u_M$ are smaller than the displayed points, and, according to Equation~(\ref{eq:glassorate}), we should expect that {\sc Loss}$_{90}\le g(t)=Ct^{-\beta}$, where $t=\bar{n}/(b\log p+\log 4)$, $\beta\approx 1/2$, and $C$ is some constant. Indeed, all displayed points concentrate around $g(t)=Ct^{-1/2}$, where $C>0$ is computed empirically. We repeat this simulation for all classes of graphs in Figure~\ref{fig:sim1}(A) and $\eta\in(0,0.3]$. In Figure~\ref{fig:sim1}(C) we plot the estimated values of $\hat\beta$ versus $\eta$. For any $\eta$, $\hat\beta$ is slightly larger than $1/2$, indeed confirming the rate of convergence in Equation~(\ref{eq:glassorate}).

\subsection{Graph recovery}\label{sim:graph}\less
We now investigate the graph recovery performance of the graph estimator $\hat{\mathcal{E}}$ produced by Algorithm~\ref{algo:fullrecoveryKn}. We consider several scenarios with number of nodes $p=50,100,200$, minimal joint sample sizes $200\le\bar n\le 10,000 $, and missingness proportions $\eta=0.1,0.2$. We quantify the graph quilting recovery performance in terms of the area under the ROC curve (AUC), summarizing the sensitivity and specificity across changes of the input parameters $\Lambda_{ij}=\lambda$ $\forall(i,j)$, $\tau$, $\tau_0$, and $ \tau_1 $ of Algorithm~\ref{algo:fullrecoveryKn}. Figure~\ref{fig:sim2} displays the AUC about the recovery of $E_O$, $E_{O^c}$, and the theoretical superset $\mathcal{S}_{\rm off}$ for an Erd\H{o}s-R\'{e}nyi graph $ER(p,\pi=p^{-1})$ (Appendix~\ref{app:simdet1}). The AUC about the recovery of $E_O$ robustly stays close to 1 for any $p$, $\bar n$, and $\eta$. The AUC about the recovery of $E_{O^c}$ and $\mathcal{S}_{\rm off}$, as expected, degrades with larger $p$ and $\eta$, but steadily increases with $\bar n$. 
\begin{figure}[t!]
\center
\includegraphics[width=1\textwidth]{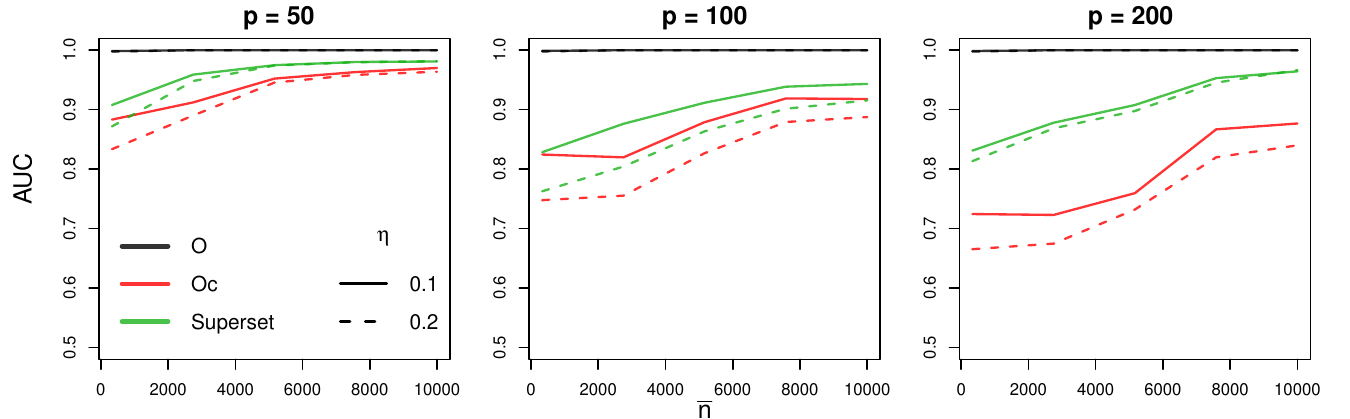}\vspace{-2mm}
\caption{Performance of the graph estimator $\hat{\mathcal{E}}$ produced by Algorithm~\ref{algo:fullrecoveryKn}. We display the AUC in $O$ (black), in $O^c$ (red), and in the superset (green) as functions of the minimal joint sample size $\bar n$ and $\eta$ for Erd\H{o}s-R\'{e}nyi graphs with $p=50, 100, 200$ nodes. The AUC in $O$ robustly stays close to 1 for any number of nodes $p$ and proportion of missingness $\eta$. The AUC in $O^c$ and in the superset degrade with $p$ and $\eta$ but steadily increases with $\bar n$.}\vspace{-3mm}
\label{fig:sim2}
\end{figure}

\section{Neuronal functional connectivity estimation from nonsimultaneous calcium imaging recordings}\label{sec:data}\less
To illustrate our methods with real data, we consider the massive publicly available data set of \cite{stringer2019spontaneous} consisting of calcium activity traces recorded from about 10,000 neurons in a cubic portion of mouse visual cortex (70--385$\mu$m depth). These neurons were simultaneously recorded  {\it in vivo} using 2--photon imaging of GCaMP6s with 2.5Hz scan rate \cite{pachitariu2017suite2p}, while the animal was free to run on an air-floating ball in total darkness for 105 minutes. 

In Figure~\ref{fig:data1}(A) we display the neurons' spatial positions occupying a 1mm $\times$ 1mm $\times$ 0.5mm 3-dimensional space, and the functional connections (see Section~\ref{sec:appliedcontexts}) recovered with the graphical lasso (Glasso) based on the full data (5,000 edges for illustration). As explained in Section~\ref{sec:appliedcontexts}, because of technology limitations, it is often preferred to record the activities of a subset of neurons at once with a finer temporal resolution rather than recording the activities of the entire neuronal population simultaneously with a coarse time resolution. This is particularly necessary when we record neuronal activities from very large numbers of neurons.  In Figure~\ref{fig:data1}(B) we illustrate a possible observational scheme where three subsets of neurons are recorded over separate experimental sessions, i.e. {\it nonsimultaneously}, generating the Graph Quilting problem with set $O$ depicted in Figure~\ref{fig:data1}(C). In Figure~\ref{fig:data1}(D) we summarize the performance of the \madgqlasso{} graph estimator $\hat{\mathcal{E}}$ (Algorithm~\ref{algo:fullrecoveryKn}) at recovering the graph that would be obtained from full data via Glasso.  We randomly select 2000 neurons and, assuming the observational scheme in (B), we drop data from the 105-minute recordings in a way that each of the three subsets of neurons is roughly recorded for 105/3 = 35 minutes. We compute the \madgqlasso{} graph estimate $\hat{\mathcal{E}}$ for different numbers of edges and proportion of missingness $\eta$ (by varying size of each neuronal subset), and assessed the similarity of the graph to Glasso (full data) in terms of area under the ROC curve (AUC) as a function of number of edges and missingness proportion $\eta$. The \madgqlasso{} estimator $\hat{\mathcal{E}}$ appears to reasonably recover similar graph structures as Glasso although, as expected, larger numbers of edges and missingness proportion negatively affect the graph quilting recovery.
\begin{figure}[t!]
\center
\includegraphics[width=1\textwidth]{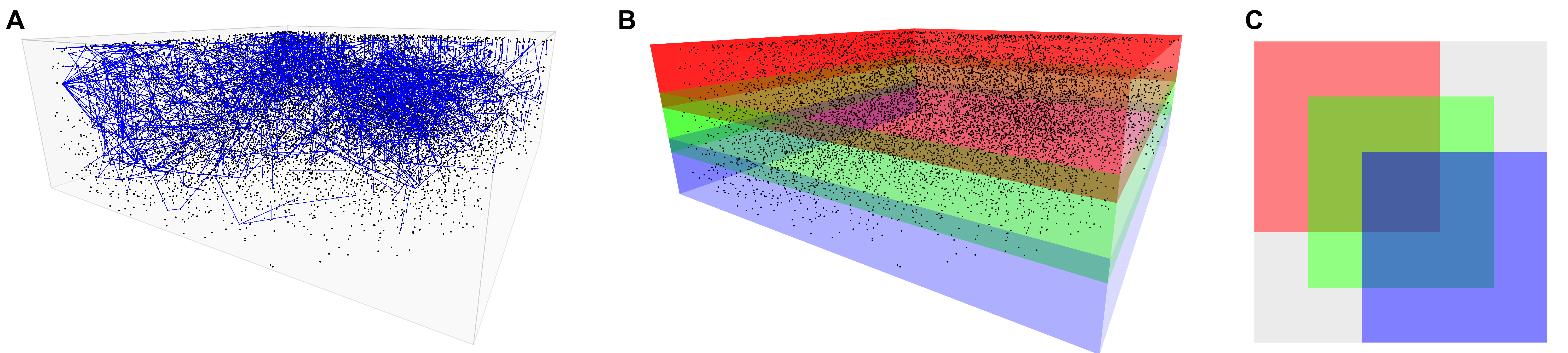}\vspace{-2mm}\\
\includegraphics[width=1\textwidth]{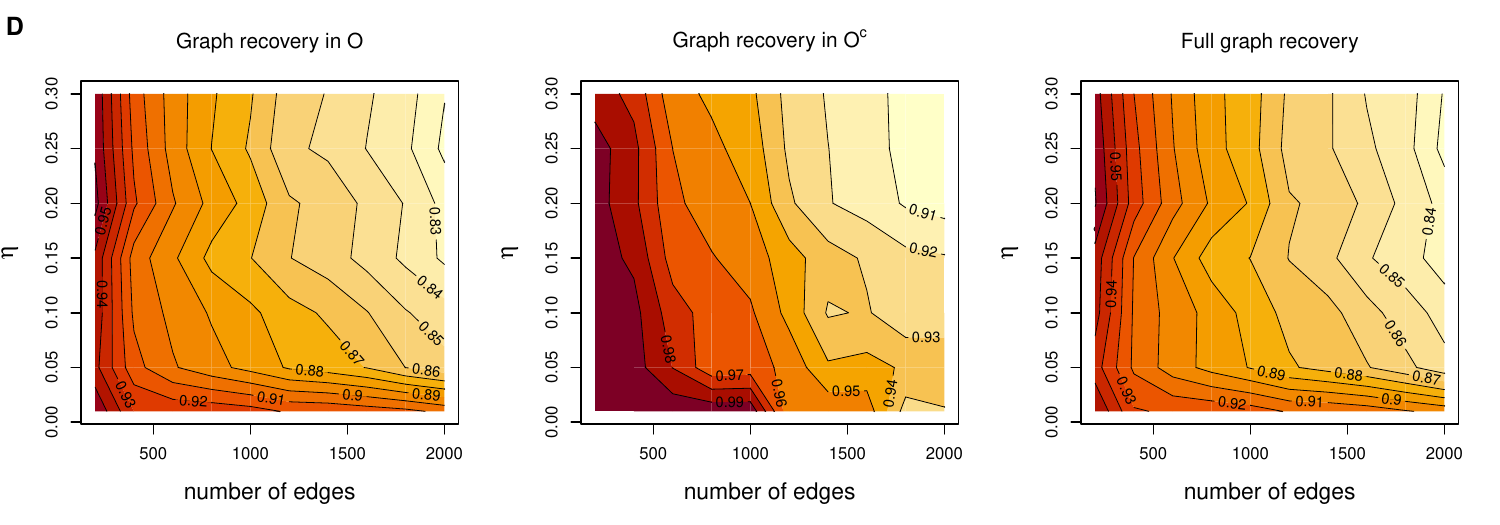}\vspace{-4mm}
\caption{(A) Brain cube functional connectivity network of 9,036 neurons in mouse visual cortex estimated from full data for a given number of edges equal to 5,000. (B) Example of nonsimultaneously recorded subsets of the brain cube. (C) Jointly observed pairs of neurons given scheme in B ($\eta\approx .2$). (D) Similarity between Glasso (full data) and \madgqlasso{} graph $\hat{\mathcal{E}}$ (Algorithm~\ref{algo:fullrecoveryKn}; incomplete data) for a random subset of 2000 neurons, in terms of area under the ROC curve (AUC) and as a function of total number of edges and missingness proportion $\eta$.}\label{fig:data1}
\end{figure}

\section{Discussion}\less
This paper has introduced a new, challenging statistical problem called {\it Graph Quilting} in which we seek to perform graphical model selection when parts of the observed covariance are completely missing, meaning many pairs of variables have no joint observations.  We characterize this new problem and introduce a simple methodological solution: partial sparse likelihood estimation via the \madgqlasso{}.  We show that our approach enjoys surprisingly strong statistical guarantees: under certain assumptions the thresholded \madgqlasso{} can perfectly recover the graph amongst the observed variable pairs; even though the graph amongst the unobserved variable pairs is not identifiable, the \madgqlasso{} plus clever use of Schur complements can recover a minimal superset of edges in this setting.  

Our work has a number of important implications.  First, we are the first to characterize a problem that seems all but impossible at first glance.  We also highlight several real-world applications of this problem and propose a simple solution.  This might inspire others to work on methodological and theoretical solutions to the Graph Quilting problem as well as apply our approach to learn graphs in neuroscience, genomics, finance, and other areas.  As we discuss and illustrate via an empirical example, Graph Quilting will be especially important for learning functional neural connectivity in large-scale calcium imaging studies with non-simultaneous recordings.  Second, our approach and the theory we develop for thresholding the \madgqlasso{} reveals new insights for graph learning with latent variables.  While \cite{wang2021thresholded} has already run with such insights based on a preprint of our work, there are likely many other fruitful directions to explore related to graph thresholding and latent variables.  Finally, our approach to learning the minimal superset of edges in $O^c$ is based on the important observation that distortions in the \madgqlasso{} result from missing members in a node's neighborhood that have other alternate connected paths through the observed part of the graph.  This insight has implications for graph learning broadly with missing and latent variables but also potentially for causal discovery in the presence of missing confounders \cite{bernstein2020ordering}.

Our work on Graph Quilting also opens the door to many possible extensions and new research directions.  First in Section~\ref{sec:gqprobl}, we highlight three possible methodological approaches for the Graph Quilting problem but only explored one of these options in this paper.  It may prove fruitful for others to explore covariance completion or observed likelihood approaches to Graph Quilting in future work. Next, the graph quilting finite sample theory may be investigated under different settings that allow for the analysis of the effects of uneven sample sizes on different parts of the graph \cite{zheng2022graphical}.  Additionally, we present an approach for learning the minimal superset of edges in $O^c$, but there are perhaps ways to leverage additional information about the graph (e.g. sparsity, hubs, cliques, motifs, graph structures)  to find the most likely set of edges in $O^c$.  This paper also focused solely on graphical model selection or structural learning, but there are possibilities of leveraging recent graph inference approaches \cite{casanellas2021robust, dasarathy2014data, dasarathy2022stochastic, katiyar2019robust, liu2013gaussian, tandon2021sga, zahin2022robust} in the context of Graph Quilting that can reflect the uncertainties associated with learning various parts of $O$ and $O^c$.  This paper also focuses on sparse inverse covariance learning, including the Gaussian Graphical Model, but one could consider the Graph Quilting problem with other types of parametric \cite{yang2015graphical} or non-parametric \cite{liu2012high} families of graphical models or even in the context of learning directed acyclic graphs.  Finally, our work considers graph learning with fixed observation sets, but one could possibly leverage our approach to adaptively learn the graph structure \cite{dasarathy2019gaussian,dasarathy2016active,eriksson2011active} by using our estimate to sequentially inform which sets of variables to measure next.  

Overall, we have proposed a completely new and challenging statistical problem we call Graph Quilting and proposed a sound methodological solution with strong theoretical guarantees.  Our work will have immediate implications for several applications, such as neuroscience and genomics, where the Graph Quilting problem naturally arises.  But, it will also inspire many possible directions for future research in graph learning.

\section*{Funding}\less
Giuseppe Vinci was supported by NSF NeuroNex-1707400, Rice Academy Postdoctoral Fellows, and Dan L. Duncan Foundation. Genevera Allen was supported by NSF NeuroNex-1707400, NIH 1R01GM140468, and
NSF DMS-2210837. Gautam Dasarathy was supported by the NIH1R01GM140468 and NSF CCF-2048223.


\bibliographystyle{imsart-nameyear}
\bibliography{BIB.bib}

\newpage

\appendix

\section{Proofs of Main Results}\label{app:proofmains}
This appendix contains the proofs of theorems, lemmas, and corollaries of Sections~\ref{sec:gqprobl}--\ref{sec:estimation}.

\subsection{Characterization of the Graph Quilting Problem}

\begin{proof}[Proof of Theorem~\ref{thm:identifiability}~ (Graph Identifiability)]
We will first prove the \emph{necessity} for $E\subseteq O$ by demonstrating a problematic example. Consider the following $3\times 3$ (valid) covariance matrices 
\begin{align}
\Sigma = \left[\begin{matrix}1 & a & a^2\\
a & 1 & a \\ 
a^2 & a & 1
\end{matrix}\right], \qquad \breve{\Sigma} = \left[\begin{matrix}1 & a^3 & a^2\\
a^3 & 1 & a \\ 
a^2 & a & 1
\end{matrix}\right], \qquad 0< \left| a \right| < 1, 
\end{align}
and suppose $O^c=\big\{(1,2),(2,1)\big\}$. That is, we suppose that we observe all covariances but for the one between the variables $X_1$ and $X_2$. Notice that $\Sigma_O = \breve{\Sigma}_O$. The inverses of these matrices are respectively given by 
\begin{align}
\Theta := \Sigma^{-1}=\left[\begin{matrix}\frac{1}{1 - a^2} & \frac{a}{a^2 - 1} & 0\\
\frac{a}{a^2 - 1} & \frac{a^2 + 1}{1 - a^2} & \frac{a}{a^2 - 1}\\
0 & \frac{a}{a^2 - 1} & \frac{1}{1 - a^2}
\end{matrix}\right], ~~~~~ \breve{\Theta} :=\breve{\Sigma}^{-1}= \left[\begin{matrix}\frac{1}{1 - a^4} & 0 & \frac{a^2}{a^4 - 1}\\
0 & \frac{1}{1 - a^2} & \frac{a}{a^2 - 1}\\
\frac{a^2}{a^4 - 1} & \frac{a}{a^2 - 1} & \frac{a^4 + a^2 + 1}{1 - a^4}
\end{matrix}\right].
\end{align}
Both these matrices correspond to different graphs that only have three edges. Thus, simply knowing the number of non-zeros in $\Theta$ is insufficient to distinguish between the two cases. Therefore, it would be impossible to identify the edge set $E$ induced by $\Theta$ from $\Sigma_O$ alone, because the value of the unobserved portion $\Sigma_{O^c}$ is pivotal to the graph. This example can of course be generalized by embedding the above matrix as a principal sub-matrix in a larger covariance matrix, or by simply considering a similar banded Toeplitz matrix as a precision matrix. 

To show the {\it sufficiency} part of the theorem, we begin by observing that a positive definite completion of $\Sigma_O$ exists since $\Sigma$ is positive definite by assumption. Thus, the max-determinant completion of $\Sigma_O$ (Equation~(\ref{eq:maxdet})) has a unique solution $\tilde\Sigma=\tilde\Theta^{-1}$ where $\tilde{\Theta}_{O^c} \equiv 0$ and $[\tilde{\Theta}^{-1}]_O = \Sigma_O$. On the other hand $[\Theta^{-1}]_O = \Sigma_O$, and since $ E\subseteq O$, we also have $\Theta_{O^c} \equiv 0$. Therefore $\tilde{\Theta} = \Theta$. 
\end{proof}

\vspace{0.5mm}

\begin{proof}[Proof of Lemma~\ref{lemma:maxdetopt}]
The solution $\tilde\Theta = \tilde\Sigma^{-1}$ to Equation~(\ref{eq:gq0}) is uniquely identified by the constraints $\det\tilde\Theta>0 \Leftrightarrow \det\tilde\Sigma>0$ and $\tilde\Theta_{O^c} = [\tilde\Sigma^{-1}]_{O^c}= 0$, and by its first order condition $[\tilde\Theta^{-1}]_{O} - \Sigma_O = 0 ~~\Leftrightarrow ~~ \tilde\Sigma_{O} = \Sigma_O$. The solution $\tilde\Sigma$ to Equation~(\ref{eq:maxdet}) is uniquely identified by the constraints $\det\tilde\Sigma>0$ and $\tilde\Sigma_O = \Sigma_O$, and by its first order condition $[\tilde \Sigma^{-1}]_{O^c} = 0$. Hence, Equations~(\ref{eq:gq0}) and (\ref{eq:maxdet}) are equivalent.
\end{proof}

\vspace{0.5mm}

\begin{proof}[Proof of Theorem~\ref{theo:nofnd} (No False Negatives in $O$)] First of all, let us define the set
 $\mathcal{S}^{++}_E:=\left\{M\succ 0:~M_{ij}\neq 0 \Leftrightarrow (i,j)\in E\right\}$ of $p\times p$ positive definite matrices with graphical structure $E$. Notice that $\mathcal{S}^{++}_E$ has dimensionality $p+|E|/2$. Let $\nu_E$ be the Lebesgue measure defined on the sigma-algebra of subsets of $\mathcal{S}^{++}_E$. For a given observed entry set $O$, let $\mathcal{N}_{hk}=\{\Theta\in\mathcal{S}^{++}_E:~\Theta_{hk}\neq 0, \tilde\Theta_{hk}=0\}$ be the set of precision matrices $\Theta$ with graphical structure $E$ and such that the \madgq{} matrix $\tilde\Theta$ obtained from $\Sigma_O=\left[\Theta^{-1}\right]_O$ is zero over the entry $(h,k)$ while $\Theta_{hk}\neq 0$. Clearly, $\mathcal{N}_{hk}=\emptyset$ if $(h,k)\notin E$ and $\mathcal{N}_{hk}\equiv\mathcal{S}^{++}_E$ for any $(h,k)\in E_{O^c}$. Since we are interested in the false negatives in $\tilde\Theta_O$, suppose $(h,k)\in E_O$, so we can rewrite $\mathcal{N}_{hk}=\{\Theta\in\mathcal{S}^{++}_E:~\tilde\Theta_{hk}=0\}$. 
Because of the constraint $ \tilde\Theta_{hk}=0$, and the fact that $\tilde\Theta$ is a 1:1 bicontinuous function (homeomorphism) of $\Sigma_O=\left[\Theta^{-1}\right]_O$  (Lemma~\ref{lemma:madcont}), the set $\mathcal{N}_{hk}$ is a low-dimensional manifold embedded in $\mathcal{S}^{++}_E$. Thus, $\nu_E(\mathcal{N}_{hk})=0$, for all $(h,k)\in E_O$, and by sub-additivity
\begin{eqnarray*}
\nu_E\left(\left\{\Theta\in\mathcal{S}^{++}_E:\exists (h,k)\in E_O, \tilde\Theta_{hk}=0\right\}\right) &=& \nu_E\left(\bigcup\limits_{(h,k)\in E_O}\mathcal{N}_{hk}\right) \le \sum\limits_{(h,k)\in E_O}\nu_E\left(\mathcal{N}_{hk}\right)=0
\end{eqnarray*}
i.e. $E_O\subseteq \tilde E_O$ almost everywhere. 
This result holds for any graphical structure $E$.
\end{proof}

\vspace{0.5mm}

\begin{proof}[Proof of Lemma~\ref{lemma:exactO} (Exact graph recovery in $O$)] 
Notice that, by definition of $\nu$ in Equation~(\ref{eq:nu}), we have $|\Theta_{ij}|\notin (0,\nu)$, for any $(i,j)\in O$, $i\neq j$. Thus, if $\delta<\frac{\nu}{2}$, we have
\begin{eqnarray}
\Theta_{ij}=0&&~\Longleftrightarrow~~-\delta\le|\tilde\Theta_{ij}|\le\delta  \\
\Theta_{ij}>0&&~\Longleftrightarrow~~\Theta_{ij}\ge\nu ~~\Longleftrightarrow~~\tilde\Theta_{ij}\ge\nu-\delta  \\
\Theta_{ij}<0 &&~\Longleftrightarrow~~\tilde\Theta_{ij}\le-(\nu-\delta) ~~\Longleftrightarrow~~  \Theta_{ij}\le-\nu
\end{eqnarray}
where $\nu-\delta>\delta$. 
Therefore, for any $\tau\in[\delta,\nu-\delta)$,
\begin{equation}
    \tilde E_O^\tau ~:=~ \{(i,j)\in O:i\neq j,|\tilde\Theta_{ij}|>\tau\} ~=~ \left\{(i,j)\in O:i\neq j,\Theta_{ij}\neq 0\right\} ~=:~E_O,
\end{equation}
and for any $(i,j)\in E_O$,
\begin{equation}
    {\rm sign}(\tilde\Theta_{ij}) ~\Longleftrightarrow~  {\rm sign}(\Theta_{ij})
\end{equation}
To demonstrate the necessity part of the theorem, simply note that $\Theta_{ij}=0 \Longrightarrow \tilde\Theta_{ij}\in [-\delta,\delta]$ and $\Theta_{ij}=\nu \Longrightarrow \tilde\Theta_{ij}\in [\nu-\delta,\nu+\delta]$, so if $\delta\ge\nu/2$ we have that in the case where
\begin{equation}
\tilde\Theta_{ij}\in [-\delta,\delta]\cap [\nu-\delta,\nu+\delta] =[\nu-\delta,\delta]\neq\emptyset
\end{equation}
it is impossible to infer whether $\Theta_{ij}= 0$ or not based on $\tilde\Theta_{ij}$, so no threshold $\tau$ would guarantee the exact identification of the true edges in $O$.
\end{proof}

\vspace{5mm}

\subsection{Graph Recovery: Population Analysis}\label{app:graphrecoverypop}

\begin{proof}[Proof of Theorem~\ref{theo:popO} (\sc{Exact Graph Recovery in $O$})] 
Lemma~\ref{lemma:madcont} shows that for a given $\Theta_O$, we have  $\delta\le \bar\delta(\gamma)$, where $\bar\delta$ is a continuous function of $\gamma=\Vert\Theta_{O^c}\Vert_\infty$. Since $\bar\delta(0)=0$, we have that for every $b>0$, there exists $c>0$ such that $\bar\delta(\gamma)<b$ for any $\gamma<c$. Thus, if we pick $c=\inf\left\{x>0:\bar\delta(x)\ge \nu/2\right\} $, $\gamma<c$ guarantees $\delta<\nu/2$, and  Lemma~\ref{lemma:exactO} applies. 
\end{proof}

\vspace{0.5mm}

\begin{proof}[Proof of Corollary~\ref{coro:popOABC}(\sc{Exact Graph Recovery in $O$ (special case $K=2$))}]
See proof of Theorem~\ref{theo:popOABCapp} in Appendix~\ref{app:abcproofs}.
\end{proof}

\vspace{0.5mm}

\begin{proof}[Proof of Corollary~\ref{coro:poplvgm} (Latent variable graphical model)] 
The result is an application of Corollary~\ref{coro:popOABC} limited to the portion $A\times A$, where the condition $\delta<\nu/2$ is guaranteed by $\gamma<\sqrt{\frac{\nu\lambda_{\rm min}}{2d_{O^c}^2}}$ (Assumption (A2), with $B=\emptyset$).
\end{proof}

\begin{proof}[Proof of Lemma~\ref{lemma:schurgen} (\madgq{} Entanglement)]
Let $\Theta\in\mathcal{S}^{++}_p $ be a positive definite matrix with inverse $\Sigma=\Theta^{-1}$. For any node set $U\subset V=\{1,...,p\}$, the Schur complement of the block $\Theta_{U^cU^c}$ is given by 
$
\Sigma_{UU}^{-1}=\Theta_{UU}-\Theta_{UU^c}\Theta_{U^cU^c}^{-1}\Theta_{U^cU}
$. 
If $\tilde\Theta\in\mathcal{S}^{++}_p $ has inverse $\tilde\Sigma=\tilde\Theta^{-1}$ such that $\Sigma_O=\tilde\Sigma_O$, then for any $U\subset V$ such that $U\times U\subseteq O$, we have $\Sigma_{UU}^{-1}=\tilde\Sigma_{UU}^{-1}$ and so
$
\Theta_{UU}-\Theta_{UU^c}\Theta_{U^cU^c}^{-1}\Theta_{U^cU}=\tilde\Theta_{UU}-\tilde\Theta_{UU^c}\tilde\Theta_{U^cU^c}^{-1}\tilde\Theta_{U^cU}
$.
\end{proof}

\vspace{0.5mm}

\begin{proof}[Proof of Theorem~\ref{theo:distprK} (Distortion Propagation)] First, note that $V_k\times V_k\subseteq O$, so by Lemma~\ref{lemma:schurgen}, 
\[
\Theta_{V_kV_k}-\Theta_{V_kV_k^c}\Theta_{V_k^cV_k^c}^{-1}\Theta_{V_k^cV_k} ~=~ \Sigma_{V_kV_k}^{-1} ~=~  \tilde\Theta_{V_kV_k}-\tilde\Theta_{V_kV_k^c}\tilde\Theta_{V_k^cV_k^c}^{-1}\tilde\Theta_{V_k^cV_k}.
\]
Thus, for $(i,j)\in V_k\times V_k$,
\[
\Theta_{ij}-\Theta_{iV_k^c}\Theta_{V_k^cV_k^c}^{-1}\Theta_{V_k^cj} ~=~ \left[\Sigma_{V_kV_k}^{-1}\right]_{ij} ~=~  \tilde\Theta_{ij}-\tilde\Theta_{iV_k^c}\tilde\Theta_{V_k^cV_k^c}^{-1}\tilde\Theta_{V_k^cj}~=:~\tilde\Theta^{(k)}_{ij}
\]
and $\delta^{(k)}_{ij}=\Theta_{ij}-\tilde\Theta^{(k)}_{ij}=\Theta_{iV_k^c}\Theta_{V_k^cV_k^c}^{-1}\Theta_{V_k^cj}$. 
We now prove each part of the theorem.
\paragraph*{(i)} For any $k$ such that $i\in V_k$, we have $H_i\subseteq V_k^c$. Thus, if $\Theta_{iH_i}\neq 0$, then $\Theta_{iV_k^c}\neq 0 $ and $\delta^{(k)}_{ii}=\Theta_{iV_k^c}\Theta_{V_k^cV_k^c}^{-1}\Theta_{V_k^ci}>0 $ since $\Theta_{V_k^cV_k^c}^{-1}\succ 0$. If $V_k^c=H_i$, then $\Theta_{iH_i}=\Theta_{iV_k^c}$, so $\delta^{(k)}_{ii}>0$ if and only if $\Theta_{iH_i}\neq 0$.

\paragraph*{(ii)} First of all, notice that the distortion $\delta_{ij}^{(k)}$ can be written as
\begin{equation}\label{eq:deltaindecomp}
\delta^{(k)}_{ij} ~=~ \Theta_{iV_k^c}\Theta_{V_k^cV_k^c}^{-1}\Theta_{V_k^cj} ~~=\sum_{h\in V_k^c}\sum_{l\in V_k^c}\Theta_{ih}\left[\Theta_{V_k^cV_k^c}^{-1} \right]_{hl}\Theta_{lj}
\end{equation}
This expression shows that $\delta_{ij}^{(k)}\neq 0$ if at least one of the addends of the sum is nonzero, except in cases where more than one addend is nonzero and cancellations happen; these cases, however, form a zero volume set in the space of precision matrices with a given graphical structure $E$, making our statements below  hold almost everywhere. We can see that $\Theta_{ih}\neq 0$ if $h\in N_{H_i}(i)$, and $\Theta_{lj}\neq 0$ if $l\in N_{H_i}(j)$, where $N_{H_i}(i)\subseteq N_{V_k^c}(i)\subseteq V_k^c$ and $N_{H_i}(j)\subseteq N_{V_k^c}(j)\subseteq V_k^c$. Moreover, if $h=l$, then $\left[\Theta_{V_k^cV_k^c}^{-1} \right]_{hl}\neq 0$ because the matrix $\Theta_{V_k^cV_k^c}^{-1}$ is positive definite. Alternatively, if $h$ and $l$ are $H_i$-connected, then they are also $V_k^c$-connected since $H_i\subseteq V_k^c$, so $\left[\Theta_{V_k^cV_k^c}^{-1} \right]_{hl}\neq 0$ almost everywhere. Therefore, if $\exists h\in N_{H_i}(i)$ and $\exists l\in N_{H_i}(j)$ such that $h=l$ or $h$ is $H_i$-connected to $l$, then $\delta_{ij}^{(k)}\neq 0$ almost everywhere.

On the other hand, if $\delta_{ij}^{(k)}\neq 0$, then at least one addend in the sum in Equation~(\ref{eq:deltaindecomp}) must be nonzero, i.e. there must be some pair $(h,l)\in V_k^c\times V_k^c$ such that $\Theta_{ih}\left[\Theta_{V_k^cV_k^c}^{-1} \right]_{hl}\Theta_{lj}\neq 0$ or, equivalently, $\Theta_{ih}\neq 0$, $\Theta_{lj}\neq 0$, and $\left[\Theta_{V_k^cV_k^c}^{-1} \right]_{hl}\neq 0$. Therefore, in the case where $H_i=V_k^c$,  we have that if $\delta_{ij}^{(k)}\neq 0$ then $\exists h\in N_{H_i}(i)$ and $\exists l\in N_{H_i}(j)$ such that $h=l$ or $h$ is $H_i$-connected to $l$. Note that this result actually holds for all positive definite precision matrices, not just almost everywhere.

\paragraph*{(iii)} If $\delta^{(k)}_{ij} ~=~ \Theta_{iV_k^c}\Theta_{V_k^cV_k^c}^{-1}\Theta_{V_k^cj}\neq 0$, then we must have $\Theta_{iV_k^c}\neq 0$ and $\Theta_{V_k^cj}\neq 0$, which implies $\delta_{ii}^{(k)}=\Theta_{iV_k^c}\Theta_{V_k^cV_k^c}^{-1}\Theta_{V_k^ci}>0$ and $\delta_{jj}^{(k)}=\Theta_{jV_k^c}\Theta_{V_k^cV_k^c}^{-1}\Theta_{V_k^cj}>0$ since $\Theta_{V_k^cV_k^c}^{-1}\succ 0$.
\end{proof}

\more

\begin{proof}[Proof of Corollary~\ref{coro:distpropagK}]
Let $(i,j)\in O^c$ where, without loss of generality,
$i<j$. \less
\paragraph*{(i)} If $\Theta_{ij}\neq 0$, then  $\Theta_{iH_i}\neq 0$ and $\Theta_{jH_j}\neq 0$, which imply, respectively, $\delta^{(k)}_{ii}\neq 0$ and $\delta^{(h)}_{jj}\neq 0$ for all $k,h$ such that $i\in V_k$ and $j\in V_h$, as per Theorem~\ref{theo:distprK}, part {\it (i)}.
\paragraph*{(ii)} If $\Theta_{ij}\neq 0$, then $j\in N_{H_i}(i)\neq\emptyset$, and if $s\in V_k\setminus \{i\}$ is $(H_i\cup \{s\})$-connected to $j$, then $\exists l\in N_{H_i}(s)\neq\emptyset$ such that either $l=j$ or $l$ is $H_i$-connected to $j$. Thus, by Theorem~\ref{theo:distprK}, part {\it (ii)}, we have $\delta^{(k)}_{is}\neq 0$ almost everywhere.  
\end{proof}

\more

\begin{proof}[Proof of Theorem~\ref{theo:OcK} (\sc{Oracle Minimal Superset of $E_{O^c}$})]~ \less
\paragraph*{(i)} First, we prove that $\mathcal{S}_{\rm diag}$ in Equation~(\ref{eq:supersetK}) enjoys property (i) of a minimal superset (Definition~\ref{def:minimalityK}): $\forall\Sigma'\in\mathcal{A}(\Sigma,O,Q)$ we have $E_{O^c}'\subseteq\mathcal{S}_{\rm diag}$, where $Q$ is the set of diagonals $\{(i,i,k):i\in V,k=1,...,K\}$. We prove this by contradiction. Suppose there exists $\Sigma'\in\mathcal{A}(\Sigma,O,Q)$ such that $E_{O^c}'\cap \mathcal{S}^c_{\rm diag} \neq \emptyset$,  
i.e. such matrix $\Sigma'$ induces one or more edges in $O^c$ and outside of $\mathcal{S}_{\rm diag}$. So, if $\exists(h,l)\in E'_{O^c}\cap\mathcal{S}_{\rm diag}^c$, where $h<l$, then we must have $h\notin \{i\in V: (i,j)\in\mathcal{S}_{\rm diag},i<j\}$ and/or $l\notin \{i\in V: (i,j)\in\mathcal{S}_{\rm diag},i>j\}$. Without loss of generality, suppose it is the case where $h\notin \{i\in V: (i,j)\in\mathcal{S}_{\rm diag},i<j\}$. However, since $(h,l)\in E_{O^c}'$, then Theorem~\ref{theo:distprK} and Corollary~\ref{coro:distpropagK} guarantee that there must be a distortion in the diagonal entry relative to node $h$ of all \madgq{} Schur complements involving node $h$, i.e. $\delta_{hh}^{(k)}> 0$ for all $k$ such that $i\in V_k$. Thus $h$ must be in the node set $\{i\in V: (i,j)\in\mathcal{S}_{\rm diag},i<j\}$, which is a contradiction.

We now prove that $\mathcal{S}_{\rm diag}$ enjoys property (ii) of a minimal superset (Definition~\ref{def:minimalityK}): $\forall\mathcal{S}' \subsetneq \mathcal{S}_{\rm diag}$, $\exists\Sigma'\in\mathcal{A}(\Sigma,O,Q)$ such that $E'_{O^c} \cap (\mathcal{S}_{\rm diag}\setminus \mathcal{S}') \neq \emptyset$. Let us first consider the case of $K=2$ observed blocks. For simplicity, let $V_1=A\cup B$ and $V_2=B\cup C$, where $A,B,C$ is a partition of $V$. In this special case the superset has the form $\mathcal{S}_{\rm diag}=(\tilde A\times \tilde C)\cup (\tilde C\times \tilde A)$, where $\tilde A\subseteq A$ and $\tilde C\subseteq C$. Consider the following optimization problem
\begin{equation*}
    T(S,\tau) ~=~ \underset{T\succ 0, ~T_{O^c}=Z_\tau}{\arg\max} ~\log\det  T - \sum_{(i,j)\in O}T_{ij}\Sigma_{ij} \end{equation*}
where $Z_\tau$ is an entry set that is zero everywhere except over the symmetric set $S\subseteq O^c$ where all entries have value $\tau\in\mathbb{R}$. We have that $T(S,0)$ equals the \madgq{} matrix $\tilde\Theta$, which is guaranteed to exist (Lemma~\ref{lemma:maxdetopt}). For any $S\subseteq O^c$ and a sufficiently small $|\tau|\neq 0$, also the solution $T(S,\tau)$ exists and is uniquely identified by the constraints $\det T>0$ and $T_{O^c} = Z_\tau$, and by the first order condition  $[T^{-1}]_{O} = \Sigma_O $. Thus, the precision matrix $\Theta'=T(\mathcal{S}_{\rm diag},\tau)$ satisfies  $[\Theta^{'-1}]_O = \Sigma_O$ and $E'_{O^c}=\mathcal{S}_{\rm diag}$, and thereby  $\mathcal{D_{\rm diag}}(\Sigma',O)=\mathcal{D}_{\rm diag}(\Sigma,O)$ by Theorem~\ref{theo:distprK}, since in this case $K=2$ we have a distortion $\tilde\Theta'_{ii}\neq\Theta'_{ii}$ if and only if $\Theta'_{iC}\neq 0$ or $\Theta'_{Ai}\neq 0$ (indeed, we have $H_i=V_1^c$ for all $i\in A$ and $H_i=V_2^c$ for all $i\in C$, so the condition in Equation~(\ref{eq:deltaii}) is sufficient and necessary). Hence, $\Sigma'\in\mathcal{A}(\Sigma,O,Q)$. This shows that the case where $E_{O^c}=\mathcal{S}_{\rm diag}$ is possible, that is, there is no set $\mathcal{S}'\subsetneq \mathcal{S}_{\rm diag}$ that could contain all plausible edge sets in $O^c$.  

In the case with $K>2$ observed sets of nodes, we can either be in the case where the distorted Schur complement diagonals all happen in correspondence of edges in $O^c$ (i.e. only the diagonals of nodes that are incident with some edge in $O^c$ are distorted), or in the case where there is at least one $i$ such that $\delta^{(k)}_{ii}>0$ is only due to edges in $O$, specifically edges in the set $(\{i\}\times V_k^c)\cap O$ for all $k$ such that $i\in V_k$. In the first case, every row and every column of $\mathcal{S}_{\rm diag}$ contains at least one true edge. In the second case, some rows or columns of $\mathcal{S}_{\rm diag} $ may contain no true edge. However, even if we knew that some of the node pairs in  $(\{i\}\times V_k^c)\cap O$ were connected, we would still not be able to exclude the possibility of the existence of edges in the portion $\{i\}\times H_i$. Hence, the two situations are in general indistinguishable based on the diagonal distortions $\mathcal{D}_{\rm diag}(\Sigma,O)$, and it is possible to find families of matrices $\Sigma\succ 0$ such that $\Sigma'\in\mathcal{A}(\Sigma,O,Q)$ have $E'_{O^c}=\mathcal{S}_{\rm diag}$ and $\mathcal{D}_{\rm diag}(\Sigma',O)=\mathcal{D}_{\rm diag}(\Sigma,O)$. Hence, it is not possible to find a superset that is smaller than $\mathcal{S}_{\rm diag}$ in all situations. 

\paragraph*{(ii)}  By the Distortion Propagation Theorem~\ref{theo:distprK}, part~(iii), we have $\delta^{(k)}_{i(-i)}\neq 0\Rightarrow\delta^{(k)}_{ii}>0$. 
Thus, the node set $D_{\rm off}$ in Equation~(\ref{eq:supersetKoff}) is a subset of the node set $D_{\rm diag}$ in Equation~(\ref{eq:supersetK}), and so $\mathcal{S}_{\rm off}\subseteq\mathcal{S}_{\rm diag}$. Moreover, by Corollary~\ref{coro:distpropagK} part~(ii), Assumption~(A4) guarantees [a.e.] that for every node $i$ that is incident to an edge in $E_{O^c}$, we have at least one off-diagonal distortion in every \madgq{} Schur Complement $\tilde\Theta^{(k)}$ on the row of node $i$, provided that $i\in V_k$. Thus, the node set $D_{\rm off}$ in Equation~(\ref{eq:supersetKoff}) includes all nodes that are incident to some edge in $E_{O^c}$, guaranteeing that $E_{O^c}\subseteq O^c\cap(D_{\rm off}\times D_{\rm off})=:\mathcal{S}_{\rm off}$ (Definition~\ref{def:minimalityK}, property (i)). If Assumption~(A4) does not hold, then the set $\mathcal{S}_{\rm off}$ may miss true edges because $D_{\rm off}$ may miss nodes that are incident to edges in $E_{O^c}$. The proof of $\mathcal{S}_{\rm off}$ being minimal is analogous to the proof for $\mathcal{S}_{\rm diag}$. Specifically, even if we knew that some of the node pairs in  $(\{i\}\times V_k^c)\cap O$ were connected, we would still not be able to exclude the possibility of the existence of edges in the portion $\{i\}\times H_i$ in the case where $i\in D_{\rm off}$. Therefore, based on the off-diagonal distortions $\mathcal{D}_{\rm off}(\Sigma,O)$, it is not possible to find a superset that is smaller than $\mathcal{S}_{\rm off}$ in all situations (Definition~\ref{def:minimalityK}, property (ii)). 
\end{proof}

\vspace{0.5mm}

\begin{proof}[Proof of Theorem~\ref{theo:fullrecoveryK} (\sc{GQ Graph recovery (population case)})] 
If $\delta<\nu/2$, then for any $\tau\in[\delta,\nu-\delta)$ the thresholded edge set $\mathcal{E}_O^\tau\equiv \tilde E_O^\tau$ in Equation~(\ref{eq:fulledegqpop}) equals the true edge set $E_O$, as per Theorem~\ref{theo:popO}. 
Moreover, notice that, by definition of $\nu$ (Equation~(\ref{eq:nu})), no off-diagonal entry of $\Theta_O$ may have magnitude in the interval $(0,\nu)$, so if $0<|\tilde\Theta^{(k)}_{ij}|<\nu$, then we must have $\delta^{(k)}_{ij}:=\Theta_{ij}-\tilde\Theta_{ij}^{(k)}\neq 0$. Thus, under Assumption~(A5), for any $\tau\in[\delta,\nu]$, the set $W_\tau$ defined in Algorithm~\ref{algo:fullrecoveryK},
\begin{equation*}
W_\tau=\left\{i\in V:~ \forall k ~s.t.~i\in V_k,~\exists j\neq i, 0<|\tilde\Theta^{(k)}_{ij}|< \tau\right\}
\end{equation*}
equals the set $D_{\rm off}$ in Equation~(\ref{eq:supersetKoff}). 
Therefore, for any $\tau\in[\delta,\nu-\delta)$, $\mathcal{E}^\tau_{O^c}=\mathcal{S}_{\rm off}$ where, under Assumption~(A4), $\mathcal{S}_{\rm off}$ is the minimal superset of $E_{O^c}$ [a.e.] based on oracle off-diagonal distortions, as per  Theorem~\ref{theo:OcK} part (ii).
\end{proof}

\vspace{5mm}

\subsection{Graph Recovery: Finite Sample Analysis}\label{app:proofmainsFiniteSample}

\begin{proof}[Proof of Lemma~\ref{lemma:gqlassoexun}]
The maximum in Equation~(\ref{eq:l1gq}), if it exists, is unique because the objective function is strictly concave as its log-determinant component is strictly concave. To see that the maximum is achieved, first note that, for $\lambda_{ij} > 0, \forall (i,j)\in O$, Lagrangian duality lets us rewrite Equation~(\ref{eq:l1gq}) as
\[
\widehat{\tilde\Theta} ~=~ \underset{\Theta\succ 0,\Theta_{O^c}=0,\Vert \Lambda\odot \Theta\Vert_{1,\rm off}\le C(\Lambda)}{\arg\max}~~ \log\det\Theta - \sum_{(i,j)\in O}\Theta_{ij}\hat\Sigma_{ij},
\]
for some scalar $C(\Lambda) < \infty$ depending on the penalty matrix $\Lambda$. This representation shows that the off-diagonal elements of $\Theta$ are bounded within the weighted $\ell_1$-ball, meaning that only the diagonals $\{\Theta_{ii}\}_{i\in V}$ might potentially diverge to infinity. We have
\begin{eqnarray*}
\log\det\Theta - \sum_{(i,j)\in O}\Theta_{ij}\hat\Sigma_{ij}&=&\log\det\Theta-\sum_{i\in V}\Theta_{ii}\hat\Sigma_{ii}- \sum_{(i,j)\in O,i\neq j}\Theta_{ij}\hat\Sigma_{ij}\\
&\le&\log\det\Theta-\sum_{i\in V}\Theta_{ii}\hat\Sigma_{ii}+{\rm const}\\
&\le &\sum_{i\in V}\left(\log\Theta_{ii} -\Theta_{ii}\hat\Sigma_{ii}\right)+{\rm const}\\ & =: &-h(\Theta_{11},...,\Theta_{pp}),
\end{eqnarray*}
where the first inequality holds because the off-diagonals of $\Theta$ are bounded in the $\ell_1$-ball and $\Vert\hat\Sigma_O\Vert_\infty<\infty$ so we can find a finite constant ${\rm const}<\infty$ such that $- \sum_{(i,j)\in O,i\neq j}\Theta_{ij}\hat\Sigma_{ij}\le {\rm const}$, while the second inequality is an application of Hadamard's Inequality  $\det\Theta\le\prod_{i\in V}\Theta_{ii} $ for positive definite matrices. The function $h$ is a coercive function of $\{\Theta_{ii}\}_{i\in V}$ since it diverges to $\infty$ for any sequence $\Vert(\Theta_{11}^t,...., \Theta_{pp}^t)\Vert_2\to +\infty$, as long as $\hat\Sigma_{ii}>0, \forall i\in V$. Therefore, the objective function $\log\det\Theta - \sum_{(i,j)\in O}\Theta_{ij}\hat\Sigma_{ij}$ may only diverge to $-\infty$, so the maximum of the objective function in Equation~(\ref{eq:l1gq}) is attained at some $\hat{\tilde\Theta}$ with bounded diagonals. Finally, a matrix $\hat{\tilde\Theta}\succ 0$ is a solution of Equation~(\ref{eq:l1gq}) if and only if there exists a matrix $ Z$ that is in the sub-differential $\partial\Vert * \Vert_{1,{\rm off}}$ evaluated at $\hat{\tilde\Theta}$ and that satisfies the first order condition $\hat\Sigma_O -[\hat\Theta^{-1}]_O +\Lambda_O\odot Z_O = 0 $.
\end{proof}

\vspace{0.5mm}

\begin{proof}[Proof of Theorem~\ref{theo:convrate} (Convergence rate of \madgqlasso{})] 
This proof follows the strategy of \cite{ravikumar2011high} for the derivation of the rates of convergence of the graphical lasso based on fully observed data. However, our proof differs in several aspects, because our estimator $\hat{\tilde\Theta}$ is an estimator of $\tilde\Theta$, not $\Theta$, and it is based on an incomplete set of empirical covariances. Moreover, our probability of concentration involves the parameter $\eta$, which measures the relative size of the set $O^c$ of node pairs that are never observed jointly. 

Let $S = \{(i,j) : \tilde\Theta_{ij}\neq 0\}$ be the edge set (with self-loops) induced by the \madgq{} matrix $\tilde\Theta$ in Equation~(\ref{eq:gq0}), and consider the optimization problem
 \begin{equation}\label{eq:l1gqS}
\hat{\tilde\Theta}^S ~:=~ \underset{\Theta\succ 0,\Theta_{S^c}=0}{\arg\max}~~ \log\det\Theta - \sum_{(i,j)\in O}\Theta_{ij}\hat\Sigma_{ij}  ~-~\Vert \Lambda\odot \Theta\Vert_{1,\rm off}
\end{equation}
which is a more constrained version of the \madgqlasso{} in Equation~(\ref{eq:l1gq}) where the constraint $\Theta_{O^c}=0$ is replaced by $\Theta_{S^c}=0$. Indeed, since $\tilde\Theta_{O^c}=0$, we have $O^c\cap S=\emptyset$, so $S\subseteq O$ and  $O^c\subseteq S^c$. Thus the edge set induced by the solution $\hat{\tilde\Theta}^S$ is guaranteed to contain no false positive edge as an estimator of the edge set induced by the \madgq{} matrix $\tilde\Theta$. Moreover, define the following quantities:
\begin{eqnarray}
W_O &:=& \hat\Sigma_O-\Sigma_O\label{eq:WO} \\
\Delta &:=& \hat{\tilde\Theta}^S - \tilde\Theta\\
\tilde\Sigma &:=& \tilde\Theta^{-1}\\
\lambda_{p,\bar n} &:=& \tfrac{8}{\alpha}\sigma(\bar n,p^b)\label{eq:lambdaopt}\\
R(\Delta) &:=& \left(\hat{\tilde\Theta}^S\right)^{-1} - \tilde\Sigma +  \tilde\Sigma\Delta\tilde\Sigma\\
\Gamma &:=& \tilde\Sigma\otimes\tilde\Sigma\\
\kappa_{\Sigma} &:=& |||\tilde\Sigma|||_\infty=\max_j\sum_{k=1}^p|\tilde\Sigma_{jk}|\\
\kappa_{\Gamma} &:=& |||(\Gamma_{SS})^{-1}|||_\infty\\
\bar n^* &:=& \min\left\{n: ~\sigma(n,p^b)\le \left[2(1+\tfrac{8}{\alpha})^{2}3\tilde d \max\{\kappa_{\Sigma}\kappa_{\Gamma}, \kappa^3_{\Sigma}\kappa^2_{\Gamma}\}\right]^{-1}\right\}\label{eq:nbarmin}
\end{eqnarray}
where $\alpha$ is the parameter defined in Assumption~(A8), $\sigma(m,\varepsilon)$ is the global tail function defined in Equation~(\ref{eq:tailmax}), and $\otimes$ denotes the Kronecker matrix product. Furthermore, recall that $\tilde d$ is the max row degree of $\tilde\Theta$, and $b$ is some user-defined scalar. We require $b>2+\tfrac{\log(1-\eta)}{\log p}$ to ensure $1-(1-\eta)p^{2-b}>0$; moreover, by Assumption~(A7) we have $0\le\eta<1-p^{-1}$, which implies $1<2+\tfrac{\log(1-\eta)}{\log p}\le 2$, thereby guaranteeing $b$ to be positive. We now show that if Assumptions (A6)--(A10) hold, and $\Lambda_{ij}=\lambda_{p,\bar n}$ for all $(i,j)\in O$ and $\bar n\ge\bar n^*$ then, with probability larger than $1-(1-\eta)p^{2-b}$, we have\vspace{-1mm}
\begin{equation}\label{ineq:wo}
\Vert W_O \Vert_\infty ~\le~ \sigma(\bar n,p^b)
\end{equation}
and that this implies $\Vert\widehat{\tilde\Theta}-\tilde\Theta\Vert_\infty  \le  C\sigma(\bar n,p^b)
$, for some scalar $C$ that depends on $\alpha$ and $\Gamma$. 
We present this proof in four parts.

~

\paragraph*{\sc Part 1} Suppose that the inequality in Equation~(\ref{ineq:wo}) holds (we will show this happens with probability larger than $1-(1-\eta)p^{2-b}$ in Part 4). Thus
\begin{eqnarray*}
\Vert W_O \Vert_\infty + \lambda_{p,\bar n} &\le& \sigma(\bar n,p^b) + \tfrac{8}{\alpha}\sigma(\bar n,p^b)~~~~~~{ \text{[by Equations~(\ref{eq:lambdaopt}) and (\ref{ineq:wo})]}} \\
(*)~~~~& = & (1+\tfrac{8}{\alpha})\sigma(\bar n,p^b)\\
& \le &  (1+\tfrac{8}{\alpha})\left[2(1+\tfrac{8}{\alpha})^{2}3\tilde d \max\{\kappa_{\Sigma}\kappa_{\Gamma}, \kappa^3_{\Sigma}\kappa^2_{\Gamma}\}\right]^{-1}~~~~~~{ \text{[because $\bar n\ge\bar n^*$]}}\\
& \le &  (2\kappa_{\Gamma})^{-1}\min\{(3\kappa_{\Sigma}\tilde d)^{-1}, (3\kappa^3_{\Sigma}\kappa_{\Gamma}\tilde d)^{-1}\}\underbrace{(1+\tfrac{8}{\alpha})^{-1}}_{\le 1}\\
& \le &  (2\kappa_{\Gamma})^{-1}\min\{(3\kappa_{\Sigma}\tilde d)^{-1}, (3\kappa^3_{\Sigma}\kappa_{\Gamma}\tilde d)^{-1}\}\\
(**)~~~& \le &  (2\kappa_{\Gamma})^{-1} (3\kappa_{\Sigma}\tilde d)^{-1}
\end{eqnarray*}
Then, Lemma~\ref{convlemma}~(iii)  implies \begin{equation}\label{ineq:Delta}
\Vert\Delta\Vert_\infty \le 2\kappa_{\Gamma} (\Vert W_O\Vert_\infty +\lambda_{p,\bar n})
\end{equation}
Consequently two useful inequalities may be established: by combining (\ref{ineq:Delta}) with (*) and with (**), respectively, we obtain the two bounds
\begin{eqnarray}
\Vert \Delta \Vert_\infty &\le& 2\kappa_{\Gamma}(1+\tfrac{8}{\alpha}) \sigma(\bar n,p^b)\label{ineq:delta2}\\
\Vert \Delta \Vert_\infty &\le&  (3\kappa_{\Sigma}\tilde d)^{-1}\label{ineq:delta1}
\end{eqnarray}

~

~

\paragraph*{\sc Part 2} The inequality in Equation~(\ref{ineq:delta1}) lets Lemma~\ref{convlemma}~(ii) hold, so that
\begin{eqnarray*}
\Vert R(\Delta)_O\Vert_\infty 
& \le &  \tfrac{3}{2}\tilde d\Vert\Delta\Vert^2_\infty\kappa^3_{\Sigma}\\
& \le & 6\kappa^3_{\Sigma}\kappa^2_{\Gamma}\tilde d(1+\tfrac{8}{\alpha})^2 [\sigma(\bar n, p^b)]^2~~~~~\text{[by Equation~(\ref{ineq:delta2})]}\\
& = & \left\{6\kappa^3_{\Sigma}\kappa^2_{\Gamma}\tilde d(1+\tfrac{8}{\alpha})^2 \sigma(\bar n, p^b)\right\} \tfrac{\alpha}{8}\lambda_{p,\bar n} ~~~~\text{[by Equation~(\ref{eq:lambdaopt})]} \\
& \le &  \left\{6\kappa^3_{\Sigma}\kappa^2_{\Gamma}\tilde d(1+\tfrac{8}{\alpha})^2 \left[2(1+\tfrac{8}{\alpha})^{2} 3\tilde d \max\{\kappa_{\Sigma}\kappa_{\Gamma}, \kappa^3_{\Sigma}\kappa^2_{\Gamma}\}\right]^{-1}\right\} \tfrac{\alpha}{8}\lambda_{p,\bar n}  ~{ \text{[since $\bar n\ge\bar n^*$]}}\\
& = &  \left\{\kappa^3_{\Sigma}\kappa^2_{\Gamma} \min\{(\kappa_{\Sigma}\kappa_{\Gamma})^{-1}, (\kappa^3_{\Sigma}\kappa^2_{\Gamma})^{-1}\}\right\} \tfrac{\alpha}{8}\lambda_{p,\bar n}\\
& \le &  \left\{\kappa^3_{\Sigma}\kappa^2_{\Gamma} (\kappa^3_{\Sigma}\kappa^2_{\Gamma})^{-1}\right\} \tfrac{\alpha}{8}\lambda_{p,\bar n}\\
& = & \tfrac{\alpha}{8}\lambda_{p,\bar n}
\end{eqnarray*}

\newpage

\paragraph*{\sc Part 3} By Equations~(\ref{eq:lambdaopt}) and (\ref{ineq:wo}) we have  $\Vert W_O \Vert_\infty \le \frac{\alpha}{8}\lambda_{p,\bar n}$, and in Part 2 we obtained $\Vert R(\Delta)_O\Vert_\infty\le\frac{\alpha}{8}\lambda_{p,\bar n}$. Therefore, 
\[
\max\big\{\Vert W_O\Vert_\infty, \Vert R(\Delta)_O\Vert_\infty\big\} \le \tfrac{\alpha}{8}\lambda_{p,\bar n}
\]
so Lemma~\ref{convlemma}~(i) implies that $\hat{\tilde\Theta} = \hat{\tilde\Theta}^S$, and consequently 
\begin{equation}\label{eq:mainrate}
\Vert \widehat{\tilde\Theta}-\tilde\Theta\Vert_\infty = \Vert  \hat{\tilde\Theta}^S-\tilde\Theta\Vert_\infty  = \Vert\Delta\Vert_\infty\le C \sigma(\bar n,p^b)
\end{equation}
where 
\begin{equation}\label{eq:constconv}
    C~:=~2\kappa_{\Gamma}(1+\tfrac{8}{\alpha})
\end{equation}
and the inequality is due to Equation~(\ref{ineq:delta2}). 

~

\paragraph*{\sc Part 4} The inequalities in steps 1-2, and thereby the inequality in Equation~(\ref{eq:mainrate}), hold as long as  $\Vert W_O \Vert_\infty \le \sigma(\bar n,p^b)$ (Equation~ (\ref{ineq:wo})). The latter inequality holds with probability larger than $1-(1-\eta)p^{2-b}$, because
\begin{eqnarray*}
P\left(\Vert W_O \Vert_\infty > \sigma(\bar n,p^b)\right) & = & P\left(\bigcup_{(i,j)\in O, i\le j}\left\{|W_{ij}|> \sigma(\bar n,p^b)\right\}\right)\\
& \le & \sum_{(i,j)\in O, i\le j} P\left(|W_{ij}|> \sigma(\bar n,p^b)\right)~~~~\text{[by Union Bound]}\\
& \le & \sum_{(i,j)\in O, i\le j} P\left(|W_{ij}|> \sigma(n_{ij},p^b)\right) ~~~~\text{[because $n_{ij}\ge \bar n$]}\\
& \le & \sum_{(i,j)\in O, i\le j} P\left(|W_{ij}|> \sigma_{ij}(n_{ij},p^b)\right) ~~~~\text{[by Equation~(\ref{eq:tailmax})]}\\
& \le & \sum_{(i,j)\in O, i\le j} p^{-b} ~~~~\text{[by Equation~(\ref{eq:tail})]}\\
& = & \frac{|O|+p}{2}p^{-b}\\
& \le & |O|p^{-b}\\
& = & (p^2-|O^c|)p^{-b}\\
& = & (p^2-\lceil \eta p^2\rceil)p^{-b} ~~~~\text{[by Assumption (A7)]}\\
& \le & (p^2-\eta p^2)p^{-b}\\
& = & (1-\eta)p^{2-b}
\end{eqnarray*}
Therefore, $\Vert\widehat{\tilde\Theta}-\tilde\Theta\Vert_\infty  \le  C\sigma(\bar n,p^b)$ with at least probability $1-(1-\eta)p^{2-b}$, where $C=2\kappa_{\Gamma}(1+\tfrac{8}{\alpha}) $.
\end{proof}

\begin{proof}[Proof of Corollary~\ref{coro:subgaussianrate} (Convergence rate of \madgqlasso{} (sub-Gaussian))]
If $\frac{X^{(j)}_i}{\sqrt{\Sigma_{ii}}}$ is a zero-mean sub-Gaussian random variable with sub-Gaussianity parameter $\omega$, that is
\[\mathbb{E}\left[\exp\left(t\tfrac{X^{(j)}}{\sqrt{\Sigma_{ii}}}\right)\right]\le \exp\left(\omega^2\tfrac{t^2}{2}\right),~~ \forall t\in\mathbb{R},
\]
then, by Lemma~1 of \cite{ravikumar2011high}, for any $\sigma\in \left(0,~8(1+4\omega^2)\max\limits_{1\le i\le p}\Sigma_{ii}\right)$, we have
\begin{eqnarray}
P(|\hat\Sigma_{ij}-\Sigma_{ij}|>\sigma) &\le& 4\exp\left(-\tfrac{n_{ij}\sigma^2}{128\cdot (1+4\omega^2)^2\cdot\max\limits_{1\le i\le p} \Sigma_{ii}^2}\right),
\end{eqnarray}
where $n_{ij}$ is the sample size used in the empirical covariance $\hat\Sigma_{ij}=\frac{1}{n_{ij}}\sum_{k}X^{(k)}_iX^{(k)}_j$ (under the assumption of zero-means). We want to upper-bound the global tail function $\sigma(m,\varepsilon)$ (Equation~(\ref{eq:tailmax})) by using the inequality above. By solving the equation
\[4\exp\left(-\tfrac{m\sigma^2}{128\cdot (1+4\omega^2)^2\cdot\max\limits_{1\le i\le p}\Sigma_{ii}^2}\right)=\varepsilon^{-1} \]
for $\sigma$, we obtain the upper-bound on the tail function (Equation~(\ref{eq:tail}))
\[
\sigma_{ij}(m,\varepsilon)~\le~ \sqrt{\tfrac{\log(4\varepsilon)128\cdot (1+4\omega^2)^2\max\limits_{1\le i\le p}\Sigma_{ii}^2}{m}}
\]
Thus, an upper-bound on the global tail function (Equation~(\ref{eq:tailmax})) is
\begin{equation}\label{eq:subgausstailbound}
\sigma(m,\varepsilon)~=~\max_{(i,j)\in O}\sigma_{ij}(m,\varepsilon)~\le~ \sigma_{\rm SG}(m,\epsilon):=\sqrt{\tfrac{\log(4\varepsilon)128\cdot (1+4\omega^2)^2\max\limits_{1\le i\le p}\Sigma_{ii}^2}{m}}
\end{equation}
so the minimal required sample size given in Equation~(\ref{eq:nbarmin}) is upper-bounded by 
\begin{equation}\label{eq:subgaussminsamconv}
\bar n^*_{SG} ~:=~ \lceil H\tilde d^2(b\log p+\log 4)\rceil
\end{equation}
where
\begin{equation}\label{eq:subgausssampcomplconst}
   H := \left(2(1+\tfrac{8}{\alpha})^{2}3 \max\{\kappa_{\Sigma}\kappa_{\Gamma}, \kappa^3_{\Sigma}\kappa^2_{\Gamma}\}\right)^2 128(1+4\omega^2)^2\max_{1\le i\le p}\Sigma_{ii}^2
\end{equation}
Therefore, under the assumptions of Theorem~\ref{theo:convrate}, for any  $\bar n\ge \bar n^*_{SG}$, with probability larger than $1-(1-\eta)p^{2-b} $, we have
\begin{equation}\label{eq:upperboundsigma}
\Vert\widehat{\tilde\Theta}-\tilde\Theta\Vert_\infty  ~\le~  C\sigma(\bar n,p^b) ~\le~ C_{\rm SG}\sqrt{\tfrac{b\log p+\log 4}{\bar n}}
\end{equation}
where $C=2\kappa_{\Gamma}(1+\tfrac{8}{\alpha}) $ as in Equation~(\ref{eq:constconv}), and
\begin{equation}\label{eq:CSG}
    C_{SG} ~:=~ 2\kappa_{\Gamma}(1+\tfrac{8}{\alpha}) \sqrt{128}(1+4\omega^2)\max\limits_{1\le i\le p}\Sigma_{ii}
\end{equation}
\end{proof}

\newpage

\begin{proof}[Proof of Theorem~\ref{theo:graphfinite} (GQ Graph recovery (finite samples))] ~
\paragraph*{(i)} If $\bar n\ge \bar n^*_O:=\max\left\{\bar n^*, \min\{m: C\sigma(m,p^b)< 
\nu/2-\delta \} \right\}$ then, by Theorem~\ref{theo:convrate}, with probability larger than $1-(1-\eta)p^{2-b}$ we have
\begin{equation}
 \Vert\widehat{\tilde\Theta}-\tilde\Theta\Vert_{\infty}  ~\le~ C\sigma(\bar n,p^b)~<~
\nu/2-\delta
\end{equation}
and by Triangle Inequality
\begin{eqnarray*}
\hat\delta &:=&\max_{(i,j)\in O,i\neq j}|\widehat{\tilde\Theta}_{ij}-\Theta_{ij}| \\
&=&\max_{(i,j)\in O,i\neq j}|\widehat{\tilde\Theta}_{ij}-\tilde\Theta_{ij}+\tilde\Theta_{ij}-\Theta_{ij}| \\
&\le  &  \max_{(i,j)\in O,i\neq j} \left(|\widehat{\tilde\Theta}_{ij}-\tilde\Theta_{ij}| +|\tilde\Theta_{ij}-\Theta_{ij}|\right)\\
&\le  &  \max_{(i,j)\in O,i\neq j}|\widehat{\tilde\Theta}_{ij}-\tilde\Theta_{ij}| +\max_{(i,j)\in O,i\neq j}|\tilde\Theta_{ij}-\Theta_{ij}|\\
&\le  &  \Vert\hat{\tilde\Theta}-\tilde\Theta\Vert_{\infty} +\delta\\
& \le & C\sigma(\bar n,p^b)+\delta \\
& < & \nu/2 - \delta + \delta\\
& = & \nu/2
\end{eqnarray*}
Since $\hat\delta<\nu/2$, Lemma~\ref{lemma:exactO} guarantees that for any $\tau\in[\hat\delta,\nu-\hat\delta)$ we have $|\widehat{\tilde\Theta}_{ij}|>\tau \Leftrightarrow \Theta_{ij}\neq 0 $ with ${\rm sign}(\hat{\tilde\Theta}_{ij})={\rm sign}(\Theta_{ij})$, $ \forall (i,j)\in E_O$. Finally, since $\hat\delta\le\delta_{\bar n,p}:=C\sigma(\bar n,p^b)+\delta<\nu/2$, we have $[\delta_{\bar n,p},\nu-\delta_{\bar n,p})\subseteq [\hat\delta,\nu-\hat\delta)$, so we can conclude that if $\tau\in [\delta_{\bar n,p},\nu-\delta_{\bar n,p})$, then with probability larger than $1-(1-\eta)p^{2-b}$ we have $\hat{\mathcal{E}}_O=E_O$.

\paragraph*{(ii)} 
Let $\hat{\tilde\Theta}^{(k)}:= \hat{\tilde\Theta}_{V_kV_k}- \hat{\tilde\Theta}_{V_kV_k^c} \hat{\tilde\Theta}_{V_k^cV_k^c}^{-1} \hat{\tilde\Theta}_{V_k^cV_k}$ be the \madgqlasso{} Schur complement relative to the node pairs $V_k\times V_k$. Under the conditions of Theorem~\ref{theo:convrate}, for any $\bar n\ge\bar n^*$, with probability larger than $1-(1-\eta)p^{2-b}$, we have
$\Vert \hat{\tilde\Theta}-\tilde\Theta\Vert_\infty \le C\sigma(\bar n,p^b)$ and also, as shown in {\sc Part 3} of the proof of Theorem~\ref{theo:convrate}, we have $\hat{\tilde\Theta}=\hat{\tilde\Theta}^S$, where $\hat{\tilde\Theta}^S$ is the matrix in Equation~(\ref{eq:l1gqS}) satisfying $\{(i,j):\hat{\tilde\Theta}^S_{ij}\neq 0\}\subseteq\{(i,j):\tilde\Theta_{ij}\neq 0\}$; hence, $\hat{\tilde\Theta}-\tilde\Theta$ has maximum row-degree no larger than $\tilde d$, which is the max row-degree of $\tilde\Theta$. 
So, for any $k=1,...,K$, by Lemma~\ref{lemma:schurineq} and Lemma~\ref{lemma:spectralinf} we have
\begin{eqnarray*}
\Vert \hat{\tilde\Theta}^{(k)}- \tilde\Theta^{(k)}\Vert_\infty &~\le~&  \frac{\lambda_{\rm max}(\hat{\tilde\Theta})}{\lambda_{\rm min}(\hat{\tilde\Theta})}\frac{\lambda_{\rm max}(\tilde\Theta)}{\lambda_{\rm min}(\tilde\Theta)}\min\{\sqrt{p+s},\tilde d\}\Vert\hat{\tilde\Theta}-\tilde\Theta\Vert_\infty\\
&~\le~&  \frac{\lambda_{\rm max}(\hat{\tilde\Theta})}{\lambda_{\rm min}(\hat{\tilde\Theta})}\frac{\lambda_{\rm max}(\tilde\Theta)}{\lambda_{\rm min}(\tilde\Theta)}\min\{\sqrt{p+s},\tilde d\} C\sigma(\bar n,p^b)
\end{eqnarray*}
where $s = |\{(i,j):i\neq j,\tilde\Theta_{ij}\neq 0\}|$ is the number of nonzero off-diagonals of $\tilde\Theta$,  
and by Lemma~\ref{lemma:spectralinf} and Lemma~\ref{lemma:eigenbounds}, 
\begin{eqnarray*}
\lambda_{\rm max}(\hat{\tilde\Theta}) &~\le~& \lambda_{\rm max}(\tilde\Theta)+\Vert \hat{\tilde\Theta}-\tilde\Theta\Vert_2 ~\le~ \lambda_{\rm max}(\tilde\Theta)+\min\{\sqrt{p+s},\tilde d\}\Vert \hat{\tilde\Theta}-\tilde\Theta\Vert_\infty\\
&~\le~& \lambda_{\rm max}(\tilde\Theta)+\min\{\sqrt{p+s},\tilde d\} C\sigma(\bar n,p^b)
\end{eqnarray*}
and
\begin{eqnarray*}
\lambda_{\rm min}(\hat{\tilde\Theta}) &~\ge~& \lambda_{\rm min}(\tilde\Theta)-\Vert \hat{\tilde\Theta}-\tilde\Theta\Vert_2 
~\ge~ \lambda_{\rm min}(\tilde\Theta)-\min\{\sqrt{p+s},\tilde d\}\Vert \hat{\tilde\Theta}-\tilde\Theta\Vert_\infty \\
&~\ge~& \lambda_{\rm min}(\tilde\Theta)-\min\{\sqrt{p+s},\tilde d\} C\sigma(\bar n,p^b)
\end{eqnarray*}
Thus, if $\bar n\ge \max\{\bar n^*, \min\{m: \min\{\sqrt{p+s},\tilde d\} C\sigma(m,p^b)<\lambda_{\rm min}(\tilde\Theta)/2\}$ and the other conditions of Theorem~\ref{theo:convrate} hold, then with probability larger than $1-(1-\eta)p^{2-b}$ we have
\[
\frac{\lambda_{\rm max}(\hat{\tilde\Theta})}{\lambda_{\rm min}(\hat{\tilde\Theta})} ~\le ~ \frac{\lambda_{\rm max}(\tilde\Theta)+\lambda_{\rm min}(\tilde\Theta)/2}{\lambda_{\rm min}(\tilde\Theta)-\lambda_{\rm min}(\tilde\Theta)/2}
~\le~\frac{\lambda_{\rm max}(\tilde\Theta)+\lambda_{\rm max}(\tilde\Theta)/2}{\lambda_{\rm min}(\tilde\Theta)/2}~=~ 3\tilde{\mathcal{K}}
\]
where $\tilde{\mathcal{K}}=\frac{\lambda_{\rm max}(\tilde\Theta)}{\lambda_{\rm min}(\tilde\Theta)}$ is the condition number of $\tilde\Theta$, and so for any $k=1,...,K$
\[
\Vert \hat{\tilde\Theta}^{(k)}- \tilde\Theta^{(k)}\Vert_\infty ~\le~  3\tilde{\mathcal{K}}^2\min\{\sqrt{p+s},\tilde d\} C\sigma(\bar n,p^b)
\]
which implies
\[
\max_{1\le k\le K} \Vert \hat{\tilde\Theta}^{(k)}- \tilde\Theta^{(k)}\Vert_\infty ~\le~D\min\{\sqrt{p+s},\tilde d\} \sigma(\bar n,p^b)
\]
where 
\begin{equation}\label{eq:constsupreco}
D:=3\tilde{\mathcal{K}}^2 C
\end{equation}
Therefore, since $3\tilde{\mathcal{K}}^2> 1$, if 
\[\bar n ~\ge~ \bar n^*_{O^c}:= \max\big\{\bar n^*, \min \big\{m: D\min\{\sqrt{p+s},\tilde d\} \sigma(m,p^b)<\tfrac{\min(\psi,\lambda_{\rm min}(\tilde\Theta))}{2}\big\} \big\}\]
and all conditions of Theorem~\ref{theo:convrate} hold, then, with probability larger than $1-(1-\eta)p^{2-b}$,
\begin{equation}\label{eq:schurin}
\max_{1\le k\le K} \Vert \hat{\tilde\Theta}^{(k)}- \tilde\Theta^{(k)}\Vert_\infty ~\le~ \tau_0~:=~ D\min\{\sqrt{p+s},\tilde d\} \sigma(\bar n,p^b) ~<~\tfrac{\psi}{2}
\end{equation}
where $\psi>0$ is defined in Equation~(\ref{eq:psi}). 
If Equation~(\ref{eq:schurin}) holds, then for any $(i,j,k)$, 
\begin{eqnarray}
|\hat{\tilde\Theta}^{(k)}_{ij}|>\tau_0 
&&~\Rightarrow~ |\tilde\Theta^{(k)}_{ij}|>0\\
|\hat{\tilde\Theta}^{(k)}_{ij}|< t-\tau_0
&&~\Rightarrow~ |\tilde\Theta^{(k)}_{ij}|< t, ~~~~~~~~~\text{for any }t>\tau_0
\end{eqnarray}
Moreover, by definition of $\psi$, we are certain that $|\tilde\Theta^{(k)}_{ij}|\notin (0,\psi)\cup(\delta-\psi,\delta) $ and $|\hat{\tilde\Theta}^{(k)}_{ij}|\notin (\tau_0,\psi-\tau_0)\cup(\delta-\psi+\tau_0,\delta-\tau_0) $, so we also have
\begin{eqnarray}
|\tilde\Theta^{(k)}_{ij}|>0 
&&~\Rightarrow~ |\tilde\Theta^{(k)}_{ij}|\ge\psi  
~\Rightarrow~ |\hat{\tilde\Theta}^{(k)}_{ij}|>\psi-\tau_0>\tau_0\\
|\tilde\Theta^{(k)}_{ij}|<\delta 
&&~\Rightarrow~|\tilde\Theta^{(k)}_{ij}|\le\delta-\psi
~\Rightarrow~|\hat{\tilde\Theta}^{(k)}_{ij}|\le\delta-\psi+\tau_0< \delta-\tau_0
\end{eqnarray}
Thus, if Equation~(\ref{eq:schurin}) holds, then for all $t>\tau_0$, we have
\begin{equation}\label{eq:distinclusion}
\{(i,j,k): \tau_0<|\hat{\tilde\Theta}^{(k)}_{ij}|< t-\tau_0\} \subseteq \{(i,j,k):0<|\tilde\Theta^{(k)}_{ij}|< t\},
\end{equation}
with equality at $t=\delta$, guaranteeing that for any  $t\ge\delta$,
\begin{equation}\label{eq:distinclusionx}
\{(i,j,k):0<|\tilde\Theta^{(k)}_{ij}|< \delta\}\subseteq \{(i,j,k): \tau_0<|\hat{\tilde\Theta}^{(k)}_{ij}|< t-\tau_0\}
\end{equation}
Now, notice that, under Assumption~(A5), for any $\tau\in[\delta,\nu]$, the set $W_\tau$ defined in Algorithm~\ref{algo:fullrecoveryK} equals the set $D_{\rm off}$ in Equation~(\ref{eq:supersetKoff}) (see Proof of Theorem~\ref{theo:fullrecoveryK}), so if Equation~(\ref{eq:distinclusionx}) holds, then the set $\hat W_{\tau_0,\tau_1}$ defined in Algorithm~\ref{algo:fullrecoveryKn} equals $D_{\rm off}$ for any $\tau_1\in[\delta-\tau_0,\nu-\tau_0]$. Therefore, if $\bar n\ge \bar n_{O^c}^*$ and all the other conditions of Theorem~\ref{theo:convrate} hold, then for any $\tau_1\in[\delta-\tau_0,\nu-\tau_0]$, with probability larger than $1-(1-\eta)p^{2-b}$,  we have $\hat{\mathcal{U}}_{\tau_0,\tau_1} = \mathcal{S}_{\rm off}$ where, under Assumptions~(A4), $\mathcal{S}_{\rm off}$ is the minimal superset of $E_{O^c}$ as per Theorem~\ref{theo:fullrecoveryK}. 
\end{proof}

\more\more

\begin{proof}[Proof of Corollary~\ref{coro:subgaussiangraph} {\small (GQ Graph recovery (finite samples, sub-Gaussian))}] This corollary restates Theorem~\ref{theo:graphfinite} for the case of sub-Gaussian data, and exploits the results in Corollary~\ref{coro:subgaussianrate}. 
\paragraph*{(i)}  By using the upper bound in Equation~(\ref{eq:upperboundsigma}), we obtain
    \begin{eqnarray*}
        ~\min\left\{m:C\sigma(m,p^b)<\tfrac{\nu}{2}-\delta\right\} &\le& \min\left\{m:C_{\rm SG}\sqrt{\tfrac{b\log p+\log 4}{m}}<\tfrac{\nu}{2}-\delta\right\}\\
        &\le& \left\lceil\tfrac{C_{\rm SG}^2(b\log p+\log 4)}{(\nu/2-\delta)^2}\right\rceil
    \end{eqnarray*}
    Thus, the minimal sample size $\bar n_O^*$ required by Theorem~\ref{theo:graphfinite}(i) is upper bounded by
    \begin{equation}
        \bar n^*_{\rm SG,O} ~:=~ \max\left\{\bar n^*_{\rm SG},\left\lceil\tfrac{C_{\rm SG}^2(b\log p+\log 4)}{(\nu/2-\delta)^2}\right\rceil \right\}
    \end{equation}
    where $\bar n^*_{\rm SG}$ is given in Equation~(\ref{eq:subgaussminsamconv}). Therefore, following the proof of Theorem~\ref{theo:graphfinite}(i), we obtain that, for any $\bar n\ge \bar n^*_{\rm SG,O}$, with probability larger than $1-(1-\eta)p^{2-b}$,
    \[
    \hat \delta := \max_{(i,j)\in O,i\neq j}|\hat{\tilde\Theta}_{ij}-\Theta_{ij}| \le  \delta+C_{\rm SG} \sqrt{\tfrac{b\log p+\log 4}{\bar n}}<\nu/2
    \]
    so if we set $\delta_{\bar n, p} :=\delta+C_{\rm SG} \sqrt{\tfrac{b\log p+\log 4}{\bar n}}$, then for any $\tau\in [\delta_{\bar n,p},\nu-\delta_{\bar n,p})$,  Algorithm~\ref{algo:fullrecoveryKn} yields $\hat{\mathcal{E}}_O=E_O$. 
    
\paragraph*{(ii)} By using the upper bound in Equation~(\ref{eq:upperboundsigma}), we obtain
    \begin{eqnarray*}
& &\min \big\{m: D\min\{\sqrt{p+s},\tilde d\} \sigma(m,p^b)<\tfrac{\min(\psi,\lambda_{\rm min}(\tilde\Theta))}{2}\big\} \\
&& ~~~~~\le \min \big\{m: D_{\rm SG}\sqrt{\tfrac{\min\{p+s,\tilde d^2\}(b\log p+\log 4)}{\bar n}}<\tfrac{\min(\psi,\lambda_{\rm min}(\tilde\Theta))}{2}\big\}\\
&& ~~~~~\le\left\lceil\tfrac{4D_{\rm SG}^2\min\{p+s,\tilde d^2\}(b\log p+\log 4)}{\min(\psi^2,\lambda_{\rm min}(\tilde\Theta)^2)}\right\rceil
    \end{eqnarray*}
where $D$ is defined in Equation~(\ref{eq:constsupreco}), and
\begin{equation}\label{eq:Dsg}
D_{\rm GS}:=3\tilde{\mathcal{K}}^2 C_{\rm SG},
\end{equation}
where $C_{\rm SG}$ is defined in Equation~(\ref{eq:CSG}) and $\mathcal{K}$ is the condition number of $\tilde\Theta$. Thus, the minimal sample size $\bar n_{O^c}^*$ required by Theorem~\ref{theo:graphfinite}(ii) is upper bounded by
    \begin{equation}
        \bar n^*_{\rm SG,O^c} ~:=~ \max\left\{\bar n^*_{\rm SG},\left\lceil\tfrac{4D_{\rm SG}^2\min\{p+s,\tilde d^2\}(b\log p+\log 4)}{\min(\psi^2,\lambda_{\rm min}(\tilde\Theta)^2)}\right\rceil \right\}
    \end{equation}
    where $\bar n^*_{\rm SG}$ is given in Equation~(\ref{eq:subgaussminsamconv}).
 Therefore, following the proof of Theorem~\ref{theo:graphfinite}(ii), we obtain that if we set $\tau_0 = D_{\rm SG}\sqrt{\tfrac{\min\{p+s,\tilde d^2\}(b\log p+\log 4)}{\bar n}}$, then for any $\bar n\ge \bar n^*_{\rm SG,O^c}$ and for any $\tau_1\in [\delta-\tau_0,\nu-\tau_0]$, with probability larger than $1-(1-\eta)p^{2-b}$, Algorithm~\ref{algo:fullrecoveryKn} yields $\hat{\mathcal{E}}_{O^c}=\mathcal{S}_{\rm off}$.
\end{proof}

\section{Auxiliary Results}\label{app:auxiliary}
This appendix contains two important lemmas that are used in the proofs in Appendix~\ref{app:proofmains}: Lemma~\ref{lemma:madcont} and Lemma~\ref{convlemma}. Lemma~\ref{lemma:madcont} establishes some continuity properties of the \madgq{} matrix $\tilde\Theta$; these properties are used in the proofs of Theorems~\ref{theo:nofnd} and \ref{theo:popO}. Lemma~\ref{convlemma} establishes three important inequalities that are used in the proof of Theorem~\ref{theo:convrate}.

\vspace{0.5mm}

\begin{lemma}[\scbf{MAD$_{GQ}$ continuity}]\label{lemma:madcont}
Let $O$ be the observed set of node pairs (including loops). 
Let $\mathcal{T}_{M_O} = \{M_{O^c}: M\succ 0\} $ be the set of all portions $M_{O^c}$ that complete the partial matrix $M_O$ into a $p\times p$ positive definite matrix $M$, and let $\mathcal{P}_O = \{M_O: \mathcal{T}_{M_O}\neq\emptyset\}$ be the set of positively completable partial $p\times p$ matrices. For a given $p\times p$ positive definite matrix $\Theta\succ 0$, define 
\begin{equation}
\mathcal{Z}(\Theta,O,\gamma) ~=~ \left\{\Theta^*\succ 0:~ \Theta^*_O=\Theta_O,~ E^*_{O^c}=E_{O^c}, ~\Vert\Theta_{O^c}^*\Vert_\infty=\gamma\right\} 
\end{equation}
which is the set of all positive definite matrices that equal $\Theta$ over the entry set $O$, have the same graphical structure of $\Theta_{O^c}$ in their $O^c$ portion, and their largest magnitude in $O^c$ is $\gamma$. Correspondingly, define the nonnegative function
\begin{equation}\label{eq:deltabargen}
\bar{\delta}(\Theta,O,\gamma) = \max\limits_{\Theta^*\in\mathcal{Z}(\Theta,O,\gamma)} |\tilde\Theta^*_{ij}-\Theta^*_{ij}|
\end{equation}
where $\tilde\Theta^*$ is the \madgq{} approximation of $\Theta^*$ (Equation~(\ref{eq:gq0})) based on the portion of covariance  $\Sigma^*_O=\left[\Theta^{*-1}\right]_O$. We have 
\begin{enumerate}[(i).]
\item The \madgq{} matrix $\tilde\Theta$ in Equation~(\ref{eq:gq0})  is a 1:1 bicontinuous function (homeomorphism) of $\Sigma_O\in\mathcal{P}_O$ (in any normed topological space).

\item For any fixed $\Theta_O\in\mathcal{P}_O$, $\tilde\Theta$ is a continuous function of $\Theta_{O^c}\in\mathcal{T}_{\Theta_O}$.

\item For any fixed $\Theta\succ 0$ and $O$, $\bar\delta(\Theta,O,\gamma)$ is a continuous function of $\gamma\in[\gamma_1,\gamma_2]$, for some $\gamma_1\ge 0$ and $\gamma_2<\min\limits_{(i,j)\in O^c}\sqrt{\Theta_{ii}\Theta_{jj}}$. If $0\in\mathcal{T}_{\Theta_O}$, then $\gamma_1=0$.\less
\end{enumerate} 
\begin{proof} ~\less\less
\paragraph*{(i)} The portion $O^c$ of the max-determinant solution in Equation~(\ref{eq:maxdet}) can be written as
\begin{equation}\label{eq:Fun}
\tilde\Sigma_{O^c} :=F(\Sigma_O) ~=~ \underset{\Sigma^*_{O^c}\in \mathcal{A}(\Sigma_O)}{\arg\max} ~~ f(\Sigma_O,\Sigma^*_{O^c}),\less
\end{equation}
where the real-valued function $f(\Sigma_O,\Sigma_{O^c}) = \det\Sigma$ is continuous over the set  \[\left\{(\Sigma^*_O,\Sigma^*_{O^c}):~\Sigma^*_O\in\mathcal{P}_O,~\Sigma^*_{O^c}\in \mathcal{T}_{\Sigma^*_O}\right\},\]
and $\mathcal{A}(\Sigma_O)\subset\mathcal{T}_{\Sigma_O}$ is any arbitrary continuous compact-valued correspondence of $\Sigma_O$ such that $ \tilde\Sigma_{O^c}\in\mathcal{A}(\Sigma_O)$; the existence of the set $\mathcal{A}(\Sigma_O)$ is guaranteed by the existence and uniqueness of the solution of the max-determinant problem (Lemma~\ref{lemma:maxdetopt}). Then, by Berge's Maximum Theorem \citep{berge1997topological}, $F(\Sigma_O)$ is an upper hemicontinuous correspondence with nonempty and compact values. Moreover, since $\Sigma_O$ is completable to a positive definite matrix, the max-determinant problem has a unique solution (Lemma~\ref{lemma:maxdetopt}), i.e. $F(\Sigma_O)$ is single-valued. Hence, $F(\Sigma_O)$ is a continuous function of $\Sigma_O$. Moreover, since $\tilde\Sigma_O\equiv\Sigma_O$, trivially also $\tilde\Sigma_O$ is a continuous function of $\Sigma_O$. Therefore, the full matrix $\tilde\Sigma $ is a continuous function of $\Sigma_O$. Correspondingly, also $\tilde\Theta = \tilde\Sigma^{-1} :=M(\Sigma_O)$ is a continuous function of $\Sigma_O$. On the other hand,  $\Sigma_O=\left[\tilde\Theta^{-1}\right]_O$ and $\tilde\Theta_{O^c}\equiv 0$, i.e. $\Sigma_O$ is a continuous function $\tilde\Theta_{O}$. Therefore, we can conclude that $\tilde\Theta_O$ is a 1:1 bicontinuous function (homeomorphism) of $\Sigma_O$, and so is $\tilde\Theta$.\less
\paragraph*{(ii)} The function $ G(\Theta):=\Theta^{-1}$ is continuous for $\Theta\succ 0$. Then, for any fixed $\Theta_O\in\mathcal{P}_O$, the function $G_{\Theta_O}(\Theta_{O^c}):=G(T)|_{T_O=\Theta_O,T_{O^c}=\Theta_{O^c}}$ is a continuous function of $\Theta_{O^c}\in\mathcal{T}_{\Theta_O}$, and so is $H_{\Theta_O}(\Theta_{O^c}) := [G_{\Theta_O}(\Theta_{O^c})]_O$. Thus, for any fixed $\Theta_O\in\mathcal{P}_O$, the \madgq{} matrix can be expressed as the composite function $\tilde\Theta=M(H_{\Theta_O}(\Theta_{O^c}))$, where $M(\Sigma_O)$ is a continuous function of $\Sigma_O\in\mathcal{P}_O$ (by part (i)). Therefore, for any fixed $\Theta_O\in\mathcal{P}_O$, $\tilde\Theta$ is a continuous function of $\Theta_{O^c}\in\mathcal{T}_{\Theta_O}$.\less
\paragraph*{(iii)} By part (ii), for a fixed $\Theta_O\in\mathcal{P}_O$, the maximal distortion $\delta$ (Equation~(\ref{eq:delta})) is a continuous function of $\Theta_{O^c}$, say $\delta=\phi(\Theta_{O^c})$. Moreover, for a fixed $\Theta_O\in\mathcal{P}_O$ and graph structure $E_{O^c}$ in $O^c$, define the set\less
 \[C_{\Theta_O,E_{O^c}}(\gamma):=\left\{\Theta^*_{O^c}\in \mathcal{T}_{\Theta_O}:~E^*_{O^c}=E_{O^c},~\Vert\Theta^*_{O^c}\Vert_\infty=\gamma\right\}\]
which is a continuous mapping of the interval $\mathcal{G}=[\gamma_1,\gamma_2]$ to $\mathcal{T}_{\Theta_O}$, where $\gamma_1,\gamma_2$ are such that $C_{\Theta_O,E_{O^c}}(\gamma)\neq\emptyset$ for any $\gamma\in\mathcal{G}$; indeed $\gamma_1\ge 0$, and $\gamma_2<\min\limits_{(i,j)\in O^c}\sqrt{\Theta_{ii}\Theta_{jj}}$ to allow $\Theta$ be positive definite. Then, we can write
\[
\bar\delta(\Theta,O,\gamma) = \max\limits_{\Theta^*_{O^c}\in C_{\Theta_O,E_{O^c}}(\gamma)} \phi(\Theta^*_{O^c})
\]
By Berge's Maximum Theorem,  $\bar\delta(\Theta,O,\gamma)$ is a continuous function of $\gamma$. If $0\in\mathcal{T}_{\Theta_O}$, then $C_{\Theta_O,E_{O^c}}(0)\neq\emptyset$ and $\gamma_1=0$.
\end{proof}
\end{lemma}

\vspace{0.5mm}

\begin{lemma}\label{convlemma}
The quantities defined in Equations~(\ref{eq:l1gqS})-(\ref{eq:nbarmin}) have the following properties: 
\begin{enumerate}[(i).]
\item If Assumption~(A8) holds and $\max\{\Vert W_O\Vert_\infty,\Vert R(\Delta)_O\Vert_\infty\}\le \tfrac{\alpha}{8}\lambda_{p,\bar n}$, then $\hat{\tilde\Theta} = \hat{\tilde\Theta}^S$. 
\item If $\Vert \Delta \Vert_\infty\le (3\kappa_{\Sigma}\tilde d)^{-1}$, then $ \Vert R(\Delta)_O \Vert_\infty \le  \Vert R(\Delta) \Vert_\infty \le \frac{3}{2} \tilde d\Vert \Delta \Vert^2_\infty \kappa^3_{\Sigma}$.
\item If $r:= 2\kappa_{\Gamma}(\Vert W_O \Vert_\infty + \lambda_{p,\bar n}) \le \min\left\{(3\kappa_{\Sigma}\tilde d)^{-1}, (3\kappa^3_{\Sigma}\kappa_{\Gamma}\tilde d)^{-1}\right\}$, then $\Vert \Delta \Vert_\infty \le r$.
\end{enumerate}

\begin{proof} This proof exploits the same techniques used in the proofs of Lemmas~4, 5, and 6 of \cite{ravikumar2011high}. However, our proofs differ in several aspects, because our estimator $\hat{\tilde\Theta}$ is an estimator of $\tilde\Theta$, not $\Theta$, and it is based on the incomplete set of empirical covariances $\hat\Sigma_O$.
\paragraph*{(i)} We are going to show that the solution $\hat{\tilde\Theta}^S$ in Equation~(\ref{eq:l1gqS}) is also a solution of the \madgqlasso{} optimization problem in Equation~(\ref{eq:l1gq}) with $\Lambda_{ij}=\lambda_{p,\bar n}$ for all $(i,j)\in O$, i.e. $\hat{\tilde\Theta}^S=\hat{\tilde\Theta}$, if Assumption~(A8) holds and $\max\{\Vert W_O\Vert_\infty,\Vert R(\Delta)_O\Vert_\infty\}\le \tfrac{\alpha}{8}\lambda_{p,\bar n}$. We prove this by following the strategy used in the proof of Lemma~4 in \cite{ravikumar2011high}.

Recall that $\tilde\Sigma=\tilde\Theta^{-1}$,   $W_O=\hat\Sigma_O-\tilde\Sigma_O=\hat\Sigma_O-\Sigma_O$ because $\tilde\Sigma_O=\Sigma_O$, $\Delta=\hat{\tilde\Theta}^S-\tilde\Theta$, and $R(\Delta)=\big(\hat{\tilde\Theta}^S\big)^{-1}-\tilde\Sigma+\tilde\Sigma\Delta\tilde\Sigma$, and $S=\{(i,j):\tilde\Theta_{ij}\neq 0\}$. 
If $\hat{\tilde\Theta}^S$ is a solution to  the optimization problem in Equation~(\ref{eq:l1gq}), then it must satisfy the optimality condition
\[
\hat\Sigma_O-\big[(\hat{\tilde\Theta}^S)^{-1}\big]_O+\lambda_{p,\bar n}Z_O = 0
\]
where $Z_O$ must be part of the sub-differential of $\Vert *\Vert_{1,\rm off}$ evaluated at $\hat{\tilde\Theta}^S_O$. 
This condition can be rewritten as
\[
\left[\tilde\Sigma\Delta\tilde\Sigma\right]_O+W_O-R(\Delta)_O+\lambda_{p,\bar n}Z_O = 0
\]

Equivalently we can rewrite
\[
\Gamma_{O,O}\Delta_O+ W_O-R(\Delta)_O+\lambda_{p,\bar n}Z_O = 0
\]
where $\Gamma_{O,O}$ is a submatrix of $\Gamma=\tilde\Sigma\otimes\tilde\Sigma$ relative to the index set of pairs $(i,j)\in O$, and all other terms $\Delta_O,W_O,R(\Delta)_O,Z_O$ are column vectors. Let us rewrite the equation above as
\begin{equation*}
    \left\{\begin{array}{l}
       \Gamma_{O\cap S,O\cap S}\Delta_{O\cap S} + W_{O\cap S}-R(\Delta)_{O\cap S}+\lambda_{p,\bar n}Z_{O\cap S} = 0 \\
       \Gamma_{O\cap S^c,O\cap S}\Delta_{O\cap S} + W_{O\cap S^c}-R(\Delta)_{O\cap S^c}+\lambda_{p,\bar n}Z_{O\cap S^c} = 0 
    \end{array}\right.
\end{equation*}
where we used the fact that $\Delta_{O\cap S^c}=0$ (note: $S\subseteq O$, so $O\cap S=S$, but we keep the notation $O\cap S$ to remind that we are restricting all computations over the set $O$). From the first equation we obtain
\[
\Delta_{O\cap S} = \left(\Gamma_{O\cap S,O\cap S}\right)^{-1}\left[-W_{O\cap S}+R(\Delta)_{O\cap S}-\lambda_{p,\bar n}Z_{O\cap S}\right]
\]
Plugging this solution in the second equation of the system, we obtain
\begin{eqnarray*}
Z_{O\cap S^c} &=& \frac{1}{\lambda_{p,\bar n}}\Gamma_{O\cap S^c,O\cap S}\left(\Gamma_{O\cap S,O\cap S}\right)^{-1}\left(W_{O\cap S}-R(\Delta)_{O\cap S}\right)\\
&& +\Gamma_{O\cap S^c,O\cap S}\left(\Gamma_{O\cap S,O\cap S}\right)^{-1}Z_{O\cap S}-\frac{1}{\lambda_{p,\bar n}}\left(W_{O\cap S^c}-R(\Delta)_{O\cap S^c}\right)
\end{eqnarray*}
Then
\begin{eqnarray*}
\Vert Z_{O\cap S^c}\Vert_\infty &\le& \frac{1}{\lambda_{p,\bar n}}\vertiii{\Gamma_{O\cap S^c,O\cap S}\left(\Gamma_{O\cap S,O\cap S}\right)^{-1}}_\infty\left(\Vert W_{O\cap S}\Vert_\infty+\Vert R(\Delta)_{O\cap S}\Vert_\infty\right)\\
&& +\vertiii{\Gamma_{O\cap S^c,O\cap S}\left(\Gamma_{O\cap S,O\cap S}\right)^{-1}}_\infty \cdot\Vert Z_{O\cap S}\Vert_\infty\\
&& +\frac{1}{\lambda_{p,\bar n}}\left(\Vert W_{O\cap S^c}\Vert_\infty +\Vert R(\Delta)_{O\cap S^c}\Vert_\infty\right)
\end{eqnarray*}
where $\vertiii X_\infty:=\max\limits_{1\le i\le p}\sum\limits_{j=1}^p|X_{ij}|$ and $\Vert X \Vert_\infty = \max\limits_{i,j}|X_{ij}|$ is the max-norm.
If assumption (A8) holds, then
\[ 
\vertiii{\Gamma_{O\cap S^c,O\cap S}\left(\Gamma_{O\cap S,O\cap S}\right)^{-1}}_\infty\le 1-\alpha
\]
and $\Vert Z_{O\cap S}\Vert_\infty\le 1$ because $Z_{O}$ is assumed to be part of the sub-differential of $\Vert *\Vert_{1,\rm off}$ evaluated at $\hat{\tilde\Theta}^S_O$. Therefore
\begin{eqnarray*}
\Vert Z_{O\cap S^c}\Vert_\infty &\le & \frac{1-\alpha}{\lambda_{p,\bar n}}(\Vert W_{O\cap S}\Vert_\infty +\Vert R(\Delta)_{O\cap S}\Vert_\infty)\\
&& + (1-\alpha)+\frac{1}{\lambda_{p,\bar n}}(\Vert W_{O\cap S^c}\Vert_\infty+\Vert R(\Delta)_{O\cap S^c}\Vert_\infty)\\
&\le & \frac{2-\alpha}{\lambda_{p,\bar n}}\left(\Vert W_O\Vert_\infty+\Vert R(\Delta)_O\Vert_\infty\right)+1-\alpha
\end{eqnarray*}
Therefore, if $\max\{\Vert W_O\Vert_\infty,\Vert R(\Delta)_O\Vert_\infty\}\le \frac{\alpha}{8}\lambda_{p,\bar n}$, then
\begin{eqnarray*}
    \Vert Z_{O\cap S^c}\Vert_\infty &\le& \frac{\alpha}{2}-\frac{\alpha^2}{4}+1-\alpha ~\le~1-\frac{\alpha}{2}~<~1
\end{eqnarray*}
confirming that $Z_O$ is indeed part of the sub-differential of $\Vert *\Vert_{1,\rm off}$ evaluated at $\hat{\tilde\Theta}^S_O$, that is $\hat{\tilde\Theta}^S$ is indeed a solution to Equation~(\ref{eq:l1gq}), and by uniqueness we must have $\hat{\tilde\Theta}^S=\hat{\tilde\Theta}$.

\paragraph*{(ii)} We want to show that if $\Vert \Delta \Vert_\infty\le (3\kappa_{\Sigma}\tilde d)^{-1}$, then $ \Vert R(\Delta)_O \Vert_\infty \le  \Vert R(\Delta) \Vert_\infty \le \frac{3}{2} \tilde d\Vert \Delta \Vert^2_\infty \kappa^3_{\Sigma}$. The first inequality is due to the fact that the max-norm of a submatrix is smaller than the max-norm of the full matrix. The second inequality is more challenging, and we prove it by following the strategy used in the proof of Lemma~5 in \cite{ravikumar2011high}.

Notice that,
\begin{eqnarray*}
\big(\hat{\tilde\Theta}^S)^{-1} &=& \left(\tilde\Theta(I+\tilde\Sigma\Delta)\right)^{-1}\\
    &=&(I+\tilde\Sigma\Delta)^{-1}\tilde\Sigma\\
    &=&\sum_{k=0}^\infty(-1)^k(\tilde\Sigma\Delta)^k\tilde\Sigma\\
    &=&\tilde\Sigma-\tilde\Sigma\Delta\tilde\Sigma+\sum_{k=2}^\infty(-1)^k(\tilde\Sigma\Delta)^k\tilde\Sigma\\
    &=&\tilde\Sigma-\tilde\Sigma\Delta\tilde\Sigma+\tilde\Sigma\Delta\tilde\Sigma\Delta M\tilde\Sigma
\end{eqnarray*}
where $M=\sum_{k=0}^\infty(-1)^k(\tilde\Sigma\Delta)^k$ is convergent if $\vertiii{\tilde\Sigma\Delta}_\infty<\frac{1}{3}$. Indeed, under the assumption $\Vert \Delta \Vert_\infty\le (3\kappa_{\Sigma}\tilde d)^{-1}$ we have
\[
\vertiii{\tilde\Sigma\Delta}_\infty\le\vertiii{\tilde\Sigma}_\infty\vertiii{\Delta}_\infty\le \kappa_\Sigma\tilde d\Vert \Delta\Vert_\infty<\tfrac{1}{3}
\]
Therefore, $R(\Delta)=\tilde\Sigma\Delta\tilde\Sigma\Delta M \tilde\Sigma$, so if we let $e_k$ be a vector containing all zeros except on the $k$-th position where it equals 1, we obtain
\begin{eqnarray*}
    \Vert R(\Delta)\Vert_\infty &=& \max_{ij}|e_i\tilde\Sigma\Delta\tilde\Sigma\Delta M\tilde\Sigma e_j|\\
    &\le& \max_i\Vert e_i^T\tilde\Sigma\Delta\Vert_\infty \cdot\max_j\Vert\tilde\Sigma\Delta M\tilde\Sigma e_j\Vert_1\\
    &\le& \max_i\Vert e_i^T\tilde\Sigma\Vert_1\Vert\Delta\Vert_\infty \cdot \max_j \Vert\tilde\Sigma\Delta M\tilde\Sigma e_j\Vert_1\\
    &\le & \vertiii{\tilde\Sigma}_\infty\Vert\Delta\Vert_\infty\vertiii{\tilde\Sigma \Delta M\tilde\Sigma}_1\\
    &\le & \vertiii{\tilde\Sigma}_\infty\Vert\Delta\Vert_\infty\vertiii{\tilde\Sigma M^T \Delta \tilde\Sigma}_\infty\\
    &\le & \vertiii{\tilde\Sigma}_\infty\Vert\Delta\Vert_\infty \vertiii{\tilde\Sigma}^2_\infty\vertiii{M^T}_\infty\vertiii{\Delta}_\infty\\
    &\le& \tilde d\Vert\Delta\Vert_\infty^2\kappa_\Sigma^3\vertiii{M^T}_\infty\\
    &\le& \frac{3}{2} \tilde d\Vert\Delta\Vert_\infty^2\kappa_\Sigma^3
\end{eqnarray*}
where the last inequality is due to
\[
\vertiii{M^T}_\infty \le \sum_{k=0}^\infty\vertiii{\Delta\tilde\Sigma}_\infty^k\le \frac{1}{1-\vertiii{\tilde\Sigma}_\infty\vertiii{\Delta}_\infty}\le \frac{1}{1-1/3}=\frac{3}{2}
\]

~

\paragraph*{(iii)} We want to show that if 
\begin{equation}\label{eq:B3cond}
r:= 2\kappa_{\Gamma}(\Vert W_O \Vert_\infty + \lambda_{p,\bar n}) \le \min\left\{(3\kappa_{\Sigma}\tilde d)^{-1}, (3\kappa^3_{\Sigma}\kappa_{\Gamma}\tilde d)^{-1}\right\}
\end{equation}
then $\Vert \Delta \Vert_\infty \le r$. We prove it by following the strategy used in the proof of Lemma~6 in \cite{ravikumar2011high}. 

Define the function $G(\Theta_{S})=-[\Theta^{-1}]_{S}+\hat\Sigma_{S}+\lambda_{p,\bar n}Z_{S}$, where $S=\{(i,j):\tilde\Theta_{ij}\neq 0\}\subseteq O$. We can see that if $\hat{\tilde\Theta}^S$ is the solution of Equation~(\ref{eq:l1gq}) and $Z_{S}$ belongs to the sub-differential of $\Vert *\Vert_{1,\rm off}$ evaluated at $\hat{\tilde\Theta}^S_S$, then $G(\hat{\tilde\Theta}^S_{S})=0$. Now, define the function
\[
F(x):=-\left(\Gamma_{S,S}\right)^{-1}g(\tilde\Theta_S+x)+\Delta_S
\]
where $g(*)={\rm vec}(G(*))$, $\Delta_S:=\hat{\tilde\Theta}^S_{S}-\tilde\Theta_{S}$ is a column vector (recall that $\Delta_{S^c}\equiv 0$, so $\Delta_S$ identifies $\Delta$), and $\Gamma_{S,S}$ is a submatrix of $\Gamma=\tilde\Sigma\otimes\tilde\Sigma$ relative to the index set of pairs $(i,j)\in S$. Notice that, $F(x)=x$ if and only if $x=\Delta_S$. We now show that if $r=2\kappa_{\Gamma}(\Vert W_O \Vert_\infty + \lambda_{p,\bar n}) \le \min\left\{(3\kappa_{\Sigma}\tilde d)^{-1}, (3\kappa^3_{\Sigma}\kappa_{\Gamma}\tilde d)^{-1}\right\}$, then $F(\mathbb{B}(r))\subseteq \mathbb{B}(r):=\{x\in\mathbb{R}^{|S|}:\Vert x\Vert_\infty\le r\}$, so by Brower's Fixed Point Theorem there must be a point $x^*\in\mathbb{B}(r)$ such that $F(x^*)=x^*$. Since such point is unique, it must be $x^*=\Delta_S$, meaning that $\Vert\Delta_S\Vert_\infty\le r$ because $x^*\in\mathbb{B}(r)$. 

We have
\begin{eqnarray*}
g(\tilde\Theta_S+\Delta_S) &=& -[(\tilde\Theta+\Delta)^{-1}]_S+\tilde\Sigma_S+W_S+\lambda_{p,\bar n}Z_S
\end{eqnarray*}
If $\Delta_S\in\mathbb{B}(r)$ and $r\le(3\kappa_\Sigma\tilde d)^{-1}$, then \[\vertiii{\tilde\Sigma\Delta}_\infty \le\vertiii{\tilde\Sigma}_\infty\vertiii{\Delta}_\infty\le \kappa_\Sigma\tilde d\Vert \Delta\Vert_\infty<\tfrac{1}{3}\]
so the matrix expansion used in the proof of Lemma~\ref{convlemma}~(ii) is valid. In particular, in that proof we obtained the equation $(\tilde\Theta(I+\tilde\Sigma\Delta))^{-1}=\tilde\Sigma-\tilde\Sigma\Delta\tilde\Sigma+R(\Delta)$, which implies (after vectorization)
\[
[(\tilde\Theta+\Delta)^{-1}-\tilde\Sigma]_S+\Gamma_{S,S}\Delta_S = R(\Delta)_S
\]
Thus,
\begin{eqnarray*}
    F(\Delta_S) &=& -(\Gamma_{S,S})^{-1}g(\tilde\Theta_S+\Delta_S)+\Delta_S\\
    &=&(\Gamma_{S,S})^{-1} ([(\tilde\Theta+\Delta)^{-1}-\tilde\Sigma]_S-W_S-\lambda_{p,\bar n}Z_S)+\Delta_S\\
    &=&(\Gamma_{S,S})^{-1} ([(\tilde\Theta+\Delta)^{-1}-\tilde\Sigma]_S+\Gamma_{S,S}\Delta_S-W_S-\lambda_{p,\bar n}Z_S)\\
   &=&(\Gamma_{S,S})^{-1}  
    R(\Delta)_S-(\Gamma_{S,S})^{-1}(W_S+\lambda_{p,\bar n}Z_S)
\end{eqnarray*}
and so
\begin{eqnarray*}
    \Vert F(\Delta_S)\Vert_\infty &\le& \Vert(\Gamma_{SS})^{-1}  
    R(\Delta)_S\Vert_\infty+\Vert(\Gamma_{SS})^{-1}(W_S+\lambda_{p,\bar n}Z_S)\Vert_\infty\\
    &\le & \vertiii{(\Gamma_{SS})^{-1}}_\infty\Vert R(\Delta)_O\Vert_\infty+\vertiii{(\Gamma_{SS})^{-1}}_\infty(\Vert  W_O\Vert_\infty+\lambda_{p,\bar n}\Vert Z_S\Vert_\infty)\\
    &\le & \kappa_\Gamma (\Vert R(\Delta)_O\Vert_\infty+\Vert  W_O\Vert_\infty+\lambda_{p,\bar n})
\end{eqnarray*}
where $\Vert Z_S\Vert_{\infty}<1$. Now, if $\Vert\Delta\Vert_\infty\le r$ and Equation~(\ref{eq:B3cond}) holds, then we have $\Vert\Delta\Vert_\infty\le (3\kappa_{\Sigma}\tilde d)^{-1}$ so by Lemma~\ref{convlemma}~(ii) 
\[\Vert R(\Delta)_O\Vert_\infty\le\frac{3}{2}\tilde d\Vert\Delta\Vert_\infty^2\kappa_\Sigma^3\le \frac{3}{2}\tilde d \kappa_\Sigma^3 r^2\]
Equation~(\ref{eq:B3cond}) also implies $\Vert\Delta\Vert_\infty\le r\le (3\kappa^3_{\Sigma}\kappa_{\Gamma}\tilde d)^{-1}$, so
\[
\Vert R(\Delta)_O\Vert_\infty\le\frac{3}{2}\tilde d \kappa_\Sigma^3 r^2 \le \frac{3}{2}\tilde d \kappa_\Sigma^3 (3\kappa^3_{\Sigma}\kappa_{\Gamma}\tilde d)^{-1} r = \frac{r}{2\kappa_\Gamma}
\]
Therefore, if Equation~(\ref{eq:B3cond}) holds, then we have $\Vert R(\Delta)_O\Vert_\infty\le \frac{r}{2\kappa_\Gamma}$ and $2\kappa_{\Gamma}(\Vert W_O \Vert_\infty + \lambda_{p,\bar n}) =r$, so
\begin{eqnarray*}
    \Vert F(\Delta_S)\Vert_\infty  &\le & \kappa_\Gamma (\Vert R(\Delta)_O\Vert_\infty+\Vert  W_O\Vert_\infty+\lambda_{p,\bar n})\\
    &\le & \kappa_\Gamma (\tfrac{r}{2\kappa_\Gamma}+\tfrac{r}{2\kappa_\Gamma}) =r
\end{eqnarray*}
Thus, we have shown that for any $\Delta_S\in \mathbb{B}(r)$, with $r$ satisfying Equation~(\ref{eq:B3cond}), we have $F(\Delta_S)\in \mathbb{B}(r)$, i.e. $F(\mathbb{B}(r))\subseteq \mathbb{B}(r)$, so by Brower's Fixed Point Theorem there must be a point $x^*\in\mathbb{B}(r)$ such that $F(x^*)=x^*$. Since such point is unique, it must be $x^*=\Delta_S$, meaning that $\Vert\Delta\Vert_\infty\equiv\Vert\Delta_S\Vert_\infty\le r$.
\end{proof}
\end{lemma}

\vspace{1mm}

\section{Matrix Inequalities}\label{app:matrixineq}
This appendix contains fundamental matrix inequalities that are used in the proofs in Appendices~\ref{app:proofmains}, \ref{app:auxiliary}, and \ref{app:popcaseK2}.

\begin{lemma}\label{lemma:maxnormineq}
Let $M\in \mathbb{R}^{d\times q}$, $P\in \mathbb{R}^{q\times p}$, and $Q\in \mathbb{R}^{p\times r}$. Then
\begin{eqnarray}\label{eq:maxnormineq2}
\Vert MP\Vert_\infty &\le& \min\{{\rm rd}(M),{\rm rd}(P^T)\}\Vert M\Vert_\infty \Vert P\Vert_\infty ~\le~ q\Vert M\Vert_\infty \Vert P\Vert_\infty
\end{eqnarray}
and
\begin{eqnarray}\label{eq:maxnormineq3}
\Vert MPQ\Vert_\infty & \le &\min\{{\rm rd}(M),{\rm rd}(P^T)\}{\rm rd}(Q^T)\Vert M\Vert_\infty \Vert P\Vert_\infty \Vert Q\Vert_\infty \\ \nonumber & \le & qp\Vert M\Vert_\infty \Vert P\Vert_\infty \Vert Q\Vert_\infty
\end{eqnarray}
where ${\rm rd}(M)=\max_i|\{(i,j):M_{ij}\neq 0\}|$ denotes the max row-degree of the matrix $M$. 
Moreover, for any $p\times p$ positive definite matrix $S\succ 0$ and $U\subseteq \{1,...,p\}$ we have
\begin{equation}
\Vert S_{UU}^{-1}\Vert_\infty ~\le~ \left(\lambda_{\rm min}(S_{UU
})\right)^{-1} ~\le~ \left(\lambda_{\rm min}(S)\right)^{-1}
\end{equation}
where $\lambda_{\rm min}(S)$ is the smallest eigenvalue of $S$.
\begin{proof}
\begin{eqnarray*}
\Vert MP\Vert_\infty &=& \max_{i,j}\left|\sum_{k=1}^q M_{ik}P_{kj}\right|\\
&=& \max_{i,j}\left|\sum_{k:M_{ik}P_{kj}\neq 0} M_{ik}P_{kj}\right|\\
&\le & \max_{i,j}\sum_{k:M_{ik}P_{kj}\neq 0} \left|M_{ik}P_{kj}\right|\\
&=& \min\left\{{\rm rd}(M),{\rm rd}(P^T)\right\}\Vert M\Vert_\infty\Vert P\Vert_\infty\\
&=& q\Vert M\Vert_\infty\Vert P\Vert_\infty
\end{eqnarray*}
Moreover,
\begin{eqnarray*}
\Vert MPQ\Vert_\infty &=& \max_{i,j}\left|\sum_{h=1}^p\sum_{k=1}^q M_{ik}P_{kh}Q_{hj}\right|\\
&=& \max_{i,j}\left|\sum_{(k,h):M_{ik}P_{kh}Q_{hj}\neq 0} M_{ik}P_{kh}Q_{hj}\right|\\
&\le& \max_{i,j}\sum_{(k,h):M_{ik}P_{kh}Q_{hj}\neq 0} \left|M_{ik}P_{kh}Q_{hj}\right|\\
&\le& \max_{i,j}\sum_{(k,h):M_{ik}P_{kh}Q_{hj}\neq 0} \Vert M\Vert_\infty\Vert P\Vert_\infty\Vert Q\Vert_\infty\\
&\le& \min\left\{{\rm rd}(M),{\rm rd}(P^T)\right\}{\rm rd}(Q^T)\Vert M\Vert_\infty\Vert P\Vert_\infty\Vert Q\Vert_\infty\\
&\le & qp \Vert M\Vert_\infty\Vert P\Vert_\infty\Vert Q\Vert_\infty
\end{eqnarray*}
Finally,
\begin{eqnarray*}
    \Vert S_{UU}^{-1}\Vert_\infty ~\le~ \Vert S_{UU}^{-1}\Vert_2 ~=~ \lambda_{\rm max}(S_{UU}^{-1})~=~\left(\lambda_{\rm min}(S_{UU})\right)^{-1} ~\le~\left(\lambda_{\rm min}(S)\right)^{-1}
\end{eqnarray*}
where the first inequality is the standard relationship between max-norm and spectral norm, and the final inequality is due to the Cauchy's Interlace Theorem which guarantees $\lambda_{\rm min}(S)\le\lambda_{\rm min}(S_{UU})$.
\end{proof}
\end{lemma}

\vspace{0.5mm}

\begin{lemma}\label{lemma:schurineq}
Let $X,Y\in\mathbb{R}^{p\times p}$ be positive definite matrices. Then, for any nonempty set $A\subset \{1,...,p\}$,
\begin{equation}\label{eq:schurineq}
\Vert X/X_{A^cA^c}~ -~ Y/Y_{A^cA^c}\Vert_\infty~~\le~~\frac{\lambda_{\rm max}(X)}{\lambda_{\rm min}(X)}\frac{\lambda_{\rm max}(Y)}{\lambda_{\rm min}(Y)}\Vert X-Y\Vert_2
\end{equation}
where $X/X_{A^cA^c}=X_{AA}-X_{AA^c}(X_{A^cA^c})^{-1}X_{A^cA}$ is the Schur Complement of the block $A^c\times A^c$ of the matrix $X$, $\lambda_{\rm min}(X)$ and $\lambda_{\rm max}(X)$ are the smallest and the largest eigenvalues of $X$, and $\Vert *\Vert_2$ is the spectral norm.
\begin{proof}
We will prove the following equivalent statement in terms of $\Sigma = X^{-1}$ and $\Psi=Y^{-1}$: for any $p\times p$ positive definite matrices $\Sigma$ and $\Psi$, and any set $A\subset\{1,...,p\}$
\[
\Vert \Sigma_{AA}^{-1}-\Psi_{AA}^{-1}\Vert_\infty\le \frac{\lambda_{\rm max}(\Sigma)}{\lambda_{\rm min}(\Sigma)}\frac{\lambda_{\rm max}(\Psi)}{\lambda_{\rm min}(\Psi)}\Vert \Sigma^{-1}-\Psi^{-1}\Vert_2
\]
Notice that the following chain of (in)equalities are true
\begin{eqnarray*}
\Vert \Sigma_{AA}^{-1}-\Psi_{AA}^{-1}\Vert_\infty & \le & \Vert \Sigma_{AA}^{-1}-\Psi_{AA}^{-1}\Vert_2\\
& = & \Vert \Sigma_{AA}^{-1}(\Sigma_{AA}-\Psi_{AA})\Psi_{AA}^{-1}\Vert_2\\
& \le & \Vert \Sigma_{AA}^{-1}\Vert_2~\Vert\Sigma_{AA}-\Psi_{AA}\Vert_2~\Vert\Psi_{AA}^{-1}\Vert_2\\
& \le & \Vert \Sigma_{AA}^{-1}\Vert_2~\Vert \Sigma-\Psi\Vert_2~\Vert\Psi_{AA}^{-1}\Vert_2\\
& = & \Vert \Sigma_{AA}^{-1}\Vert_2~\Vert \Sigma(\Sigma^{-1}-\Psi^{-1})\Psi\Vert_2~\Vert\Psi_{AA}^{-1}\Vert_2\\
& \le & \Vert \Sigma_{AA}^{-1}\Vert_2~\Vert \Sigma\Vert_2~\Vert \Sigma^{-1}-\Psi^{-1}\Vert_2~\Vert \Psi\Vert_2~\Vert\Psi_{AA}^{-1}\Vert_2\\
& = & \frac{\lambda_{\rm max}(\Sigma)}{\lambda_{\rm min}(\Sigma_{AA})}\frac{\lambda_{\rm max}(\Psi)}{\lambda_{\rm min}(\Psi_{AA})}\Vert \Sigma^{-1}-\Psi^{-1}\Vert_2\\
& \le & \frac{\lambda_{\rm max}(\Sigma)}{\lambda_{\rm min}(\Sigma)}\frac{\lambda_{\rm max}(\Psi)}{\lambda_{\rm min}(\Psi)}\Vert \Sigma^{-1}-\Psi^{-1}\Vert_2
\end{eqnarray*}
where the second and fifth steps follow from observing that for any two invertible matrices $W$ and $Z$, we can write $W^{-1} - Z^{-1} = Z^{-1}(Z-W)W^{-1}$. The third and sixth steps come from the sub-multiplicativity of the spectral norm. The fourth step follows from the fact that the spectral norm of a submatrix is smaller than the $\ell_2$-norm of the full matrix. The seventh step is due to the definition of spectral norm. The eighth step is due to Cauchy's Interlace Theorem.
\end{proof}
\end{lemma}

\begin{lemma}\label{lemma:spectralinf}
For any $p\times p$ symmetric matrix $X$ with max row-degree smaller than or equal to $d$, we have 
\begin{equation}
\Vert X\Vert_2\le\min(\sqrt{\Vert X\Vert_0},d)\Vert X\Vert_\infty
\end{equation}
where $\Vert X\Vert_2$ is the spectral norm, $\Vert X\Vert_0:=|\{(i,j):X_{ij}\neq 0\}|$,  and $\Vert X\Vert_\infty$ is the max norm.
\begin{proof}
Because of the relationship between spectral norm and Frobenius norm, we have
\[
\Vert X\Vert_2\le \Vert X\Vert_F = \sqrt{\sum_{i=1}^p\sum_{j=1}^p X_{ij}^2} \le \sqrt{\Vert X\Vert_0\Vert X\Vert_\infty^2}=\sqrt{\Vert X\Vert_0}\Vert X\Vert_\infty
\]
Moreover, by Holder's Inequality and the fact that $X$ is symmetric and has max row-degree smaller than or equal to $d$
\[
\Vert X\Vert_2\le \sqrt{ \vertiii X_1 \vertiii X_\infty} =
\vertiii X_\infty \le d\Vert X\Vert_\infty
\]
where $\vertiii X_1:=\max\limits_{1\le j\le p}\sum\limits_{i=1}^p|X_{ij}|$ and $\vertiii X_\infty:=\max\limits_{1\le i\le p}\sum\limits_{j=1}^p|X_{ij}|$.
\end{proof}
\end{lemma}

\vspace{0.5mm}

\begin{lemma}\label{lemma:eigenbounds}
For any two $p\times p$ symmetric matrices $X$ and $Y$, we have
\begin{eqnarray}
\lambda_{\rm min}(X) &~\ge~& \lambda_{\rm min}(Y)-\Vert X-Y\Vert_2\\
\lambda_{\rm max}(X) &~\le~& \lambda_{\rm max}(Y)+\Vert X-Y\Vert_2
\end{eqnarray}
\begin{proof}
First of all note that
    \begin{eqnarray}
        \lambda_{\rm min}(X) &~=~& \min_{\Vert v\Vert=1} v^TXv\\
        \lambda_{\rm max}(X) &~=~&\max_{\Vert v\Vert=1} v^TXv
    \end{eqnarray}
For any unit vector $v\in\mathbb{R}^p$, we have
\begin{equation}
    v^TXv ~=~ v^T(X-Y)v+v^TYv
\end{equation}
where 
\[
\lambda_{\rm min}(Y)~\le~ v^TYv~\le~ \lambda_{\rm max}(Y)
\]
and, by Cauchy-Schwartz's Inequality,
\[
|v^T(X-Y)v|~\le~ \Vert v\Vert^2 \Vert X-Y\Vert_2 ~=~ \Vert X-Y\Vert_2
\]
that is
\[
-\Vert X-Y\Vert_2 ~\le~ v^T(X-Y)v ~\le~ \Vert X-Y\Vert_2
\]

Therefore
\begin{eqnarray*}
\lambda_{\rm min}(X) &~=~& \min_{\Vert v\Vert=1} v^TXv\\
&~=~& \min_{\Vert v\Vert=1} \left(v^T(X-Y)v+v^TYv\right)\\
&~\ge~& \min_{\Vert v\Vert=1} v^T(X-Y)v+\min_{\Vert v\Vert=1}v^TYv\\
&~\ge~& -\Vert X-Y\Vert_2+\lambda_{\rm min}(Y)
\end{eqnarray*}
and 
\begin{eqnarray*}
\lambda_{\rm max}(X) &~=~& \max_{\Vert v\Vert=1} v^TXv\\
&~=~& \max_{\Vert v\Vert=1} \left(v^T(X-Y)v+v^TYv\right)\\
&~\le~& \max_{\Vert v\Vert=1} v^T(X-Y)v+\max_{\Vert v\Vert=1}v^TYv\\
&~\le~& \Vert X-Y\Vert_2+\lambda_{\rm max}(Y)
\end{eqnarray*}
\end{proof}
\end{lemma}

 \more 
\more
 
\section{Details on simulations and figures}
\subsection{Details on simulations}\label{app:simdet}\less
\subsubsection{Classes of graphs}\label{app:simdet1}\less
 The following classes of graphs are used in simulations (Section~\ref{sec:simulations}) and illustrated in Figure~\ref{fig:sim1}A, where $p$ is the number of nodes and $d:=\max_{i\in V}|\{(i,j):\Theta_{ij}\neq 0\}|$  
denotes the graph degree, and $i< j$:
\begin{enumerate}[(i).]
\item {\it Chain}: $(i,j)$ are connected if and only if $j=i+1$ ($d=2$).
\item {\it Loop}:  $(i,j)$ are connected if and only if $j=i+1$ or $(i,j)=(1,p)$ ($d=2$).
\item {\it Star}:  $(i,j)$ are connected only if $i=k$ or $j=k$, for some fixed $k$ ($1\le d\le p$).
\item {\it Tree} (binary):  $(i,j)$ are connected if and only if $j=2i,2i+1$ ($d=3$).
\item {\it Spatial model} $S(p,w,r)$, where nodes have spatial positions $w=\{w_1$,..., $w_p\}\subset\mathbb{R}^2$, and nodes $(i,j)$ are connected if their distance $D_{ij}=\Vert w_i-w_j\Vert_2$ is below $r>0$.
\item {\it Erd\H{o}s-R\'{e}nyi model} $ ER(p,\pi) $ \citep{erdos1959random}, where the edges are randomly assigned to node pairs independently with probability $\pi$. The node degree $d$ is random and has expectation $p\pi$.
\item {\it Barab\'{a}si-Albert model} $BA(p,p_0)$ \citep{albert2002statistical}, where $p_0<p$ is the initial number of connected nodes, and the other $p-p_0$ nodes are sequentially added to the network by connecting each of them to an existing node $i$ that has node degree $d_i$ with probability proportional to $d_i/\sum_k d_k$.
\item {\it Spatial-Random model} $SR(p,w,f)$, where nodes have spatial positions $w=\{w_1,...,w_p\}$, and a pair $(i,j)$ is connected with probability $\pi_{ij} = f(D_{ij}) $, where $f$ is a decreasing function of the distance $D_{ij}=\Vert w_i-w_j\Vert_2$, e.g. $f(x)= e^{-ax}, a>0$. This model produces networks that reflect the spatial neuronal functional connectivity structure observed in some brain cortical areas \citep{vinci2018adjusted, vinci2018adjustedB}, where two neurons are more likely to be conditionally independent when physically farther apart ($D_{ij}$ large).
\end{enumerate} 
In Appendix~\ref{app:figdetails} we provide more specific details on the graphs used in the figures.

\subsubsection{Observational scheme}\label{app:simdet3}\less
The nodal sets $V_1,...,V_K\subset V$ used in the simulations of Section~\ref{sec:simulations} are given by
\begin{equation}\label{eq:simvk}
V_k=\left\{1+\left\lfloor\tfrac{k-1}{K-1}(p-q_0)\right\rfloor, ...., q_0+\left\lceil\tfrac{k-1}{K-1}(p-q_0)\right\rceil\right\},
\end{equation} 
so that $|V_k|\approx q_0, \forall k=1,...,K$, where $p/K<q_0<p$.

\subsubsection{Oracle estimation}\label{app:simdet2}\less
Throughout Section~\ref{sec:simulations} we set $\Lambda_{ij}\equiv \lambda,\forall i\neq j$ in Equation~(\ref{eq:l1gq}) and denoting the \madgqlasso{} estimator by $\hat{\tilde\Theta}(\lambda)$, in the simulations of Section~\ref{sim:rates} we pick 
\begin{equation}
\lambda^* ~=~ \arg\min_{\lambda\ge 0}~ \Vert\hat{\tilde\Theta}(\lambda)-\tilde\Theta\Vert_\infty \end{equation}
which produces the oracle \madgqlasso{} denoted by $ \hat{\tilde\Theta}(\lambda^*)$. This oracle quantity may be viewed as the best possible evaluation of $\hat{\tilde\Theta}(\lambda)$ as an estimator of $\tilde\Theta$ that could ever be achieved with any penalty parameter selection criterion that aims at minimizing the $\ell_\infty$ distortion between $\hat{\tilde\Theta}(\lambda)$ and $\tilde\Theta$.  Similarly, we denote the oracle graphical lasso by $\hat\Theta_{\rm glasso}(\lambda^{**})$, where 
\begin{equation}
\lambda^{**} ~=~ \arg\min_{\lambda\ge 0}~ \Vert\hat\Theta_{\rm glasso}(\lambda)-\Theta\Vert_\infty
\end{equation}

\more

\subsection{Details on figures}\label{app:figdetails}
\subsubsection{Figure~\ref{fig:intro}} The graph shown in panel (C) is a spatial random graph (Appendix~\ref{app:simdet1}, model (viii)) with connection probability function $f(\omega)=\exp(-\omega)$ and node spatial positions over the discrete grid $\{1,...,10\}\times\{1,...,10\}\subset\mathbb{R}^2$.

\subsubsection{Figure~\ref{fig:poprecovery}}
The matrix $\Theta$ used in Figure~\ref{fig:poprecovery} has graph structure generated as an Erd\H{o}s-R\'enyi graph $ER(p=40,\pi=0.05)$. Moreover, the diagonal entries of $\Theta$ are all equal to 1, while the nonzero off-diagonals have magnitudes all equal to $p^{-1}$; 25\% of the nonzero entries are positive and 75\% are negative (correspondingly, 75\% of nonzero partial correlations are positive).

\subsubsection{Figures~\ref{fig:sim1}} All ground truth precision matrices used in simulations have graphical structures listed in Appendix~\ref{app:simdet1}, and have diagonals equal to 1, while the nonzero off-diagonals have magnitudes all equal to $p^{-1}$. Moreover, 25\% of the nonzero entries are then set to be positive and 75\% are set to be negative (correspondingly, 75\% of nonzero partial correlations are positive). Furthermore, the Star graphs have degree $ d=\lceil p/4 \rfloor$; the Erd\H{o}s-R\'enyi graphs were generated with $\pi=0.01$; the Barab\'{a}si-Albert graphs were generated using $p_0=1$; the Spatial Random graphs were generated using $f(x)=\exp(-2x)$ with nodes occupying $p$ positions on the grid $\{1,...,\lceil \sqrt{p}\rceil\}\times \{1,...,\lceil \sqrt{p}\rceil\}$. The columns/rows of every precision matrix were permuted to ensure $E_{O^c}\neq\emptyset$.

\subsubsection{Figures~\ref{fig:sim2}} The Erd\H{o}s-R\'enyi graphs were generated with $\pi=2/p$. All ground truth precision matrices have diagonals equal to 1, while the nonzero off-diagonals have magnitudes all equal to $0.2$. Moreover, 25\% of the nonzero entries are then set to be positive and 75\% are set to be negative (correspondingly, 75\% of nonzero partial correlations are positive).

\subsubsection{Figures~\ref{fig:data1}}
The data \cite{stringer2019spontaneous} can be found at {\small \url{https://figshare.com/articles/dataset/Recordings_of_ten_thousand_neurons_in_visual_cortex_during_spontaneous_behaviors/6163622}}, file  \\ ``spont\_M161025\_MP030\_2016-11-20.mat''. The AUC that assesses the similarity between the GQ graphs and the Glasso graphs with given number of edges is computed by varying the input parameters of Algorithm~\ref{algo:fullrecoveryKn}, $\tau_0,\tau,\tau_1$, and $\Lambda$, with $\Lambda_{ij}=\lambda$ for all $(i,j)$. 

\subsubsection{Figure~\ref{fig:poprecoveryABCapp}}
The matrix $\Theta$ used in Figure~\ref{fig:poprecoveryABCapp} has graph structure generated as an Erd\H{o}s-R\'enyi graph $ER(p=40,\pi=0.05)$. Moreover, the diagonal entries of $\Theta$ are all equal to 1, while the nonzero off-diagonals have magnitudes all equal to $p^{-1}$, except the ones in $O^c$ which are set equal to $(4p)^{-1}$; 25\% of the nonzero entries are positive and 75\% are negative (correspondingly, 75\% of nonzero partial correlations are positive).

~

All computations for simulations, data analyses, and figures were implemented in R.

\newpage
\section{Population Analysis in the special case \texorpdfstring{$K=2$}{Lg}}\label{app:popcaseK2}\less
In this appendix we restate several of our results for the special case where we observe only two vertex subsets $V_1\neq V_2$, that is $O=(V_1\times V_1)\cup (V_2\times V_2)$. 
For this case, the \madgq{} optimization problem in Equation~(\ref{eq:gq0}) has a tractable closed-form solution (Equation~(\ref{thetaMD3x3})), which allows us to analyse the Graph Quilting problem in greater detail analytically. For simplicity of exposition, we shall let $V_1=A\cup B$ and $V_2=B\cup C$, where $A,B,C$ is a partition of $V$, so that $O^c=(A\times C)\cup (C\times A)$ and $B$ contains the overlapping vertices between the two observation sets. The first three subsections of this appendix are organized similarly to Section~\ref{sec:popK}. We investigate the graph recovery in $O$ and in $O^c$ separately, in Appendices~\ref{sec:OABC} and \ref{sec:OcABC}, and then condense the results into one algorithm for the recovery of the full graph in Appendix~\ref{sec:fullABC}. Note that our results on the graph recovery in $O$ in the special case $K=2$ are already stated in Section~\ref{sec:popO}, but we restate them here for completeness.

\subsection{Graph recovery in \texorpdfstring{$O$}{Lg}}\label{sec:OABC}\less
Lemma~\ref{lemma:exactO} states that if $\delta<\nu/2$ then we can recover the edge set and signs in $O$ exactly by simply thresholding the entries of $\tilde\Theta_O$ at level $\tau\in [\delta,\nu-\delta)$. We identify three situations where the condition $\delta<\nu/2$ is satisfied, and we specify them in terms of $\gamma$:

\begin{enumerate}
    \item[(A1).] $E_{AC}=\emptyset$, i.e. $\gamma=0$.
    \item[(A2).] $B$ is disconnected from $A$ and $C$ and $0<\gamma <\sqrt{\frac{\nu\lambda_{\min}}{2d_{O^c}^2}} $, where $\lambda_{\min}$ is the smallest eigenvalue of $\Theta$, and $d_{O^c}$ is the max node-degree in the sub-graph $E_{AC}$.

    \item[(A3).] $0<\gamma<\frac{-b+\sqrt{b^2+2a\nu}}{2a}$, where $a=d_{O^c}^2(\lambda_{\min}^{-1}+2q^2d_B^2\gamma_B^2\lambda_{\min}^{-3})$, $b=d_{O^c}(d_B\gamma_B\lambda_{\min}^{-1}+2qd_B^2\gamma_B^2\lambda_{\min}^{-2})$, 
 $q=\max\{|A|,|C|\}$, $d_B$ is the largest number of edges from one node in $B$ to $A$ or to $C$, and $\gamma_B=\Vert\Theta_{B(AC)}\Vert_\infty$. 
\end{enumerate}

The following theorem (identical to Corollary~\ref{coro:popOABC})  states conditions for the exact recovery of $E_O$ in the special case $K=2$:

\begin{theorem}[\scbf{Exact Graph Recovery in $O$ ($K=2$)}]\label{theo:popOABCapp} If Condition (A1) or (A2) or (A3) hold, then $\delta <\nu/2$  and for any $\tau\in [\delta,\nu-\delta)$, we have $\tilde E_O^{\tau} = E_O$ (Equation~(\ref{eq:Etau})), and ${\rm sign}(\tilde\Theta_{ij})={\rm sign}(\Theta_{ij})$,  $\forall(i,j)\in E_O$.
\end{theorem}

Condition (A1) corresponds to the simplest situation depicted by Theorem~\ref{thm:identifiability}, where  $E_{AC}=\emptyset$ guarantees $\tilde\Theta=\Theta$, yielding $\delta=0<\nu/2$. Conversely, conditions (A2) and (A3) exploit several, rather technical, matrix inequalities given in Appendix~\ref{app:graphrecoverypop},  Appendix~\ref{app:auxiliary}, and Appendix~\ref{app:matrixineq},
which explicitly relate the magnitude $\gamma$ of the strongest edge in $O^c$ to the other quantities characterizing $\Theta$. Overall, however, we can see that the exact graph recovery in $O$ is easier to accomplish when the magnitude $\nu$ of the weakest edge in $O$ and the smallest eigenvalue $\lambda_{\rm min}$ of $\Theta$ are large, while the size $q$ and maximum node degree $d_{O^c}$ in $O^c$ are small. Finally, note that (A3) reduces to (A2) if $\gamma_B\to 0$.

\begin{figure}[t!]
\center
\includegraphics[width=1\textwidth]{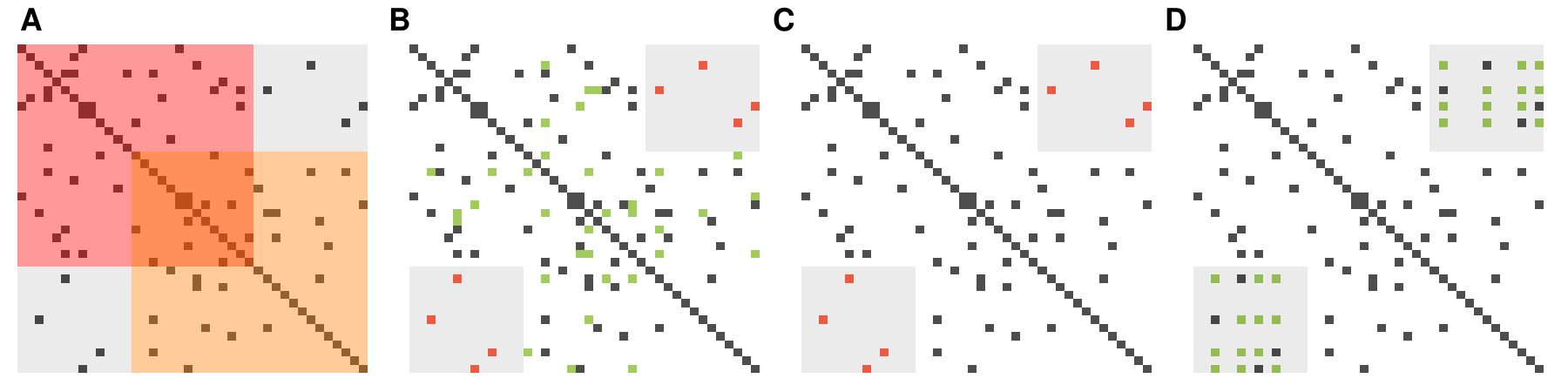}\vspace*{-3mm}
\caption{Example of graph recovery via \madgq{}. 
{\bf (A)} Support of $\Theta$ ($p=40$) and observed node pairs set $O$ (colored area), and $O^c$ (grey area). 
{\bf (B)} Support of the \madgq{} matrix $\tilde\Theta$: the green entries denote false positive edges, whereas the red ones are false negatives. No false negatives are in $O$ as per Theorem~\ref{theo:nofnd}.
{\bf (C)} Recovered graph $\tilde E^{\tau}$ (Equation~(\ref{eq:Etau})) with $\tau=\nu/2\in[\delta,\nu-\delta)$. In this example, the largest magnitude of $\Theta_{O^c}$, $\gamma$, is sufficiently small, so Condition~(A3) is satisfied, and $\tilde E^\tau_O$ perfectly match the true edge set $E_O$, as per Theorem~\ref{theo:popOABCapp}. 
{\bf (D)} Fully recovered edge set $\mathcal{E}^\tau$ via Algorithm~\ref{algo:fullrecoveryK2}, consisting of the union of the recovered graph in $O$ as in (C), with the minimal superset of edges in $O^c$.}
\label{fig:poprecoveryABCapp}
\end{figure}

In Figure~\ref{fig:poprecoveryABCapp} (analogous to Figure~\ref{fig:poprecovery}, in Section~\ref{sec:popK}) we present an example with $p=40$ nodes where the precision matrix $\Theta$ in panel (A) (see Appendix~\ref{app:figdetails} for more details) contains four edges in $O^c$, but the largest magnitude $\gamma = 0.00625$ of these edges is small enough to ensure $\delta<\nu/2$. Indeed, since $\nu=0.025$, $\lambda_{\rm min}=  0.93688$, $q=13$, and $d_{O^c}=1$, and $d_B=2$, we have $\gamma=0.00625<\frac{-b+\sqrt{b^2+2a\nu}}{2a}=0.0526$, so Condition~(A3) is satisfied, which ensures $\delta<\nu/2$. In Figure~\ref{fig:poprecoveryABCapp}(B) we display the support of the \madgq{} matrix $\tilde\Theta$, which contains several false positives in $O$ (green), four false negatives in $O^c$ (red), but no false negatives in $O$, in agreement with Theorem~\ref{theo:nofnd}. Finally, in Figure~\ref{fig:poprecoveryABCapp}(C) we plot the edge set $\tilde E^{\tau}$ (Equation~(\ref{eq:Etau})) with $\tau=\nu/2\in[\delta,\nu-\delta)$, which perfectly matches the graph over the set $O$ as per Theorem~\ref{theo:popOABCapp}, since all false positives had magnitudes smaller than the threshold $\nu/2$. Figure~\ref{fig:poprecoveryABCapp}(D) is discussed in Appendix~\ref{sec:OcABC}.\less

\subsection{Recovery in \texorpdfstring{$O^c$}{Lg} via Oracle Distortions in \texorpdfstring{$O$}{Lg}}\label{sec:OcABC}\less
This section is organized as Section~\ref{sec:popOc}, but with results for the special cases $K=2$. We first study how the distortions propagate across the entries of $\tilde\Theta_O$ depending on the precise position of the edges in $O^c=(A\times C)\cup (C\times A)$. Note that, because of the symmetry of the set $E_{O^c}$, for convenience we will often state our results just in terms of the portion $A\times C$.  We then introduce the definition of minimal superset of the edge set $E_{O^c}$ based on the oracle knowledge of the distortions. The oracle results presented in this subsection, while impractical, constitute the theoretical foundations of the more practical approach proposed in Appendix~\ref{sec:fullABC}, where the oracle knowledge of the distortions is not required. 

\paragraph*{Notation} The following graph theoretic terminology will help characterize the graph recovery in $O^c$. This notation is the same as the one in Section~\ref{sec:popOc}, and we restate it here for convenience. Let $U \subseteq V$ be an arbitrary subset of nodes and let $G_U$ denote the subgraph of $G$ induced by $U$, i.e., the graph whose vertex set is $U$ and edge set is $E \cap (U\times U)$; indeed, $G_V = G$. We will let $N(i):=\{j\in V: (i,j)\in E\}$ denote the neighborhood of $i$. Two nodes $i$ and $j$ are neighbours (a.k.a. adjacent) if $i\in N(j)$, or equivalently, $j\in N(i)$. We further let $N_U(i):=N(i)\cap U$ be the set of neighbours of $i$ that are in $U$. Given two subsets $U,F\subseteq V$, we let $N_U(F) := \bigcup_{i\in F}N_U(i)\subseteq U$ be the set of nodes in $U$ that are neighbours of one or more nodes in $F$. Two nodes $i,j\in U$ are $U$-connected if they are  connected through some path completely within $U$.

\subsubsection{Distortion propagation}\less 
We now investigate how the distortions propagate in $\tilde\Theta_O$ when the assumption $\Theta_{O^c}=0$ is incorrect, in the special case $K=2$. We will say that the $(i,j)$ entry of $\tilde\Theta$ is distorted if the difference
\begin{equation}\label{eq:deltaij}
\delta_{ij} ~:= ~\Theta_{ij}-\tilde\Theta_{ij}
\end{equation}
is nonzero. The following theorem, analogous to Theorem~\ref{theo:distprK}, precisely describes the relationship between distortions and edges in $O^c$ in the special case $K=2$:\more

\begin{theorem}[\scbf{Distortion Propagation  ($K=2$)}]\label{theo:distpropABC}
We have: 
\begin{enumerate}[(i).]
\item For $i\in A$ and $j\in C$, 
\begin{eqnarray}\label{eq:deltaiiABC}
\delta_{ii}> 0 ~~~\Leftrightarrow~~~ \Theta_{iC}\neq 0\\
\delta_{jj}> 0 ~~~\Leftrightarrow~~~ \Theta_{jA}\neq 0
\end{eqnarray}
\item For $(i,j)\in A\times (A\cup B)$, $i\neq j$, we have [a.e.]
\begin{equation}\label{eq:deltaAB}
\delta_{ij}\neq 0 ~\Leftrightarrow~ \exists h\in N_C(i) \text{ and } \exists k\in N_C(j) \text{ s.t. } h=k \text{ or } h \text{ is } C\text{-connected to } k.
\end{equation}
For $(i,j)\in C\times (B\cup C)$, $i\neq j$, we have [a.e.]
\begin{equation}\label{eq:deltaBC}
\delta_{ij}\neq 0 ~\Leftrightarrow~ \exists h\in N_A(i) ~\text{and}~ \exists k\in N_A(j) \text{ s.t. } h=k \text{ or } h \text{ is } A\text{-connected to } k.
\end{equation}
\item
For any $(i,j) \in A\times (A\cup B)$, $i< j$, if $\delta_{ij}\neq 0$ then $\delta_{ii}>0$.\\
For any $(i,j) \in C\times (B\cup C)$, $j<i$, if $\delta_{ij}\neq 0$ then $\delta_{ii}>0$.
\end{enumerate}
\end{theorem}\more

Part~(i) of Theorem~\ref{theo:distpropABC} states that there is a distortion on the diagonal entry of node $i\in A$ if and only if node $i$ is connected to some node in $C$; similarly, there is a distortion on the diagonal entry of node $j\in C$ if and only if node $j$ is connected to some node in $A$. Note that a diagonal distortion $\delta_{kk}$, $k\in A\cup C$, is always nonnegative, so $\delta_{kk}>0$ is the only possible kind of diagonal distortion. Part~(ii) states that an off-diagonal entry $(i,j)\in A\times (A\cup B)$ of the \madgq{} matrix $\tilde\Theta$ is distorted if and only if nodes $i$ and $j$ are connected through some path of length $>1$ completely within $C\cup\{i,j\}$. Similarly, an off-diagonal entry $(i,j)\in C\times (B\cup C)$ of the \madgq{} matrix $\tilde\Theta$ is distorted if and only if nodes $i$ and $j$ are connected through some path of length $>1$ completely within $A\cup\{i,j\}$. Finally, part~(iii) reveals that an off-diagonal entry $(i,j)$ in the portions $A\times (A\cup B)$ or $C\times (B\cup C)$ of $\tilde\Theta$ is distorted only if the corresponding diagonal entry $(i,i)$ in $A\times A$ or $C\times C$ is also distorted. The following corollary, analogous to Corollary~\ref{coro:distpropagK}, highlights other important implications of Theorem~\ref{theo:distprK} in the special case $K=2$:\more
\begin{corollary}\label{coro:distpropagABC}
Let $(i,j)\in A\times C$, $k\in A\cup B\setminus\{i\}$, and $h\in B\cup C\setminus\{j\}$. Then
\begin{enumerate}
    \item[(i).] $\Theta_{ij}\neq 0\Rightarrow \delta_{ii},\delta_{jj}> 0$.
    \item[(ii).] $\Theta_{ij}\neq 0\Rightarrow \delta_{ik}\neq 0$, for all $k$ where $k$ is $(C\cup\{k\})$-connected to $j$ [a.e.].\\
 $\Theta_{ij}\neq 0\Rightarrow \delta_{hj}\neq 0$, for all $h$ where $h$ is $(A\cup\{h\})$-connected to $i$ [a.e.].
\end{enumerate}
\end{corollary}
Corollary~\ref{coro:distpropagABC} states that if nodes $i\in A$ and $j\in C$ are neighbours, then, in all situations, both $\tilde\Theta_{ii}$ and $\tilde\Theta_{jj}$ will be distorted entries of the \madgq{} matrix $\tilde\Theta$.  On the other hand, an off-diagonal distortion is generated on the row $i$ of the portion $A\times (A\cup B)$ as long as at least one node $k\in A\cup B\setminus \{i\}$ is either a neighbour of $j$ or a neighbour of a node in $C$ that is $C$-connected to $j$. Similarly, to have an off-diagonal distortion on the column $j$ of the portion $(B\cup C)\times C$ of $\tilde\Theta$, we need that at least one node $h\in B\cup C\setminus \{j\}$ is either neighbour of $i$ or neighbour of a node in $A$ that is $A$-connected to $i$. Hence, if there is no such node $k$ or $h$, there will be no off-diagonal distortion on the row $i$ of the portion $A\times (A\cup B)$ or in column $j$ of the portion $ (B\cup C)\times C$. \less\less

\subsubsection{Superset minimality}\label{sec:minimality}\less\less
In this section we establish that with incomplete covariance information it is at least possible to recover a minimal superset of $E_{O^c}$ by exploiting the types of distortions considered in the Distortion Propagation Theorem~\ref{theo:distpropABC}. The minimal superset is defined as follows:
\begin{definition}[\scbf{Minimal Superset of $E_{O^c}$ ($K=2$)}]\label{def:minimality}
Let 
\begin{equation}
\mathcal{D}_Q(\Sigma, O):=\left\{(i,j)\in Q: \delta_{ij}\neq 0\right\}
\end{equation}
be the set of known distortions in $\tilde\Theta_O$, where $Q\subseteq O$, and let 
\begin{equation}\label{eq:admissABC}
\mathcal{A}(\Sigma,O,Q):=\left\{\Sigma'\succ 0: ~\Sigma'_O=\Sigma_O,~ \mathcal{D}_Q(\Sigma',O)=\mathcal{D}_Q(\Sigma,O)\right\}
\end{equation}
be the set of all positive definite covariance matrices that agree with the observed $\Sigma_O$ and distortions $\mathcal{D}_Q(\Sigma, O)$. A set $\mathcal{S}$ is the minimal superset of $E_{O^c}$ with respect to $\Sigma_O$ and $\mathcal{D}_Q(\Sigma,O)$ if it satisfies the following properties:
\begin{enumerate}
\item[(i).] $\forall\Sigma'\in\mathcal{A}(\Sigma,O,Q)$ we have $E_{O^c}'\subseteq \mathcal{S}$;
\item[(ii).] $\forall\mathcal{S}' \subsetneq \mathcal{S}$, $\exists\Sigma'\in\mathcal{A}(\Sigma,O,Q)$ such that $E'_{O^c} \cap (\mathcal{S}\setminus \mathcal{S}') \neq \emptyset$.   
\end{enumerate}
\end{definition}

Thus, a minimal superset $\mathcal{S}$  of $E_{O^c}$ given the set of known (oracle) distortions $\mathcal{D}_Q(\Sigma,O)$ is the smallest possible superset in the sense that it includes all plausible graphical structures $E_{O^c}$ that would induce the same known (oracle) distortions in the \madgq{} Schur complements. Thus, any other set $\mathcal{S}'\neq\mathcal{S}$ is either not a superset of $E_{O^c}$, or it is larger than $\mathcal{S}$, or it does not include one or more plausible edges. An expression of the minimal superset defined in Definition~\ref{def:minimality} is given by\less
\begin{equation}
\mathcal{S} := \bigcup\limits_{\Sigma'\in \mathcal{A}(\Sigma,O,Q)}\left\{(i,j)\in O^c:~ \left[\Sigma'^{-1} \right]_{ij}\neq 0\right\}
\end{equation}\less

In the following, we will consider the cases where we have oracle knowledge of all distortions on the diagonal entries or on the off-diagonals.\less\less

\subsubsection{Oracle minimal superset recovery}\less\less
Towards the statement of our main Theorem~\ref{theo:abcOc} for the oracle recovery of $E_{O^c}$, let us first define some quantities. Define the set\less
\begin{equation}\label{eq:supersetAC}
\mathcal{S}_{\rm diag} ~:=~ O^c\cap (D_{\rm diag}\times D_{\rm diag})
\end{equation}\less
where $D_{\rm diag}=\{i\in A\cup C: ~\delta_{ii}>0\}$ is the set of nodes in $A\cup C$ with a diagonal distortion. 
Moreover, define the set 
\begin{equation}\label{eq:minsupoff}
    \mathcal{S}_{\rm off} ~:=~ O^c\cap (D_{\rm off}\times D_{\rm off})
\end{equation}
where $D_{\rm off}= \{i\in A\cup C:\exists k\neq i, (i,k)\in O, \delta_{ij}\neq 0\}$ is the set of nodes in $A\cup C$ that are incident to at least one distorted edge (e.g. a false positive) in $O$. Furthermore, consider the following assumption:\less
\begin{enumerate}
\item[(A4$^\star$).] For every $i\in N_A(C)$, there exists at least one node $k\in (A\cup B)\setminus\{i\}$ that is $(C\cup\{k\})$-connected to some node in $N_C(i)$, and for every $j\in N_C(A)$, there exists at least one node $h\in (B\cup C)\setminus\{j\}$ that is $(A\cup\{h\})$-connected to some node in $N_A(j)$.\less
\end{enumerate}

\noindent We are now ready to state our main theorem for the oracle recovery of $E_{O^c}$:\more

\begin{theorem}[\scbf{Oracle Minimal Superset of $E_{O^c}$ ($K=2$)}]\label{theo:abcOc} 
Let $\mathcal{S}_{\rm diag}$ and $\mathcal{S}_{\rm off}$ be the sets in Equations~(\ref{eq:supersetAC}) and (\ref{eq:minsupoff}). Then, in the sense of Definition~\ref{def:minimality}:
\begin{enumerate}[(i).]
\item The set $\mathcal{S}_{\rm diag}$ is the minimal superset of $E_{O^c}$ given the set of diagonal distortions $\mathcal{D}_{\rm diag}(\Sigma,O)=\{(i,i): ~i\in A\cup C, ~\delta_{ii}>0\}$.
\item The set $\mathcal{S}_{\rm off}$ equals $\mathcal{S}_{\rm diag}$ [a.e.] and is the minimal superset of $E_{O^c}$  given the set of off-diagonal distortions $\mathcal{D}_{\rm off}(\Sigma,O):=\{(i,j)\in O\setminus (B\times B): i\neq j, ~\delta_{ij}\neq 0\}$ if and only if assumption (A4) holds.
\end{enumerate}
\end{theorem}

Part {\it (i)} of the theorem establishes that $\mathcal{S}_{\rm diag}$ is the minimal superset of $E_{O^c}$ in the sense of Definition~\ref{def:minimality}, based on the knowledge of the distortions on the diagonal entries $\{(i,i):i\in A\cup C\}$ of $\tilde\Theta$. Part {\it (ii)} of the theorem establishes that, under Assumption~(A4$^\star$), we have $\mathcal{S}_{\rm off}=\mathcal{S}_{\rm diag}$ [a.e.]. Indeed, assumption~(A4$^\star$) guarantees that for every row in $A\times C$ containing an edge (i.e. every $i\in N_A(C)$) we will have at least one off-diagonal distortion on the same row in the portion $A\times (A\cup B)$ of $\tilde\Theta$; and for every column in $A\times C$ containing an edge (i.e. every $j\in N_C(A)$) we will have at least one off-diagonal distortion on the same column in the portion $(B\cup C)\times C$ of $\tilde\Theta$. In other words, condition (A4$^\star$) guarantees that all diagonal distortions can be retrieved from the off-diagonal ones [a.e.], so that, in Equations~(\ref{eq:supersetAC}) and (\ref{eq:minsupoff}) $D_{\rm off}=D_{\rm diag}$ and thereby $\mathcal{S}_{\rm off}=\mathcal{S}_{\rm diag}$. \more

In this special case $K=2$, we can provide additional results about the minimal superset $\mathcal{S}_{\rm diag}$ (same as $\mathcal{S}_{\rm off}$ [a.e.] under Assumption~(A4$^\star$)):

\begin{lemma}[\scbf{Cardinality of plausible edge sets in $O^c$ ($K=2$)}]\label{lemma:cardplausibABC}
Let $m=\min\{|A\cap D_{\rm diag}|,|C\cap D_{\rm diag}|\}$ and $M=\max\{|A\cap D_{\rm diag}|,|C\cap D_{\rm diag}|\}$. Then, 
\begin{enumerate}
\item[(a)] $M\le|E_{AC}|\le mM$.
\item[(b)] The number of plausible graph structures in $O^c$ with $\kappa$ edges contained in the minimal superset $\mathcal{S}_{\rm diag}$ in Equation~(\ref{eq:supersetAC}) is equal to
\begin{equation}\label{eq:xik}
\xi_\kappa :=~ \binom{mM}{\kappa}~-~\sum_{j=1}^M(-1)^{j+1}\hspace{-4mm}\sum_{
\tiny\begin{array}{c}
0\le x\le m\\
0\le y\le M\\
s.t.~x+y=j
\end{array}
}\hspace{-4mm}\binom{m}{x}\binom{M}{y}\binom{(m-x)(M-y)}{\kappa}
\end{equation}
This number $\xi_\kappa$ can be much smaller than $\varphi_\kappa:=\binom{|A||C|}{\kappa}$, which is the number of graphs with $\kappa$ edges that may be contained in $O^c$.
\end{enumerate}
\end{lemma}
The properties of $\mathcal{S}_{\rm diag}$ highlighted by Lemma~\ref{lemma:cardplausibABC} are due to the fact that the Cartesian product $\mathcal{S}_{\rm diag}\cap(A\times C)=(A\cap D_{\rm diag})\times (C\cap D_{\rm diag})$ is the minimal superset of $E_{AC}$, so it must contain at least one true edge in every row and in every column. Thus, the number of edges in $O^c$ is bounded between $M$ and $mM$. Moreover, $E_{AC}$ is one of the several plausible subgraphs that would have induced distortions on the same observed diagonal positions. Equation~(\ref{eq:xik}) gives us the exact number $\xi_\kappa$ of all such plausible subgraphs in $O^c$ if the number of edges in $O^c$ was $\kappa$. Most interestingly, the number $\xi_\kappa$ can be very much smaller than the number $\varphi_\kappa=\binom{|A||C|}{\kappa}$ of all possible graph structures of $\kappa$ edges connecting nodes in $A$ to nodes in $C$. For example, if $|A|=|C|=7$, $m=4$ and $M=5$, then $\xi_5=240<\varphi_5=1,906,884$.

An important question one may ask is: Of all the plausible graphical structures contained in $\mathcal{S}_{\rm diag}$, which one could we pick? If $m=1$, then $\mathcal{S}_{\rm diag}=E_{O^c}$, i.e. $\mathcal{S}_{\rm diag}$ recovers $E_{O^c}$ exactly. In the case $m>1$, several criteria may be chosen. We could decide to take the densest possible graph, which is $\mathcal{S}_{\rm diag}$. Alternatively, we could assume $E_{AC}$ is the sparsest possible, i.e. $|E_{AC}|=M$, and then follow some criterion to pick one of the $\xi_M$ plausible graphs. Possible criteria include: (a) randomly pick one of the plausible graphs (if $|E_{AC}|=M$ truly, then we would pick it with probability $1/\xi_M$, possibly much larger than the probability $1/\varphi_M$ of picking it from the set of all possible subgraphs in $A\times C$); (b) pick the subgraph that produces the smallest number of paths in the full graph; (c) pick the subgraph that produces the smallest maximum node degree; or (d) pick the graph that yields the smallest average path length. We leave the investigation of these criteria as future research.\more

\subsection{Full graph recovery}\label{sec:fullABC}\less
We now condense the results of Appendices~\ref{sec:OABC} and \ref{sec:OcABC} into one algorithm, Algorithm~\ref{algo:fullrecoveryK2}, for the recovery of the full edge set $E$. This algorithm does not require the oracle knowledge of the distortions for the recovery of the edges in $O^c$, but instead it only exploits the off-diagonal entries in the \madgq{} matrix $\tilde\Theta$ that are identified as distorted because their magnitudes are too small. Theorem~\ref{theo:fullrecoveryK2} establishes the properties of the output edge set $\mathcal{E}^\tau$ of Algorithm~\ref{algo:fullrecoveryK2}, and requires the following assumption:\more
\begin{enumerate}
    \item[(A5$^\star$).] For every $(i,j)$ such that $\delta_{ij}\neq 0$,  there exists $ h\neq i$ such that $0<|\tilde\Theta_{ih}|<\delta$.
\end{enumerate}\more

\begin{theorem}[\scbf{GQ Graph recovery (population case, $K=2$)}]\label{theo:fullrecoveryK2}
If assumptions (A4$^\star$)-(A5$^\star$) and one of (A1), (A2), or (A3) hold, then, for any $\tau\in [\delta,\nu-\delta)$, the output edge set $\mathcal{E}^\tau$ of Algorithm~\ref{algo:fullrecoveryK2} satisfies  $\mathcal{E}_O^\tau=E_O$ and $\mathcal{E}_{O^c}^\tau=\mathcal{S}_{\rm diag} $, where $\mathcal{S}_{\rm diag}$ is the minimal superset of $E_{O^c}$  in Equation~(\ref{eq:supersetAC}).
\end{theorem}
Theorem~\ref{theo:fullrecoveryK2} combines Theorem~\ref{theo:popOABCapp} and Theorem~\ref{theo:abcOc}. Each of the assumptions (A1), (A2), (A3) alone guarantees that $\delta<\nu/2$ so, for any $\tau\in[\delta,\nu-\delta)$, the thresholded edge set $\mathcal{E}_O^\tau\equiv \tilde E_O^\tau$ equals the true edge set $E_O$, as per Theorem~\ref{theo:popOABCapp}. This means that no off-diagonal entry of $\Theta_O$ has magnitude in the interval $(0,\tau]$. Hence, if $0<|\tilde\Theta_{ij}|<\tau$, then $\delta_{ij}\neq 0$. Thus, under Assumption~(A5$^\star$), the set $W_\tau$ in Algorithm~\ref{algo:fullrecoveryK2} contains every node $i$ that is associated with at least one off-diagonal distortion. Therefore, $W_\tau$ matches the set $D_{\rm off}$ in Equation~(\ref{eq:minsupoff}) and thereby $\mathcal{E}_{O^c}^\tau\equiv\mathcal{S}_{\rm off}$ where, under Assumption (A4$^\star$), $\mathcal{S}_{\rm off}\equiv 
\mathcal{S}_{\rm diag}$ by Theorem~\ref{theo:abcOc}. An example of full graph recovery is shown in Figure~\ref{fig:poprecoveryABCapp}(D).

\begin{algorithm}[t]\normalsize
\scbf{Input}: $A,B,C\subset V$, $\Sigma_O$, $\tau>0$\;
 \begin{enumerate}
    \item  Compute the \madgq{} matrix
    \[
    \tilde\Theta ~=~  \underset{\Theta\succ 0, ~\Theta_{O^c}=0}{\arg\max}~ \log\det \Theta - \sum_{(i,j)\in O}\Theta_{ij}\Sigma_{ij}
    \]
    \item Find the edge set $\tilde E_O^\tau=\left\{(i,j)\in O: i\neq j, |\tilde\Theta_{ij}|>\tau\right\}
    $.
    \item Obtain the node set
    \[W_\tau=\left\{i\in A\cup C:~\exists j\neq i, 0<|\tilde\Theta_{ij}|< \tau\right\}
    \]
    \item Obtain the set $\mathcal{U}_\tau=O^c\cap(W_\tau\times W_\tau)$.
\end{enumerate}
\scbf{Output}:  Edge set 
\begin{equation}\label{eq:fulledegqpopABC}
\mathcal{E}^\tau ~=~ \tilde E_O^\tau\cup \mathcal{U}_\tau\vspace{-6mm}
\end{equation}
\caption{\scbf{GQ graph recovery (population case, $K=2$)}}\label{algo:fullrecoveryK2}
\end{algorithm}

Finally, it is worthwhile to mention that in special situations, Algorithm~\ref{algo:fullrecoveryK2} can be used to deal with the general case $K\ge 2$. We can proceed as follows: (a) find a partition $A,B,C\subseteq V$ such that $O^c=(\cup_{k=1}^K(V_k\times V_k))^c \subseteq \Omega^c:=(A\times C)\cup (C\times A)$, and then (b) apply Algorithm~\ref{algo:fullrecoveryK2} based on the smaller portion $\Sigma_\Omega$ in place of $\Sigma_{O}$. However, this strategy is sub-optimal because the information carried by the observed set of covariances $\Sigma_{O\setminus\Omega}$ is discarded; also, a suitable partition $A,B,C\subseteq V$ satisfying $O^c\subseteq \Omega^c$ does not always exist. Algorithm~\ref{algo:fullrecoveryK} for the general case $K\ge 2$ in  Section~\ref{sec:popOc} is optimal and applies to all situations. 

\more

\subsection{Auxiliary results}\label{app:abcaux}
This section contains two lemmas establishing important properties of the \madgq{} matrix $\tilde\Theta$ (Equation~(\ref{eq:gq0})) in the special case $K=2$: Lemma~\ref{lemma:mdbias} presents $\tilde\Theta$ as an explicit function of $\Theta$, and Lemma~\ref{lemma:linfinityABCbounds} provides explicit bounds on the $\ell_\infty$ distortions between $\Theta$ and $\tilde\Theta$. These lemmas are widely used in the proofs of the results established in this Appendix~\ref{app:popcaseK2}.
\more

\begin{lemma}[\scbf{\madgq{} as a function of $\Theta$ ($K=2$)}]\label{lemma:mdbias}
Suppose $V_1=A\cup B$ and $V_2=B\cup C$, where $A,B,C$ is a partition of $V$, and $O^c=(A\times C)\cup (C\times A)$. The MAD$_{GQ}$ solution $\tilde\Theta$ in Equation~(\ref{eq:gq0}) has components:

\begin{eqnarray}
\tilde\Theta_{AC} &=& 0\\
\tilde\Theta_{AA} &=& \Theta_{AA} -\Theta_{AC}\Theta_{CC}^{-1}\Theta_{CA}\\
\tilde\Theta_{AB} &=& \Theta_{AB} -\Theta_{AC}\Theta_{CC}^{-1}\Theta_{CB}\\
\tilde\Theta_{BC} &=& \Theta_{BC}  -\Theta_{BA}\Theta_{AA}^{-1}\Theta_{AC}\\
\tilde\Theta_{CC} &=& \Theta_{CC} -\Theta_{CA}\Theta_{AA}^{-1}\Theta_{AC}\\
\label{eq:BBportion}
\tilde\Theta_{BB} &=& \Theta_{BB} - \Theta_{BC}\Theta_{CC}^{-1}\Theta_{CB} + \tilde\Theta_{BC}\tilde\Theta_{CC}^{-1}\tilde\Theta_{CB}\\
\nonumber&=& \Theta_{BB} - \Theta_{BA}\Theta_{AA}^{-1}\Theta_{AB} + \tilde\Theta_{BA}\tilde\Theta_{AA}^{-1}\tilde\Theta_{AB}
\end{eqnarray}
\end{lemma}

\more

\begin{lemma}[\scbf{\madgq{} $\ell_\infty$-distortion bounds ($K=2$)}]\label{lemma:linfinityABCbounds}
Let $V_1=A\cup B$ and $V_2=B\cup C$, where 
$A,B,C$ is a partition of $V=\{1,...,p\}$.  Let $\Theta\succ 0$ be a $p\times p$ positive definite matrix, and let $\Sigma=\Theta^{-1}$. Let $\tilde\Theta$ be the \madgq{} matrix (Equation~(\ref{eq:gq0})) based on $\Sigma_O$, where $O=(V_1\times V_1)\cup(V_2\times V_2)$. Moreover, define $d_{AB}:={\rm rd}(\Theta_{AB})$ and $
\gamma_{AB}:=\Vert\Theta_{AB}\Vert_\infty$, where ${\rm rd}(M)=\max_i|\{(i,j):M_{ij}\neq 0\}|$ denotes the max row-degree of the matrix $M$, and let $\theta=\lambda_{\min}(\Theta)>0$ be the smallest eigenvalue of $\Theta$. Define $d_{BA},d_{AC},d_{CA}, \gamma_{AB},\gamma_{BC}$, $\gamma_{AC}$ analogously. Then, we have
\begin{eqnarray}
\Vert \Theta_{AC}-\tilde\Theta_{AC} \Vert_\infty & = &\gamma_{AC}\\
\Vert \Theta_{AB}-\tilde\Theta_{AB} \Vert_\infty & \le & d_{AC}d_{BC}\gamma_{AC}\gamma_{BC}\theta^{-1} \\
\Vert \Theta_{BC}-\tilde\Theta_{BC} \Vert_\infty & \le & d_{BA}d_{CA}\gamma_{AB}\gamma_{AC}\theta^{-1} \\
\Vert \Theta_{AA}-\tilde\Theta_{AA} \Vert_\infty & \le & d_{AC}^2\gamma_{AC}^2\theta^{-1} \\
\Vert \Theta_{CC}-\tilde\Theta_{CC} \Vert_\infty & \le & d_{CA}^2\gamma_{AC}^2\theta^{-1}\\
\Vert \Theta_{BB}-\tilde\Theta_{BB} \Vert_\infty & \le & \gamma_{AC}^2|A|^2d_{AC}^2\theta^{-3}(d_{BA}^2\gamma_{AB}^2+d_{BC}^2\gamma_{BC}^2)\\
\nonumber&&+\gamma_{AC}2|C|d_{BA}d_{BC}d_{CA}\gamma_{AB}\gamma_{BC}\theta^{-2}\\
\Vert \Theta_{BB}-\tilde\Theta_{BB} \Vert_\infty & \le & \gamma_{AC}^2|C|^2d_{CA}^2\theta^{-3}(d_{BA}^2\gamma_{AB}^2+d_{BC}^2\gamma_{BC}^2)\\
\nonumber&&+\gamma_{AC}2|A|d_{BA}d_{BC}d_{AC}\gamma_{AB}\gamma_{BC}\theta^{-2}
\end{eqnarray}
\end{lemma}

\subsection{Proofs}\label{app:abcproofs}
This section contains the proofs of all theorems, lemmas, and corollaries of Appendices~\ref{sec:OABC}--\ref{app:abcaux}.

\begin{proof}[Proof of Theorem~\ref{theo:popOABCapp} (Exact Graph Recovery in $O$ ($K=2$))] We show that each condition (A1), (A2), and (A3) guarantees that $\delta< \nu/2$, so by Lemma~\ref{lemma:exactO} $\tilde E_O^{\tau} = E_O$ for any $\tau\in [\delta,\nu-\delta)$, and ${\rm sign}(\tilde\Theta_{ij})={\rm sign}(\Theta_{ij})$,  $\forall(i,j)\in E_O$.

\paragraph*{Condition (A1)} Condition (A1) implies $\delta=0$ by Theorem~\ref{thm:identifiability}, hence, we have $\delta<\nu/2$.

\paragraph*{Condition (A2)} Condition (A2) is a special case of Condition (A3), where $B$ is disconnected from $A$ and from $C$, which corresponds to $\gamma_B,d_B=0$, so that the coefficients specified in Condition (A3) reduce to $a=d_{O^c}^2\lambda_{\min}^{-1}$ and $b=0$, yielding the constraint $0<\gamma<\frac{-b+\sqrt{b^2+2a\nu}}{2a}=\sqrt{\frac{\nu\lambda_{\min}}{2d_{O^c}^2}}$.

\paragraph*{Condition (A3)} Let $d_{O^c}:=\max\{d_{AC},d_{CA}\}$, $d_B:=\max\{d_{BA},d_{BC}\}$,  $\gamma_B:=\max\{\gamma_{AB},\gamma_{BC}\}$, and $q:=\max\{|A|,|C|\}$. By Lemma~\ref{lemma:linfinityABCbounds}, we obtain
\[
\max\{\Vert\Theta_{AA}-\tilde\Theta_{AA}\Vert_\infty,\Vert\Theta_{CC}-\tilde\Theta_{CC}\Vert_\infty\} ~\le~\gamma^2 d_{O^c}^2\lambda_{\min}^{-1}
\]
\[
\max\{\Vert\Theta_{AB}-\tilde\Theta_{AB}\Vert_\infty, \Vert\Theta_{BC}-\tilde\Theta_{BC}\Vert_\infty \}~\le~\gamma d_{O^c}d_B\gamma_B\lambda_{\min}^{-1}
\]
and
\[
\Vert\Theta_{BB}-\tilde\Theta_{BB}\Vert_\infty \le \gamma^2 q^2d_{O^c}^2\lambda_{\min}^{-3}2d_B^2\gamma_B^2+\gamma 2qd_B^2d_{O^c}\gamma_B^2\lambda_{\min}^{-2}
\]
We now combine the three inequalities above into the following bound on $\delta$:
\begin{eqnarray*}
\delta &\le&\Vert\Theta_O-\tilde\Theta_O\Vert_\infty \\
&\le& \max\left\{\gamma^2 d_{O^c}^2\lambda_{\min}^{-1},~\gamma d_{O^c}d_B\gamma_B\lambda_{\min}^{-1}, ~\gamma^2 q^2d_{O^c}^2\lambda_{\min}^{-3}2d_B^2\gamma_B^2+\gamma 2qd_B^2d_{O^c}\gamma_B^2\lambda_{\min}^{-2}\right\}\\
&\le& \gamma^2 d_{O^c}^2\lambda_{\min}^{-1}+\gamma d_{O^c}d_B\gamma_B\lambda_{\min}^{-1}+\gamma^2 q^2d_{O^c}^2\lambda_{\min}^{-3}2d_B^2\gamma_B^2+\gamma 2qd_B^2d_{O^c}\gamma_B^2\lambda_{\min}^{-2}\\
&=& a\gamma^2+b\gamma
\end{eqnarray*}
where
\begin{eqnarray*}
a &=& d_{O^c}^2\left(\lambda_{\min}^{-1}+2q^2d_B^2\gamma_B^2\lambda_{\min}^{-3}\right)\\
b &=& d_{O^c}\left(d_B\gamma_B\lambda_{\min}^{-1}+2qd_B^2\gamma_B^2\lambda_{\min}^{-2}\right)
\end{eqnarray*}
For $\gamma>0$, the quadratic inequality 
\[a\gamma^2+b\gamma<\nu/2\]
has solution
$0<\gamma<\frac{-b+\sqrt{b^2+2a\nu}}{2a}$. Therefore, if $0<\gamma<\frac{-b+\sqrt{b^2+2a\nu}}{2a}$, then $\delta<\nu/2$.
\end{proof}

\more

\begin{proof}[Proof of Theorem~\ref{theo:distpropABC} (Distortion Propagation ($K=2$))] ~\less
\paragraph*{(i)} We have $V_1=A\cup B$, so if $i\in A$ then $ i\in V_1$, $\delta_{ii}=\delta_{ii}^{(1)}$, and $H_i=V_1^c=C$. Therefore, by Theorem~\ref{theo:distprK} part (i), $\delta_{ii}>0$ if and only if $\Theta_{iC}=\Theta_{iH_i}\neq 0$. Similarly, $V_2=B\cup C$, so if $j\in C$ then $ j\in V_2$, $\delta_{jj}=\delta_{jj}^{(2)}$, and $H_j=V_2^c=A$. Therefore, by Theorem~\ref{theo:distprK} part (i), $\delta_{jj}>0$ if and only if $\Theta_{jA}=\Theta_{jH_j}\neq 0$.
\item If $(i,j)\in A\times (A\cup B)$, $i\neq j$, then $i,j\in V_1$, $\delta_{ij}=\delta_{ij}^{(1)}$, and $H_i=V_1^c=C$. Therefore, by Theorem~\ref{theo:distprK} part (ii), almost everywhere, $\delta_{ij}\neq 0$ if and only if $\exists h\in N_C(i) $ and $\exists k\in N_C(j)$ such that $h=k$ or $h \text{ is } C$-connected to  $k$. 
Similarly, if $(i,j)\in C\times (B\cup C)$, $i\neq j$, then $i,j\in V_2$, $\delta_{ij}=\delta_{ij}^{(2)}$, and $H_i=V_2^c=A$. Therefore, by Theorem~\ref{theo:distprK} part (ii), almost everywhere,  $\delta_{ij}\neq 0$ if and only if $\exists h\in N_A(i) $ and $\exists k\in N_A(j)$ such that $h=k$ or $h \text{ is } A$-connected to  $k$.
\paragraph*{(ii)}  If $(i,j)\in A\times (A\cup B)$, $i<j$, then $i,j\in V_1$, $\delta_{ij}=\delta_{ij}^{(1)}$, $\delta_{ii}=\delta_{ii}^{(1)}$, and $\delta_{jj}=\delta_{jj}^{(1)}$, so by Theorem~\ref{theo:distprK} part (iii), if $\delta_{ij}\neq 0$ then $\delta_{ii}>0$ and $\delta_{jj}>0$. Similarly, if $(i,j)\in C\times (B\cup C)$, $j<i$, then $i,j\in V_2$, $\delta_{ij}=\delta_{ij}^{(2)}$, $\delta_{ii}=\delta_{ii}^{(2)}$, and $\delta_{jj}=\delta_{jj}^{(2)}$, so by Theorem~\ref{theo:distprK} part (iii), if $\delta_{ij}\neq 0$ then $\delta_{ii}>0$ and $\delta_{jj}>0$.
\end{proof}

\newpage

\begin{proof}[Proof of Corollary~\ref{coro:distpropagABC}]
~\less
\paragraph*{(i)} If $\Theta_{ij}\neq 0$ with $(i,j)\in A\times C$, then $\Theta_{iC}\neq 0$ and $\Theta_{jA}\neq 0$, so by Theorem~\ref{theo:distpropABC} part (i) we have $\delta_{ii}\neq 0$ and $\delta_{jj}\neq 0$.
\paragraph*{(ii)}  If $\Theta_{ij}\neq 0$ with $(i,j)\in A\times C$, and $ k\in A\cup B\setminus\{i\}  $ is $(C\cup \{k\})$-connected to $j$, then, by Theorem~\ref{theo:distpropABC} part (ii) we have $\delta_{ik}\neq 0$ [a.e.]. Similarly, if $\Theta_{ij}\neq 0$ with $(i,j)\in A\times C$, and if $h\in B\cup C\setminus\{j\}$ is $(A\cup \{h\})$-connected to $i$, then, by Theorem~\ref{theo:distpropABC} part (ii) we have $\delta_{hj}\neq 0$ [a.e.].
\end{proof}

\more

\begin{proof}[Proof of Theorem~\ref{theo:abcOc} (Oracle Minimal Superset of $E_{O^c}$ ($K=2$))]  ~\less
\paragraph*{(i)} First, we prove that $\mathcal{S}_{\rm diag}$ in Equation~(\ref{eq:supersetAC}) enjoys property (i) of a minimal superset (Definition~\ref{def:minimality}): $\forall\Sigma'\in\mathcal{A}(\Sigma,O,Q)$ we have $E_{O^c}'\subseteq\mathcal{S}_{\rm diag}$. We prove this by contradiction. Suppose there exists $\Sigma'\in\mathcal{A}(\Sigma,O,Q)$ such that $E_{AC}'\cap \mathcal{S}_{\rm diag}^c\neq \emptyset$, i.e. such matrix $\Sigma'$ induces one or more edges in $O^c$ and outside of $\mathcal{S}_{\rm diag}$. So, if  $\exists (h,l)\in E'_{AC}\cap\mathcal{S}_{\rm diag}^c$, we must have $h\in A\cap D_{\rm diag}^c$ and/or $l\in C\cap D_{\rm diag}^c$. Without loss of generality, suppose it is the case where $h\in A\cap D_{\rm diag}^c$. However, if $(h,l)\in E_{AC}'$, then Theorem~\ref{theo:distpropABC} part~(i) guarantees that there must be a diagonal distortion on the entry $(h,h)$ and so $h\in D_{\rm diag}$, which is a contradiction. 

We now prove that $\mathcal{S}_{\rm diag}$ enjoys property (ii) of a minimal superset (Definition~\ref{def:minimality}): $\forall\mathcal{S}' \subsetneq \mathcal{S}_{\rm diag}$, $\exists\Sigma'\in\mathcal{A}(\Sigma,O,Q)$ such that $E'_{AC} \cap (\mathcal{S}_{\rm diag}\setminus \mathcal{S}') \neq \emptyset$. Consider the following optimization problem
\begin{equation*}
    T(S,\tau) ~=~ \underset{T\succ 0, ~T_{O^c}=Z_\tau}{\arg\max} ~\log\det  T - \sum_{(i,j)\in O}T_{ij}\Sigma_{ij} \end{equation*}
where $Z_\tau$ is an entry set that is zero everywhere except over the symmetric set $S\subseteq O^c$ where all entries have value $\tau\in\mathbb{R}$. We have that $T(S,0)$ equals the \madgq{} matrix $\tilde\Theta$, which is guaranteed to exist (Lemma~\ref{lemma:maxdetopt}). For any $S\subseteq O^c$ and a sufficiently small $|\tau|\neq 0$, also the solution $T(S,\tau)$ exists and is uniquely identified by the constraints $\det T>0$ and $T_{O^c} = Z_\tau$, and by the first order condition  $[T^{-1}]_{O} = \Sigma_O $. Thus, the precision matrix $\Theta'=T(\mathcal{S}_{\rm diag},\tau)$ satisfies  $[\Theta^{'-1}]_O = \Sigma_O$ and $E'_{O^c}=\mathcal{S}_{\rm diag}$, and thereby  $\mathcal{D_{\rm diag}}(\Sigma',O)=\mathcal{D}_{\rm diag}(\Sigma,O)$ by Theorem~\ref{theo:distpropABC}, since in this case $K=2$ we have a distortion $\tilde\Theta'_{ii}\neq\Theta'_{ii}$ if and only if $\Theta'_{iC}\neq 0$ or $\Theta'_{Ai}\neq 0$. Hence, $\Sigma'\in\mathcal{A}(\Sigma,O,Q)$. This shows that the case where $E_{AC}=\mathcal{S}_{\rm diag}$ is possible, that is, there is no set $\mathcal{S}'\subsetneq \mathcal{S}_{\rm diag}$ that could contain all plausible edge sets in $O^c$.  

\paragraph*{(ii)} If Assumption~(A4$^\star$) holds, then Corollary~\ref{coro:distpropagABC} guarantees that for every  $(i,j)\in A\times C$, almost everywhere,
\[
\Theta_{ij}\neq 0 ~~~\Longrightarrow~~~  \delta_{ik}\neq 0, \delta_{hj}\neq 0, ~\text{for some}~k\in (A\cup B)\setminus\{i\}, ~h\in (B\cup C)\setminus\{k\} 
\]
Thus, $D_{\rm off}=D_{\rm diag} $ almost everywhere, and thereby $\mathcal{S}_{\rm off}=\mathcal{S}_{\rm diag}$ (Equations~(\ref{eq:supersetAC}) and (\ref{eq:minsupoff})), almost everywhere. If Assumption~(A4$^\star$) does not hold, then it is not guaranteed to have an off-diagonal distortion on the same row and column of every edge in $O^c$. That is, $\mathcal{S}_{\rm off}$ may contain at least one row or column less than $\mathcal{S}_{\rm diag}$, indeed missing at least one true edge of $E_{O^c}$.
\end{proof}

\vspace{0.5mm}

\begin{proof}[Proof of Lemma~\ref{lemma:cardplausibABC} (Cardinality of plausible edge sets in $O^c$ ($K=2$))] ~
\paragraph*{(i)} Note that $\mathcal{S}_{\rm diag}\cap (A\times C)=(A\cap D_{\rm diag})\times (C\cap D_{\rm diag})$  is the minimal superset of $E_{AC}$, so it must contain at least one true edge in every row and in every column. Thus, the number of edges in $O^c$ cannot be smaller than $M$, or larger than $|\mathcal{S}_{\rm diag}\cap (A\times C)|=mM$. 
\paragraph*{(ii)} Equation~(\ref{eq:xik}) is due to the Inclusion-Exclusion principle applied to the problem of arranging $\kappa$ objects over an $m\times M$ grid, with the constraint of at least one object on every row and at least one object on every column. Finally, we have $\xi_\kappa\le\binom{mM}{\kappa}\le\binom{|A||C|}{\kappa} =\varphi_\kappa$ because $mM\le |A| |C|$.
\end{proof}

\vspace{0.5mm}

\begin{proof}[Proof of Theorem~\ref{theo:fullrecoveryK2} (GQ Graph recovery (population case, $K=2$)]
Each of the assumptions (A1), (A2), (A3) alone guarantees that $\delta<\nu/2$ so, for any $\tau\in[\delta,\nu-\delta)$, the thresholded edge set $\mathcal{E}_O^\tau\equiv \tilde E_O^\tau$ in Equation~(\ref{eq:fulledegqpopABC}) equals the true edge set $E_O$, as per Theorem~\ref{theo:popOABCapp}. Moreover, notice that, by definition of $\nu$ (Equation~(\ref{eq:nu})), no off-diagonal entry of $\Theta_O$ may have magnitude in the interval $(0,\nu)$, so if $0<|\tilde\Theta_{ij}|<\nu$, then we must have $\delta_{ij}:=\Theta_{ij}-\tilde\Theta_{ij}\neq 0$. Thus, under Assumption~(A5$^\star$), for any $\tau\in[\delta,\nu]$, the set $W_\tau$ defined in Algorithm~\ref{algo:fullrecoveryK2},
\begin{equation*}
W_\tau=\left\{i\in A\cup C:~\exists j\neq i, 0<|\tilde\Theta_{ij}|< \tau\right\}
\end{equation*}
equals the set $D_{\rm off}$ in Equation~(\ref{eq:minsupoff}). Therefore, for any $\tau\in[\delta,\nu-\delta)$, $\mathcal{E}^\tau_{O^c}=\mathcal{S}_{\rm off}$ where, under Assumption~(A4$^\star$), $\mathcal{S}_{\rm off}$ is the minimal superset of $E_{O^c}$ [a.e.] based on oracle off-diagonal distortions, as per  Theorem~\ref{theo:abcOc} part (ii).
\end{proof}

\vspace{0.5mm}

\begin{proof}[Proof of Lemma~\ref{lemma:mdbias} (\madgq{} as a function of $\Theta$ ($K=2$))]
Let 
\begin{equation*}
\Sigma_1 = \left[\begin{array}{cc}
\Sigma_{AA} & \Sigma_{AB} \\
\Sigma_{BA} & \Sigma_{BB} 
\end{array} \right],~~~~
\Sigma_2 = \left[\begin{array}{cc}
\Sigma_{BB} & \Sigma_{BC} \\
\Sigma_{CB} & \Sigma_{CC} 
\end{array} \right]~
\end{equation*}
\begin{equation*}
\Theta_1 = \left[\begin{array}{cc}
\Theta_{AA} & \Theta_{AB} \\
\Theta_{BA} & \Theta_{BB} 
\end{array} \right], ~~~~
\Theta_2 = \left[\begin{array}{cc}
\Theta_{BB} & \Theta_{BC} \\
\Theta_{CB} & \Theta_{CC} 
\end{array} \right],
\end{equation*}
\begin{equation*}
\Theta_{1,C} = \left[\begin{array}{c}
\Theta_{AC}  \\
\Theta_{BC} 
\end{array} \right], ~~~~
\Theta_{2,A} = \left[\begin{array}{c}
\Theta_{BA}  \\
\Theta_{CA} 
\end{array} \right],
\end{equation*}
and $\Theta_{C,1}=\Theta_{1,C}^T$ and $\Theta_{A,2}=\Theta_{2,A}^T$. By computing Schur complements, we obtain
\begin{eqnarray}\label{sig1min1}
\Sigma_1^{-1} & = & \Theta_1-\Theta_{1,C}\Theta_{CC}^{-1}\Theta_{C,1}\nonumber\\
&=&
\left[ 
\begin{array}{cc}
\Theta_{AA} - \Theta_{AC}\Theta_{CC}^{-1}\Theta_{CA}, & \Theta_{AB} - \Theta_{AC}\Theta_{CC}^{-1}\Theta_{CB}\\
\Theta_{BA} - \Theta_{BC}\Theta_{CC}^{-1}\Theta_{CA}, & \Theta_{BB} - \Theta_{BC}\Theta_{CC}^{-1}\Theta_{CB}
\end{array}\right]
\end{eqnarray}
and
\begin{eqnarray}\label{sig2min1}
\Sigma_2^{-1} & = &  \Theta_2-\Theta_{2,A}\Theta_{AA}^{-1}\Theta_{A,2}\nonumber\\
&=&
\left[ 
\begin{array}{cc}
\Theta_{BB} - \Theta_{BA}\Theta_{AA}^{-1}\Theta_{AB}, & \Theta_{BC} - \Theta_{BA}\Theta_{AA}^{-1}\Theta_{AC}\\
\Theta_{CB} - \Theta_{CA}\Theta_{AA}^{-1}\Theta_{AB}, & \Theta_{CC} - \Theta_{CA}\Theta_{AA}^{-1}\Theta_{AC}
\end{array}\right].
\end{eqnarray}
Analogous formulas hold for $\tilde\Theta$, but since $\tilde\Theta_{AC}=0$, these formulas reduce to 
\begin{equation}\label{sigminzero}
\Sigma_1^{-1}  =  \left[ 
\begin{array}{cc}
\tilde\Theta_{AA}, & \tilde\Theta_{AB}\\
\tilde\Theta_{BA}, & \tilde\Theta_{BB} - \tilde\Theta_{BC}\tilde\Theta_{CC}^{-1}\tilde\Theta_{CB}
\end{array}\right]
\end{equation}
\begin{equation}\label{sigminzero2}
\Sigma_2^{-1}  =  \left[ 
\begin{array}{cc}
\tilde\Theta_{BB} - \tilde\Theta_{BA}\tilde\Theta_{AA}^{-1}\tilde\Theta_{AB}, & \tilde\Theta_{BC}\\
\tilde\Theta_{CB}, & \tilde\Theta_{CC}
\end{array}\right]
\end{equation}
By solving (\ref{sigminzero}) and (\ref{sigminzero2}) for each block of $\Theta$, we obtain
\begin{equation}\label{thetaMD3x3}
\tilde\Theta = 
\left[\begin{array}{ccc}
\left[\Sigma_1^{-1}\right]_{AA} & \left[\Sigma_1^{-1}\right]_{AB} & 0\vspace{2mm} \\
\left[\Sigma_1^{-1}\right]_{BA} & \tilde\Theta_{BB}  & \left[\Sigma_2^{-1}\right]_{BC}\vspace{2mm}\\
0 & \left[\Sigma_2^{-1}\right]_{CB} & \left[\Sigma_2^{-1}\right]_{CC}
\end{array} \right]
\end{equation}
where
\begin{eqnarray}
 \tilde\Theta_{BB} &=& \left[\Sigma_1^{-1}\right]_{BB}+\left[\Sigma_2^{-1}\right]_{BC}\left[\Sigma_2^{-1}\right]_{CC}^{-1}\left[\Sigma_2^{-1}\right]_{CB}\\
&=& \left[\Sigma_2^{-1}\right]_{BB}+\left[\Sigma_1^{-1}\right]_{BA}\left[\Sigma_1^{-1}\right]_{AA}^{-1}\left[\Sigma_1^{-1}\right]_{AB}
\end{eqnarray}
Rewriting (\ref{thetaMD3x3}) in terms of $\Theta$ components using expressions in (\ref{sig1min1}) and (\ref{sig2min1}) completes the proof.
\end{proof}

\more\more

\begin{proof}[Proof of Lemma~\ref{lemma:linfinityABCbounds} (\madgq{} $\ell_\infty$-distortion bounds ($K=2$))] The first equality is true because $\tilde\Theta_{AC}=0$. By using the closed form expression of $\tilde\Theta$ in Lemma~\ref{lemma:mdbias} and the inequalities in Lemma~\ref{lemma:maxnormineq}, we obtain
\begin{eqnarray*}
\Vert\Theta_{AB}-\tilde\Theta_{AB}\Vert_\infty &=& \Vert\Theta_{AC}\Theta_{CC}^{-1}\Theta_{CB}\Vert_\infty\\
&\le & \min\{{\rm rd}(\Theta_{AC}),{\rm rd}(\Theta_{CC}^{-1})\}{\rm rd}(\Theta_{BC})\Vert\Theta_{AC}\Vert_\infty\Vert\Theta_{BC}\Vert_\infty\Vert\Theta_{CC}^{-1}\Vert_\infty\\
&\le & d_{AC}d_{BC}\gamma_{AC}\gamma_{BC}\theta^{-1}
\end{eqnarray*}
where, by Lemma~\ref{lemma:maxnormineq},  $\Vert\Theta_{AA}^{-1}\Vert_\infty\le \theta^{-1}$ and $\Vert\Theta_{CC}^{-1}\Vert_\infty\le \theta^{-1}$. 
Similarly, we obtain
\begin{eqnarray*}
\Vert\Theta_{BC}-\tilde\Theta_{BC}\Vert_\infty &=&\Vert\Theta_{BA}\Theta_{AA}^{-1}\Theta_{AC}\Vert_\infty~\le~ d_{BA}d_{CA}\gamma_{AB}\gamma_{AC}\theta^{-1}\\
\Vert\Theta_{AA}-\tilde\Theta_{AA}\Vert_\infty &=&\Vert\Theta_{AC}\Theta_{CC}^{-1}\Theta_{CA}\Vert_\infty~\le~ d_{AC}^2\gamma_{AC}^2\theta^{-1}\\
\Vert\Theta_{CC}-\tilde\Theta_{CC}\Vert_\infty &=&\Vert\Theta_{CA}\Theta_{AA}^{-1}\Theta_{AC}\Vert_\infty~\le~ d_{CA}^2\gamma_{AC}^2\theta^{-1}
\end{eqnarray*}
To bound the component $\Vert\Theta_{BB}-\tilde\Theta_{BB}\Vert_\infty$, first let us manipulate the expression of $\Theta_{BB}-\tilde\Theta_{BB}$. Lemma~\ref{lemma:mdbias} provides the expression 
\[\tilde\Theta_{BB}=\Theta_{BB} - \Theta_{BA}\Theta_{AA}^{-1}\Theta_{AB} + \tilde\Theta_{BA}\tilde\Theta_{AA}^{-1}\tilde\Theta_{AB}
\]
so
\begin{eqnarray*}
  \Theta_{BB}-\tilde\Theta_{BB}  &=& \Theta_{BA}\Theta_{AA}^{-1}\Theta_{AB}-\tilde\Theta_{BA}\tilde\Theta_{AA}^{-1}\tilde\Theta_{AB}\\
  &=& \Theta_{BA}\Theta_{AA}^{-1}\Theta_{AB}-(\Theta_{BA} -\Theta_{BC}\Theta_{CC}^{-1}\Theta_{CA})\tilde\Theta_{AA}^{-1}(\Theta_{AB} -\Theta_{AC}\Theta_{CC}^{-1}\Theta_{CB})\\
  &=& \Theta_{BA}\Theta_{AA}^{-1}\Theta_{AB}-\Theta_{BA}\tilde\Theta_{AA}^{-1}\Theta_{AB}-\Theta_{BC}\Theta_{CC}^{-1}\Theta_{CA}\tilde\Theta_{AA}^{-1}\Theta_{AC}\Theta_{CC}^{-1}\Theta_{CB}\\
  & & +\Theta_{BA}\tilde\Theta_{AA}^{-1}\Theta_{AC}\Theta_{CC}^{-1}\Theta_{CB}+\Theta_{BC}
  \Theta_{CC}^{-1}\Theta_{CA}\tilde\Theta_{AA}^{-1}\Theta_{AB}\\
  &=& \Theta_{BA}(\Theta_{AA}^{-1}-\tilde\Theta_{AA}^{-1})\Theta_{AB}-\Theta_{BC}\Theta_{CC}^{-1}\Theta_{CA}\tilde\Theta_{AA}^{-1}\Theta_{AC}\Theta_{CC}^{-1}\Theta_{CB}\\
  & & +\Theta_{BA}\tilde\Theta_{AA}^{-1}\Theta_{AC}\Theta_{CC}^{-1}\Theta_{CB}+\Theta_{BC}
  \Theta_{CC}^{-1}\Theta_{CA}\tilde\Theta_{AA}^{-1}\Theta_{AB}
\end{eqnarray*}
Thus,
\begin{eqnarray*}
    \Vert\Theta_{BB}-\tilde\Theta_{BB}\Vert_\infty &\le & \Vert\Theta_{BA}(\Theta_{AA}^{-1}-\tilde\Theta_{AA}^{-1})\Theta_{AB}\Vert_\infty+\Vert\Theta_{BC}\Theta_{CC}^{-1}\Theta_{CA}\tilde\Theta_{AA}^{-1}\Theta_{AC}\Theta_{CC}^{-1}\Theta_{CB}\Vert_\infty\\
  & & +2\Vert\Theta_{BA}\tilde\Theta_{AA}^{-1}\Theta_{AC}\Theta_{CC}^{-1}\Theta_{CB}\Vert_\infty
\end{eqnarray*}
We now find upper-bounds for each of the three addends above by applying Lemma~\ref{lemma:maxnormineq} multiple times. First,
\[
\Vert\Theta_{BA}(\Theta_{AA}^{-1}-\tilde\Theta_{AA}^{-1})\Theta_{AB}\Vert_\infty ~=~ \Vert\Theta_{BA}\Theta_{AA}^{-1}(\tilde\Theta_{AA}-\Theta_{AA})\tilde\Theta_{AA}^{-1}\Theta_{AB}\Vert_\infty
\]
\begin{eqnarray*}
&\le & {\rm rd}(\Theta_{BA}\Theta_{AA}^{-1}){\rm rd}(\Theta_{BA}\tilde\Theta_{AA}^{-1})\Vert\Theta_{BA}\Theta_{AA}^{-1}\Vert_\infty \Vert \tilde\Theta_{AA}-\Theta_{AA}\Vert_\infty \Vert \tilde\Theta_{AA}^{-1}\Theta_{AB}\Vert_\infty \\
&\le & |A|^2 d_{BA}\gamma_{AB}\theta^{-1}d_{AC}^2\gamma_{AC}^2\theta^{-1}d_{BA}\gamma_{AB}\theta^{-1}\\
&=& |A|^2d_{BA}^2\gamma_{AB}^2d_{AC}^2\gamma_{AC}^2\theta^{-3}
\end{eqnarray*}
Next,
\[
\Vert\Theta_{BC}\Theta_{CC}^{-1}\Theta_{CA}\tilde\Theta_{AA}^{-1}\Theta_{AC}\Theta_{CC}^{-1}\Theta_{CB}\Vert_\infty
\]
\begin{eqnarray*}
&\le & {\rm rd}(\Theta_{BC}\Theta_{CC}^{-1}\Theta_{CA})^2\Vert\Theta_{BC}\Theta_{CC}^{-1}\Theta_{CA}\Vert_\infty^2 \Vert\tilde\Theta_{AA}^{-1}\Vert_\infty\\
&\le & |A|^2 d_{AC}^2d_{BC}^2\gamma_{AC}^2\gamma_{BC}^2\theta^{-3}
\end{eqnarray*}
where we used the bound derived for $\Vert\Theta_{AB}-\tilde\Theta_{AB}\Vert_\infty$, and by Lemma~\ref{lemma:mdbias} and Lemma~\ref{lemma:maxnormineq},
\begin{eqnarray*}
\Vert\tilde\Theta_{AA}^{-1}\Vert_\infty &=& \left\Vert \left[\left(\Sigma_{AB,AB}^{-1}\right)_{AA}\right]^{-1}\right\Vert_\infty\\
&\le& \lambda_{\max}\left(\left[\left(\Sigma_{AB,AB}^{-1}\right)_{AA}\right]^{-1}\right)\\
&=& \left[\lambda_{\min}\left(\left(\Sigma_{AB,AB}^{-1}\right)_{AA}\right)\right]^{-1}\\
&\le& \left[\lambda_{\min}\left(\Sigma_{AB,AB}^{-1}\right)\right]^{-1}\\
&=& \lambda_{\max}\left(\Sigma_{AB,AB}\right)\\
&\le & \lambda_{\max}(\Sigma)\\
&=& \left(\lambda_{\min}(\Theta)\right)^{-1}\\
&=& \theta^{-1}
\end{eqnarray*}
where the last two inequalities follow from Cauchy's Interlace Theorem. 

For the third addend, we have
\begin{eqnarray*}
\Vert\Theta_{BA}\tilde\Theta_{AA}^{-1}\Theta_{AC}\Theta_{CC}^{-1}\Theta_{CB}\Vert_\infty&\le& \min\{{\rm rd}(\Theta_{BA}\tilde\Theta_{AA}^{-1},{\rm rd}(\Theta_{CA})\}{\rm rd}(\Theta_{BC}\Theta_{CC}^{-1})\\
&&\times\Vert\Theta_{BA}\tilde\Theta_{AA}^{-1}\Vert_\infty\gamma_{AC}\Vert\Theta_{BC}\tilde\Theta_{CC}^{-1}\Vert_\infty\\
&\le & |C|d_{BA}d_{BC}d_{CA}\gamma_{AB}\gamma_{AC}\gamma_{BC}\theta^{-2}
\end{eqnarray*}
Therefore
\begin{eqnarray}\label{eq:BBbound1}
    \Vert\Theta_{BB}-\tilde\Theta_{BB}\Vert_\infty&\le & \gamma_{AC}^2|A|^2d_{AC}^2\theta^{-3}(d_{BA}^2\gamma_{AB}^2+d_{BC}^2\gamma_{BC}^2)\\
  \nonumber  && +\gamma_{AC} 2|C|d_{BA}d_{BC}d_{CA}\gamma_{AB}\gamma_{BC}\theta^{-2}
\end{eqnarray}
On the other hand, Lemma~\ref{lemma:mdbias} provides the alternative expression 
\[\tilde\Theta_{BB}=\Theta_{BB} - \Theta_{BC}\Theta_{CC}^{-1}\Theta_{CB} + \tilde\Theta_{BC}\tilde\Theta_{CC}^{-1}\tilde\Theta_{CB}
\]
and doing steps similar to the ones that we used to obtain Equation~(\ref{eq:BBbound1}), we get
\begin{eqnarray}\label{eq:BBbound2}
    \Vert\Theta_{BB}-\tilde\Theta_{BB}\Vert_\infty&\le & \gamma_{AC}^2|C|^2d_{CA}^2\theta^{-3}(d_{BA}^2\gamma_{AB}^2+d_{BC}^2\gamma_{BC}^2)\\
  \nonumber  && +\gamma_{AC} 2|A|d_{BA}d_{BC}d_{AC}\gamma_{AB}\gamma_{BC}\theta^{-2}
\end{eqnarray}
\end{proof}

\end{document}